\documentclass[11pt]{article}

\usepackage{amssymb, amsmath, amsthm, latexsym}
\usepackage{fullpage}
\usepackage{url}
\usepackage{tabularx}
\usepackage{paralist}
\usepackage{mathtools}

\usepackage[utf8]{inputenc} 
\usepackage[T1]{fontenc}    
\usepackage{hyperref}       
\usepackage{url}            
\usepackage{booktabs}       
\usepackage{amsfonts}       
\usepackage{nicefrac}       
\usepackage{microtype}      

\usepackage{amssymb, amsmath, amsthm, latexsym}
\usepackage{color}
\usepackage{url}
\usepackage{algorithm}
\usepackage{algpseudocode}
\usepackage{tabularx}
\usepackage{paralist}
\usepackage{mathtools}
\usepackage{tcolorbox}
\usepackage{xcolor}

\usepackage{bbm} 

\usepackage{makecell}
\usepackage{multirow}

\usepackage{nicefrac}       

\definecolor{niceblue}{rgb}{0.0,0.19,0.56}
\definecolor{BrickRed}{rgb}{0.80,0.25,0.33}
\definecolor{PineGreen}{rgb}{0.00,0.47,0.44}

\usepackage{hyperref}
\hypersetup{colorlinks,linkcolor={blue},citecolor={niceblue},urlcolor={blue}}

\usepackage{enumitem}

\usepackage{natbib}
\usepackage{sidecap}

\usepackage{pifont}
\usepackage{subfigure}

\newcommand{\R}{\mathbb{R}}
\newcommand{\eqdef}{\stackrel{\text{def}}{=}}

\def\<#1,#2>{\left\langle #1,#2\right\rangle}

\theoremstyle{plain}  
\newtheorem{lemma}{Lemma}[section]
\newtheorem{theorem}{Theorem}[section]
\newtheorem{definition}{Definition}[section]

\newtheorem{corollary}{Corollary}[section]

\newcommand{\circledOne}{\text{\ding{172}}}
\newcommand{\circledTwo}{\text{\ding{173}}}
\newcommand{\circledThree}{\text{\ding{174}}}
\newcommand{\circledFour}{\text{\ding{175}}}
\newcommand{\circledFive}{\text{\ding{176}}}

\usepackage[colorinlistoftodos,bordercolor=orange,backgroundcolor=orange!20,linecolor=orange,textsize=scriptsize]{todonotes}

\newcommand{\cM}{{\cal M}}
\newcommand{\cN}{{\cal N}}
\newcommand{\cO}{{\cal O}}

\newcommand{\mA}{{\bf A}}

\newcommand{\EE}{\mathbf{E}}

\def\R{\mathbb{R}}

\def\R{\mathbb R}

\def\EE{\mathbb E}
\def\PP{\mathbb P}

\def\la{\langle}
\def\ra{\rangle}

\def\tnabla{\widetilde{\nabla}}

\def\Bxi{\boldsymbol{\xi}}

\def\Bxi{\boldsymbol{\xi}}

\def\clip{\text{clip}}

\usepackage{hyperref}

\usepackage{xspace}

\newcommand{\algname}[1]{{\sf  #1}\xspace}
\newcommand{\algnames}[1]{{\sf \small #1}\xspace}

\usepackage{pifont}
\newcommand{\cmark}{{\color{PineGreen}\ding{51}}}%
\newcommand{\xmark}{{\color{BrickRed}\ding{55}}}%

\usepackage{colortbl}
\definecolor{bgcolor}{rgb}{0.8,1,1}
\definecolor{bgcolor2}{rgb}{0.8,1,0.8}

\usepackage{accents}
\newlength{\dhatheight}

\definecolor{mycolor}{RGB}{0,150,70}
\def\revision#1{{\color{black}#1}} 
    
\begin{document}

\title{High Probability Complexity Bounds for Non-Smooth Stochastic Optimization with Heavy-Tailed Noise}
\author{Eduard Gorbunov$^{1}$ \quad Marina Danilova$^{2, 3}$ \quad Innokentiy Shibaev$^{2,4}$ \\
Pavel Dvurechensky$^{5}$\quad Alexander Gasnikov$^{6, 2, 7}$}
\date{$^1$ Mohamed bin Zayed University of Artificial Intelligence, United Arab Emirates\\
$^2$ Moscow Institute of Physics and Technology, Russian Federation\\
$^3$ Institute of Control Sciences RAS, Russian Federation\\
$^4$ National Research University Higher School of Economics, Russian Federation \\
$^5$ Weierstrass Institute for Applied Analysis and Stochastics, Germany\\
$^6$ Innopolis University, Russian Federation\\
$^7$ Institute for Information Transmission Problems RAS, Russian Federation}
\maketitle

\begin{abstract}
Stochastic first-order methods are standard for training large-scale machine learning models. Random behavior may cause a particular run of an algorithm to result in a highly suboptimal objective value, whereas theoretical guarantees are usually proved for the expectation of the objective value. Thus, it is essential to theoretically guarantee that algorithms provide small objective residuals with high probability. Existing methods for non-smooth stochastic convex optimization have complexity bounds with the dependence on the confidence level that is either negative-power or logarithmic but under an additional assumption of sub-Gaussian (light-tailed) noise distribution that may not hold in practice. In our paper, we resolve this issue and derive the first high-probability convergence results with logarithmic dependence on the confidence level for non-smooth convex stochastic optimization problems with non-sub-Gaussian (heavy-tailed) noise. To derive our results, we propose novel stepsize rules for two stochastic methods with gradient clipping. Moreover, our analysis works for generalized smooth objectives with H\"older-continuous gradients, and for both methods, we provide an extension for strongly convex problems. Finally, our results imply that the first (accelerated) method we consider also has optimal iteration and oracle complexity in all the regimes, and the second one is optimal in the non-smooth setting.
\end{abstract}

\section{Introduction}\label{sec:intro}
Stochastic first-order optimization methods like \algname{SGD} \citep{robbins1951stochastic}, \algname{Adam} \citep{kingma2014adam}, and their various modifications are extremely popular in solving a number of different optimization problems, especially those appearing in statistics \citep{spokoiny2012parametric}, machine learning, and deep learning \citep{Goodfellow-et-al-2016}. The success of these methods in real-world applications motivates the researchers to investigate the theoretical properties of the methods and to develop new ones with better convergence guarantees. Typically, stochastic methods are analyzed in terms of the convergence in expectation (see \revision{\citep{ghadimi2013stochastic, gower2019sgd, moulines2011non}} and references therein), whereas high-probability complexity results are established \revision{more} rarely. However, as illustrated in \citep{gorbunov2020clipped_sstm}, guarantees in terms of the convergence in expectation have a much worse correlation with the real behavior of the methods than high-probability convergence guarantees when the noise in the stochastic gradients has \textit{heavy-tailed distribution}.

Recent studies \revision{\citep{csimcsekli2019heavy, simsekli2019tail,zhang2020why}} show that in several popular problems such as training {\tt BERT} \citep{devlin2018bert} on \revision{the} {\tt Wikipedia} dataset the noise in the stochastic gradients is heavy-tailed. Moreover, \citep{zhang2020why} justify empirically that in such cases, \algname{SGD} works significantly worse than \algname{clipped-SGD} \citep{pascanu2013difficulty} and \algname{Adam}. Therefore, it is important to theoretically study the methods' convergence when the noise is heavy-tailed.

For convex and strongly convex problems with Lipschitz continuous gradient, i.e., smooth convex and strongly convex problems, this question was properly addressed in \revision{\citep{davis2021low,gorbunov2020clipped_sstm, nazin2019algorithms}}, where the first high-probability complexity bounds with logarithmic dependence on the confidence level were derived for the stochastic problems with heavy-tailed noise. However, a number of practically important problems are non-smooth \textit{on the whole space} \revision{\citep{mai2021stability, zhang2020gradient}}. For example, in deep neural network training, the loss function often grows polynomially fast when the norm of the network's weights goes to infinity. {Moreover, non-smoothness of the activation functions such as ReLU or loss functions such as hinge loss implies the non-smoothness of the whole problem.} 
\revision{While being well-motivated by practical applications, the existing high-probability convergence guarantees for stochastic first-order methods applied to solve non-smooth convex optimization problems with heavy-tailed noise have a drawback. Namely, the existing complexity bounds depend on the negative power of the confidence level. This dramatically increases the number of iterations required to obtain high accuracy  {of the solution} with probability close to one.}
Such a discrepancy in the theory between algorithms for stochastic smooth and non-smooth problems leads us to the natural question: \textit{is it possible to obtain high-probability complexity bounds with logarithmic dependence on the confidence level for \textbf{non-smooth} convex stochastic problems with heavy-tailed noise?} In this paper, we give a positive answer to this question. Moreover, we derive the corresponding bounds under much weaker assumptions than the ones used in the previous works. {To achieve this we focus on gradient clipping methods\revision{, as in} \revision{\citep{gehring2017convolutional,mai2021stability,menon2020can,pascanu2013difficulty, zhang2020gradient,zhang2020why}}.}

\subsection{Preliminaries}\label{sec:prelim}
Before we describe our contributions in detail, we formally state the considered setup.

\paragraph{Notation and standard definitions.} We use standard notation for stochastic optimization literature. For all $x\in\R^n$ we use $\|x\|_2 = \sqrt{\la x,x\ra}$ to denote standard Euclidean norm, where $\la x,y \ra = x_1y_1 + x_2y_2 + \ldots + x_ny_n$, \revision{$x = (x_1,\ldots,x_n)^\top \in \R^n$}. Next, we use $\EE[\xi]$ and $\EE[\xi \mid \eta]$ to denote expectation of $\xi$ and expectation of $\xi$ conditioned on $\eta$ respectively. In some places of the paper, we also use $\EE_\xi[\cdot]$ to denote conditional expectation taken w.r.t.\ the randomness coming from $\xi$. The probability of event $E$ is defined as $\PP\{E\}$.
Finally, we use the following definition.
\begin{definition}\label{eq:str_cvx_def}
    \revision{A} differentiable function $f:Q\subseteq \R^n \to \R$ is called $\mu$-strongly convex for some $\mu \geq 0$ if for all $x,y \in Q$
    \begin{equation*}
        f(y) \ge f(x) + \langle\nabla f(x), y-x \rangle + \frac{\mu}{2}\|y-x\|_2^2.
    \end{equation*}
    When $\mu = 0$\revision{, the function} $f$ is called convex.
\end{definition}

\paragraph{Stochastic optimization.} We focus on the following problem
\begin{equation}
    \min\limits_{x\in \R^n}f(x),\quad f(x) = \EE_\xi\left[f(x,\xi)\right],\label{eq:main_problem}
\end{equation}
where \revision{$f$} is a convex but possibly non-smooth function. Next, we assume that at each point $x\in \R^n$ we have an access to the unbiased estimator $\nabla f(x,\xi)$ of $\nabla f(x)$ such that $\EE_\xi\left[\left\|\nabla f(x,\xi) - \nabla f(x)\right\|_2^2\right] < \infty$ and if additionally $x \in Q\subseteq \R^n$, then
\begin{equation}
    \EE_\xi[\nabla f(x,\xi)] = \nabla f(x),\quad
    \EE_\xi\left[\left\|\nabla f(x,\xi) - \nabla f(x)\right\|_2^2\right] \le \sigma^2,\quad \sigma > 0.\label{eq:bounded_variance_clipped_SSTM}
\end{equation}
This assumption with $Q = \R^n$ on the stochastic {oracle} is widely used in stochastic optimization literature \citep{ghadimi2012optimal, ghadimi2013stochastic, juditsky2011first, lan2012optimal, nemirovski2009robust}. In contrast, in our theoretical results, we assume that $Q$ is the ball centered at \revision{some}\footnote{\revision{Our proofs work for any $x^*$. In particular, one can choose $x^*$ being a projection of $x^0$ on the solutions set.}} solution $x^*$ of \eqref{eq:main_problem} with radius $\sim R_0 \ge \|x^0 - x^*\|_2$, where $x^0$ is a starting point of the method, i.e., \textit{our analysis does not require \eqref{eq:bounded_variance_clipped_SSTM} to hold on $\R^n$.} That is, our assumption on the noise is much more general than those used in previous works in the area. Finally, we emphasize that we do not assume that the stochastic gradients have so-called ``light tails'' \citep{lan2012optimal}, i.e., sub-Gaussian noise distribution meaning that $\PP\{\|\nabla f(x,\xi) - \nabla f(x)\|_2 > b\} \le 2\exp(-\nicefrac{b^2}{(2\sigma^2)})$ for all $b > 0$. 

\paragraph{Level of smoothness.} Finally, we assume that function $f$ has $(\nu, M_{\nu})$-H\"older continuous gradients\footnote{By default, we always write ``gradients'', though our analysis works for non-differentiable convex functions as well (when $\nu = 0$): \revision{at any point where the gradient is now calculated, it is sufficient to use any subgradient at this point.}
This remark is valid for Definition~\ref{eq:str_cvx_def} as well.} on a compact set $Q \subseteq \R^n$ for some $\nu \in [0,1]$, $M_\nu > 0$ meaning that
\begin{equation}
    \|\nabla f(x) - \nabla f(y)\|_2 \le M_\nu \|x- y\|_2^\nu\quad \forall x,y \in Q.\label{eq:holder_def}
\end{equation}
When $\nu = 1$ inequality \eqref{eq:holder_def} implies $M_1$-smoothness of $f$, and when $\nu = 0$ we have that $\nabla f(x)$
has bounded variation which is equivalent to being uniformly bounded.
Moreover, when $\nu = 0$ differentiability of $f$ is not needed: one can assume uniform boundedness of the subgradients of $f$ throughout the proofs. Linear regression in the case when the noise has generalized Gaussian distribution (Example 4.4 from \citep{Chaux_2007}) serves as a natural example of the situation with $\nu \in (0,1)$. Moreover, when \eqref{eq:holder_def} holds for $\nu = 0$ and $\nu = 1$ simultaneously then it holds for all $\nu\in[0,1]$ with $M_\nu \le M_0^{1-\nu}M_1^\nu$ \citep{nesterov2015universal}. As we show in our results, it is sufficient to assume that the set $Q$ is the ball centered at the solution $x^*$ of \eqref{eq:main_problem} with radius $\sim R_0 \ge \|x^0 - x^*\|_2$, where $x^0$ is a starting point of the method, i.e., \textit{our analysis does not require \eqref{eq:holder_def} to hold on $\R^n$.}

\revision{In addition to inequality \eqref{eq:holder_def}, we also assume that
\begin{equation}
        \|\nabla f(x)\|_2 \le \left(\frac{1+\nu}{\nu}\right)^{\frac{\nu}{1+\nu}}M_\nu^{\frac{1}{1+\nu}} \left(f(x) - f(x^*)\right)^{\frac{\nu}{1+\nu}},\quad \forall x \in Q, \label{eq:holder_gradient_bound}
\end{equation}
where for $\nu = 0$ we use $\left[\left(\frac{1+\nu}{\nu}\right)^{\frac{\nu}{1+\nu}}\right]_{\nu=0} := \lim_{\nu\to 0}\left(\frac{1+\nu}{\nu}\right)^{\frac{\nu}{1+\nu}} = 1$, and
\begin{equation}
        \|\nabla f(x)\|_2^2 \le
        2\left(\frac{1}{\delta}\right)^{\frac{1-\nu}{1+\nu}}M_{\nu}^{\frac{2}{1+\nu}}\left(f(x)-f(x^*)\right) + \delta^{\frac{2\nu}{1+\nu}} M_{\nu}^{\frac{2}{1+\nu}},\quad \forall x \in Q. \label{eq:holder_gradient_bound_2}
\end{equation}
As we prove in Lemmas~\ref{lem:gradient_bound} and \ref{lem:gradient_bound_2}, inequalities \eqref{eq:holder_gradient_bound} and \eqref{eq:holder_gradient_bound_2} follow from \eqref{eq:holder_def} when $Q = \R^n$. However, when $Q \neq \R^n$ minimal value of $M_\nu$ such that \eqref{eq:holder_def} holds on $Q$ can be smaller than the minimal value of $M_\nu$ such that \eqref{eq:holder_gradient_bound} and \eqref{eq:holder_gradient_bound_2} hold (see also the discussion in Appendix B from \cite{sadiev2023high}).
}

\paragraph{High-probability convergence.} For a given accuracy $\varepsilon > 0$ and confidence level $\beta \in (0,1)$ we are interested in finding $\varepsilon$-solutions of problem \eqref{eq:main_problem} with probability at least $1-\beta$, i.e., such $\widehat x$ that $\PP\{f(\widehat x) - f(x^*) \le \varepsilon\} \ge 1 - \beta$. For brevity, we will call such (in general, random) points $\widehat x$ as $(\varepsilon,\beta)$-solution of \eqref{eq:main_problem}. Moreover, by \revision{the} high-probability iteration/oracle complexity of a stochastic method $\cM$ we mean a sufficient number of iterations/oracle calls (number of $\nabla f(x,\xi)$ computations) needed to guarantee that $\cM$ returns an $(\varepsilon,\beta)$-solution of \eqref{eq:main_problem}.

\subsection{Contributions}
We summarize our main contributions below.

\begin{table}[t]
    \centering
    \scriptsize
    \caption{\small Summary of known and new high-probability complexity bounds for solving \eqref{eq:main_problem} with $f$ being \textbf{convex} and having $(\nu, M_\nu)$-H\"older continuous gradients. Columns: ``Complexity'' = high-probability complexity ($\varepsilon$ -- accuracy, $\beta$ -- confidence level, numerical constants, and logarithmic factors are omitted), ``HT'' = heavy-tailed noise, ``UD'' = unbounded domain, ``CS'' = bound on the variance of the stochastic gradient and H\"older continuity of the gradient is required only on the compact set. Notation: $R_0 = \|x^0 - x^*\|_2$, where $x^*$ is the closest solution to $x^0$; $\sigma^2$ = upper bound on the variance (see \eqref{eq:bounded_variance_clipped_SSTM}); $D$ = diameter of the set where optimization problem is defined. The results labeled by $^{\clubsuit}$ are obtained from the convergence guarantees in expectation via Markov's inequality. Negative-power dependencies on the confidence level $\beta$ are colored in red. Our results are highlighted in green.}
    \label{tab:cvx_case_comparison}
    {
    \begin{tabular}{|c|c|c|c|c|c|}
        \hline
        Method & {\centering Complexity} & $\nu$ & HT? & UD? & CS?\\
        \hline
        \hline
        \begin{tabular}{c}
            \algnames{SGD} \\ \citep{nemirovski2009robust}
        \end{tabular} & $\max\left\{\frac{M_0^2D^{2}}{\varepsilon^2},\frac{\sigma^2D^2}{\varepsilon^2}\right\}$ & $0$ & \xmark & \xmark & \xmark\\
        \hline
        \begin{tabular}{c}
            \algnames{AC-SA} \\
            \citep{ghadimi2012optimal,lan2012optimal}
        \end{tabular} & $\max\left\{\sqrt{\frac{M_1R_0^{2}}{\varepsilon}},\frac{\sigma^2R_0^2}{\varepsilon^2}\right\}$ & $1$ & \xmark & \xmark & \xmark\\
        \hline
        \begin{tabular}{c}
            \algnames{SIGMA} \\ \citep{dvurechensky2016stochastic}
        \end{tabular} & $\max\left\{\frac{M_\nu^{\frac{2}{1+3\nu}}R_0^{\frac{2(1+\nu)}{1+3\nu}}}{\varepsilon^{\frac{2}{1+3\nu}}},\frac{\sigma^2R_0^2}{\varepsilon^2}\right\}$ & $[0,1]$ & \xmark & \xmark & \xmark\\
        \hline
        \begin{tabular}{c}
            \algnames{SGD} \\  \citep{nemirovski2009robust}$^{\clubsuit}$
        \end{tabular} & $\max\left\{\frac{M_0^2R_0^{2}}{{\color{BrickRed}\beta^2}\varepsilon^2},\frac{\sigma^2R_0^2}{{\color{BrickRed}\beta^2}\varepsilon^2}\right\}$ & $0$ & \cmark & \xmark & \xmark\\
        \hline
        \begin{tabular}{c}
            \algnames{AC-SA} \\ \citep{ghadimi2012optimal,lan2012optimal}$^{\clubsuit}$
        \end{tabular} & $\max\left\{\sqrt{\frac{M_1R_0^{2}}{{\color{BrickRed}\beta}\varepsilon}},\frac{\sigma^2R_0^2}{{\color{BrickRed}\beta^2}\varepsilon^2}\right\}$ & $1$ & \cmark & \cmark & \xmark\\
        \hline
        \begin{tabular}{c}
            \algnames{SIGMA} \\ \citep{dvurechensky2016stochastic}$^{\clubsuit}$
        \end{tabular} & $\max\left\{\frac{M_\nu^{\frac{2}{1+3\nu}}R_0^{\frac{2(1+\nu)}{1+3\nu}}}{{\color{BrickRed}\beta^{\frac{2}{1+3\nu}}}\varepsilon^{\frac{2}{1+3\nu}}},\frac{\sigma^2R_0^2}{{\color{BrickRed}\beta^2}\varepsilon^2}\right\}$ & $[0,1]$ & \cmark & \cmark & \xmark\\
        \hline
        \begin{tabular}{c}
            \algnames{clipped-SSTM} \\ \citep{gorbunov2020clipped_sstm}
        \end{tabular} & $\max\left\{\sqrt{\frac{M_1R_0^{2}}{\varepsilon}},\frac{\sigma^2R_0^2}{\varepsilon^2}\right\}$ & $1$ & \cmark & \cmark & \xmark \\
        \hline
        \begin{tabular}{c}
            \algnames{clipped-SGD} \\ \citep{gorbunov2020clipped_sstm}
        \end{tabular} & $\max\left\{\frac{M_1R_0^2}{\varepsilon},\frac{\sigma^2R_0^2}{\varepsilon^2}\right\}$ & $1$ & \cmark & \cmark & \xmark \\
        \hline
        \rowcolor{bgcolor2}\begin{tabular}{c}
            \algnames{clipped-SSTM} \\ (Theorem~\ref{thm:main_result_clipped_SSTM_main})
        \end{tabular} & $\max\left\{\frac{M_\nu^{\frac{2}{1+3\nu}}R_0^{\frac{2(1+\nu)}{1+3\nu}}}{\varepsilon^{\frac{2}{1+3\nu}}},\frac{\sigma^2R_0^2}{\varepsilon^2}\right\}$ & $[0,1]$ & \cmark & \cmark & \cmark \\
        \hline
        \rowcolor{bgcolor2}\begin{tabular}{c}
            \algnames{clipped-SGD} \\ (Theorem~\ref{thm:main_result_clipped_SGD_main})
        \end{tabular} & $\max\left\{\frac{M_\nu^{\frac{2}{1+\nu}}R_0^2}{\varepsilon^{\frac{2}{1+\nu}}},\frac{\sigma^2R_0^2}{\varepsilon^2}\right\}$ & $[0,1]$ & \cmark & \cmark & \cmark \\
        \hline
    \end{tabular}}
\end{table}

\begin{table}[t]
    \centering
    \scriptsize
    \caption{\small Summary of known and new high-probability complexity bounds for solving \eqref{eq:main_problem} with $f$ being \textbf{$\mu$-strongly convex} and having $(\nu, M_\nu)$-H\"older continuous gradients. Columns: ``Complexity'' = high-probability complexity ($\varepsilon$ -- accuracy, $\beta$ -- confidence level, numerical constants, and logarithmic factors are omitted), ``HT'' = heavy-tailed noise, ``UD'' = unbounded domain, ``CS'' = bound on the variance of the stochastic gradient and H\"older continuity of the gradient is required only on the compact set. Notation: $R_0 = \|x^0 - x^*\|_2$, where $x^*$ is the closest solution to $x^0$; $\sigma^2$ = upper bound on the variance (see \eqref{eq:bounded_variance_clipped_SSTM}). The results labeled by $^{\clubsuit}$ are obtained from the convergence guarantees in expectation via Markov's inequality. Negative-power dependencies on the confidence level $\beta$ are colored in red. Our results are highlighted in green.}
    \label{tab:str_cvx_case_comparison}
    {
    \begin{tabular}{|c|c|c|c|c|c|}
        \hline
        Method & {\centering Complexity} & $\nu$ & HT? & UD? & CS?\\
        \hline
        \hline
        \begin{tabular}{c}
            \algnames{SGD} \\ \citep{nemirovski2009robust}
        \end{tabular} & $\max\left\{\frac{M_0^2}{\mu\varepsilon},\frac{\sigma^2}{\mu\varepsilon}\right\}$ & $0$ & \xmark & \xmark & \xmark\\
        \hline
        \begin{tabular}{c}
            \algnames{AC-SA} \\
            \citep{ghadimi2012optimal,lan2012optimal}
        \end{tabular} & $\max\left\{\sqrt{\frac{M_1}{\mu}},\frac{\sigma^2}{\mu\varepsilon}\right\}$ & $1$ & \xmark & \xmark & \xmark\\
        \hline
        \begin{tabular}{c}
            \algnames{SIGMA} \\ \citep{dvurechensky2016stochastic}
        \end{tabular} & \makecell{$\max\left\{\hat{N},\frac{\sigma^2}{\mu\varepsilon}\right\}$,\\ $\hat{N} = \max\!\left\{\! \left(\!\frac{M_\nu}{\mu R_0^{1\!-\!\nu}}\!\right)^{\!\frac{2}{1\!+\!3\nu}}\!,\! \left(\!\frac{M_\nu^2}{\mu^{1\!+\!\nu}\varepsilon^{1\!-\!\nu}}\!\right)^{\frac{1}{1\!+\!3\nu}}\!\right\}$} & $[0,1]$ & \xmark & \xmark & \xmark\\
        \hline
        \begin{tabular}{c}
            \algnames{SGD} \\  \citep{nemirovski2009robust}$^{\clubsuit}$
        \end{tabular} & $\max\left\{\frac{M_0^2}{\mu{\color{BrickRed}\beta}\varepsilon},\frac{\sigma^2}{\mu{\color{BrickRed}\beta}\varepsilon}\right\}$ & $0$ & \cmark & \xmark & \xmark\\
        \hline
        \begin{tabular}{c}
            \algnames{AC-SA} \\ \citep{ghadimi2012optimal,lan2012optimal}$^{\clubsuit}$
        \end{tabular} & $\max\left\{\sqrt{\frac{M_1}{\mu}},\frac{\sigma^2}{\mu{\color{BrickRed}\beta}\varepsilon}\right\}$ & $1$ & \cmark & \cmark & \xmark\\
        \hline
        \begin{tabular}{c}
            \algnames{SIGMA} \\ \citep{dvurechensky2016stochastic}$^{\clubsuit}$
        \end{tabular} & \makecell{$\max\left\{\hat{N},\frac{\sigma^2}{\mu\hat{\varepsilon}}\right\}$, $\hat{\varepsilon} = {\color{BrickRed}\beta}\varepsilon$,\\ $\hat{N} \!=\! \left(\!\frac{M_\nu}{\mu R_0^{1\!-\!\nu}}\!\right)^{\!\frac{2}{1\!+\!3\nu}}\!+\! \left(\!\frac{M_\nu^2}{\mu^{1\!+\!\nu}\hat{\varepsilon}^{1\!-\!\nu}}\!\right)^{\frac{1}{1\!+\!3\nu}}$} & $[0,1]$ & \cmark & \cmark & \xmark\\
        \hline
        \begin{tabular}{c}
            \algnames{R-clipped-SSTM} \\ \citep{gorbunov2020clipped_sstm}
        \end{tabular} & $\max\left\{\sqrt{\frac{M_1}{\mu}},\frac{\sigma^2}{\mu\varepsilon^2}\right\}$ & $1$ & \cmark & \cmark & \xmark \\
        \hline
        \begin{tabular}{c}
            \algnames{R-clipped-SGD} \\ \citep{gorbunov2020clipped_sstm}
        \end{tabular} & $\max\left\{\frac{M_1}{\mu},\frac{\sigma^2}{\mu\varepsilon^2}\right\}$ & $1$ & \cmark & \cmark & \xmark \\
        \hline
        \rowcolor{bgcolor2}\begin{tabular}{c}
            \algnames{R-clipped-SSTM} \\ (Theorem~\ref{thm:main_result_clipped_SSTM_str_cvx_main})
        \end{tabular} & \makecell{$\max\left\{\hat{N},\frac{\sigma^2}{\mu\varepsilon}\right\}$,\\ $\hat{N} = \max\!\left\{\! \left(\!\frac{M_\nu}{\mu R_0^{1\!-\!\nu}}\!\right)^{\!\frac{2}{1\!+\!3\nu}}\!,\! \left(\!\frac{M_\nu^2}{\mu^{1\!+\!\nu}\varepsilon^{1\!-\!\nu}}\!\right)^{\frac{1}{1\!+\!3\nu}}\!\right\}$} & $[0,1]$ & \cmark & \cmark & \cmark \\
        \hline
        \rowcolor{bgcolor2}\begin{tabular}{c}
            \algnames{R-clipped-SGD} \\ (Theorem~\ref{thm:main_result_clipped_SGD_str_cvx_main})
        \end{tabular} & $\max\left\{\frac{M_\nu^{\frac{2}{1+\nu}}}{\mu^{\frac{2}{1+\nu}}R_0^{\frac{2(1-\nu)}{1+\nu}}}, \frac{M_\nu^{\frac{2}{1+\nu}}}{\mu \varepsilon^{\frac{1-\nu}{1+\nu}}}, \frac{\sigma^2}{\mu\varepsilon}\right\}$ & $[0,1]$ & \cmark & \cmark & \cmark \\
        \hline
    \end{tabular}}
\end{table}

\begin{itemize}
    \item[$\diamond$] \textbf{The first near-optimal high-probability bounds for non-smooth problems with heavy-tailed noise.} We propose novel stepsize rules for \algname{clipped-SSTM} \citep{gorbunov2020clipped_sstm} to  {handle} problems with the objective having \revision{a} $(\nu, M_\nu)$-H\"older continuous gradient and derive in this setting high-probability complexity guarantees for convex stochastic optimization problems without using \revision{the} ``light tails'' assumption, i.e., we prove that  our version of \algname{clipped-SSTM} has
    \begin{equation*}
        \cO\left(\max\left\{D\ln^{\frac{2(1+\nu)}{1+3\nu}}\frac{D}{\beta},\frac{\sigma^2R_0^2}{\varepsilon^2}\ln\frac{D}{\beta}\right\}\right), \quad D = \frac{M_\nu^{\frac{2}{1+3\nu}}R_0^{\frac{2(1+\nu)}{1+3\nu}}}{\varepsilon^{\frac{2}{1+3\nu}}}
    \end{equation*}
    high-probability complexity. Unlike all previous high-probability complexity results in this setup with $\nu < 1$ (see Table~\ref{tab:cvx_case_comparison}), our result depends only logarithmically on the confidence level $\beta$ that is highly important when $\beta$ is small. Moreover, up to the difference in logarithmic factors, the derived complexity guarantees meet the known lower bounds \revision{\citep{guzman2015lower, lan2012optimal}} obtained \revision{for problems} with light-tailed noise.  {In particular, when $\nu = 1$\revision{,} we recover \revision{the} accelerated convergence rate \revision{\citep{lan2012optimal, nesterov1983method}}}. That is, neglecting the logarithmic factors, our results are unimprovable  {and,  {surprisingly}, coincide with the best-known results in the ``light-tailed case''.}
    
    \item[$\diamond$] \textbf{New high-probability bounds for \algname{clipped-SGD}.} We derive the first high-probability complexity bounds for \algname{clipped-SGD} when the objective function is convex with $(\nu, M_\nu)$-H\"older continuous gradient and the noise is heavy tailed., i.e., we derive
    \begin{equation*}
        \cO\left(\max\left\{D^2,\max\left\{D^{1+\nu},\frac{\sigma^2R_0^2}{\varepsilon^2}\right\}\ln\frac{D^2 + D^{1+\nu}}{\beta}\right\}\right),\quad D = \frac{M_\nu^{\frac{1}{1+\nu}}R_0}{\varepsilon^{\frac{1}{1+\nu}}}
    \end{equation*}
    \revision{the} high-probability complexity bound. Interestingly, when $\nu = 0$, the derived bound for \algname{clipped-SGD} has better dependence on the logarithms than the corresponding one for \algname{clipped-SSTM}. Moreover, neglecting the dependence on $\varepsilon$ under the logarithm, our bound for \algname{clipped-SGD} has the same dependence on the confidence level as the tightest known result in this case under the ``light tails'' assumption \citep{guigues2017non}.
    
    \item[$\diamond$] \textbf{Extensions to the strongly convex case.} Using \revision{the} restarts technique we extend the obtained results for \algname{clipped-SSTM} and \algname{clipped-SGD} to the strongly convex case (see Table~\ref{tab:str_cvx_case_comparison}). As in the convex case, the obtained results are superior to all previously known results in the general setup we consider. Moreover, the results derived for \algname{clipped-SSTM} are optimal up to logarithmic factors \revision{\citep{guzman2015lower, lan2012optimal}}.
    
    \item[$\diamond$] \textbf{Generality of the results.} As one of the key contributions of this work, we emphasize that in our theoretical results it is sufficient to assume boundedness of the variance of the stochastic gradient \eqref{eq:variance_bound_clipped_SSTM} and H\"older continuity of the gradients of $f$ only on the ball with radius $\sim R_0 = \|x^0 - x^*\|_2$ and centered at the closest to the starting point solution of the problem. This makes our results applicable to a much wider class of problems than functions with H\"older continuous gradients on $\R^n$, e.g., our analysis works even for polynomially growing objectives. Moreover, this feature of our analysis allows us to consider strongly convex functions. Indeed, the class of strongly convex functions on $\R^n$ with H\"older continuous gradients on $\R^n$ with $\nu < 1$ is empty. Therefore, it is crucial to assume H\"older continuity of gradients only on a bounded set\footnote{It is also worth mentioning that some functions have H\"older continuous gradients for multiple $\nu$ simultaneously \citep{nesterov2015universal}. Therefore, if constants $M_\nu$ are available, one can choose the best $\nu$ in terms of the iteration/oracle complexity of a method.}. Next, we do not require the variance of the stochastic gradient to be uniformly bounded on the whole space, e.g., we allow the variance at point $x$ to grow when $\|x - x^*\|_2 \to \infty$. We emphasize that even for smooth problems ($\nu = 1$) all previous works in the area rely on the uniform boundedness of the variance on the whole space (see Tables~\ref{tab:cvx_case_comparison} and \ref{tab:str_cvx_case_comparison}). Next, in the works focusing on the ``light tails'' case, the uniform boundedness of sub-Gaussian variance and H\"older continuity of the gradients are also assumed on $\R^n$. All of these facts emphasize the generality of our results.
    
    \item[$\diamond$] \textbf{Experiments.} To test the performance of the considered methods, we conduct several numerical experiments on image classification and NLP tasks and observe that 1) \algname{clipped-SSTM} and \algname{clipped-SGD} show comparable performance with \algname{SGD} on the image classification task, when the noise distribution is almost sub-Gaussian, 2) converge much faster than \algname{SGD} on the NLP task, when the noise distribution is heavy-tailed, and 3) \algname{clipped-SSTM} achieves a comparable performance with \algname{Adam} on the NLP task enjoying both the best known theoretical guarantees and good practical performance. We also compare \algname{clipped-SSTM}, \algname{clipped-SGD}, \algname{SGD}, and \algname{Adam} on solving the convex problem, corresponding to the linear regression with the noise having generalized Gaussian distribution. Our codes are publicly available: \url{https://github.com/ClippedStochasticMethods/clipped-SSTM}.
\end{itemize}

\subsection{Related Work}
\paragraph{Light-tailed noise.} The theory of high-probability complexity bounds for convex stochastic optimization with light-tailed noise is well-developed. Lower bounds and optimal methods for the problems with $(\nu,M_\nu)$-H\"older continuous gradients are obtained in \citep{nemirovski2009robust} for $\nu = 0$, and in \citep{ghadimi2012optimal} for $\nu = 1$. Up to the logarithmic dependencies, these high-probability convergence bounds coincide with the corresponding results for the convergence in expectation (see first two rows of Table~\ref{tab:cvx_case_comparison}).
While not being directly derived in the literature, the lower bound for the case when $\nu \in (0,1)$ can be obtained as a combination of lower bounds in the deterministic \revision{\citep{guzman2015lower, nemirovsky1983problem}} and smooth stochastic settings \citep{ghadimi2012optimal}. The corresponding optimal methods are analyzed in \citep{devolder2013exactness,dvurechensky2016stochastic} \revision{based on the concept of inexact oracle.}
\vspace{-1em}
\paragraph{Heavy-tailed noise.} Unlike in the ``light-tailed'' case, the first theoretical guarantees with reasonable dependence on both the accuracy $\varepsilon$ and the confidence level $\beta$ appeared just recently. In \citep{nazin2019algorithms}, the first such results \revision{without the kind of acceleration appearing in \citep{nesterov1983method}} were derived for Mirror Descent with special truncation technique for smooth ($\nu = 1$) convex problems on a bounded domain, and then were accelerated and extended in \citep{gorbunov2020clipped_sstm}. \revision{For strongly} convex problems, the first accelerated high-probability convergence guarantees were obtained in \citep{davis2021low} for the special method called \algname{proxBoost} \revision{requiring the solving of nontrivial auxiliary} problems at each iteration. These bounds were tightened in \citep{gorbunov2020clipped_sstm}.

In contrast, for the case when $\nu < 1$ and, in particular, when $\nu = 0$ the best-known high-probability complexity bounds suffer from the negative-power dependence on the confidence level $\beta$, i.e., have a factor $\nicefrac{1}{\beta^{\alpha}}$ for some $\alpha >0$, that affects the convergence rate dramatically for small enough $\beta$. Without additional assumptions on the tails these results are obtained via Markov's inequality $\PP\{f(x) - f(x^*) > \varepsilon\} < \nicefrac{\EE[f(x) - f(x^*)]}{\varepsilon}$ from the guarantees for the convergence in expectation to the accuracy $\varepsilon\beta$, see the results labeled by $^\clubsuit$ in Table~\ref{tab:cvx_case_comparison}. Under an additional assumption on noise tails that $\PP\{\|\nabla f(x,\xi) - \nabla f(x)\|_2^2 > s\sigma^2\} = \revision{\cO}(s^{-\alpha})$ for $\alpha > 2$ these results can be tightened \citep{gasnikov2015efficiency} when $\nu = 0$ as 
$\sim \max\left\{ \nicefrac{\ln\left(\beta^{-1}\right)}{\varepsilon^2},\left(\nicefrac{1}{\beta\varepsilon^{\alpha}}\right)^{\nicefrac{2}{(3\alpha - 2)}}\right\}$
without removing the negative-power dependence on the confidence level $\beta$. Different stepsize policies allow to change the last term in $\max$ to $\beta^{-\frac{1}{2\alpha - 1}}\varepsilon^{-\frac{2\alpha}{2\alpha - 1}}$ without removing the negative-power dependence on $\beta$.
\vspace{-1em}
\paragraph{Comparison with \citep{gorbunov2020clipped_sstm}.} Although our results and proof technique are based on the ones proposed in \citep{gorbunov2020clipped_sstm}, our work extends and significantly differs from \citep{gorbunov2020clipped_sstm}. First of all, we consider problems with H\"{o}lder continuous gradients, while the authors of \citep{gorbunov2020clipped_sstm} obtain their results only for the smooth functions. To derive a proper generalization of the results from \citep{gorbunov2020clipped_sstm}, we propose different stepsizes for \algname{clipped-SSTM} and \algname{clipped-SGD} and we also modify the proofs significantly to circumvent the additional issues arising due to the partial smoothness of the problem, especially in the part where we prove high-probability bound for the norm of the gradient (see the derivation of inequality \eqref{eq:main_thm_clipped_SSTM_technical_4}). Since this part is one of the most important ones in the proof, this fact highlights the difference between two approaches. Moreover, \citep{gorbunov2020clipped_sstm} assume that the variance is uniformly upper bounded and the gradient is Lipschitz-continuous on $\R^n$, while our analysis relies on much weaker assumptions that \eqref{eq:bounded_variance_clipped_SSTM} and \eqref{eq:holder_def} hold on a ball around the solution, i.e., on a compact set. Thus, our results are proven for a much wider class of problems, including ones with polynomially growing objective function/variance $\EE_\xi\left[\left\|\nabla f(x,\xi) - \nabla f(x)\right\|_2^2\right]$ when $\|x - x^*\|_2 \to \infty$.

\paragraph{Gradient clipping.} The methods based on gradient clipping \citep{pascanu2013difficulty} and normalization \citep{hazan2015beyond} are popular in different machine learning and deep learning tasks due to their robustness in practice to the noise in the stochastic gradients and rapid changes of the objective function \citep{Goodfellow-et-al-2016}. In \revision{\citep{mai2021stability, zhang2020gradient}}, \algname{clipped-GD} and \algname{clipped-SGD} are theoretically studied in applications to non-smooth problems with an objective that can grow polynomially fast when $\|x - x^*\|_2 \to \infty$ showing the superiority of gradient clipping methods to the methods without clipping. The results from \citep{zhang2020gradient} are obtained for non-convex problems with almost surely bounded noise, and in \citep{mai2021stability}, the authors derive the stability and expectation convergence guarantees for strongly convex objectives  {under an assumption that the central $p$-th moment of the stochastic gradient is bounded for $p \ge 2$. Since \citep{mai2021stability} do not provide convergence guarantees with explicit dependencies on all important parameters of the problem, it complicates direct comparison with our results. Nevertheless, convergence guarantees from \citep{mai2021stability} are sub-linear and are given for the convergence in expectation, and, as a consequence, the corresponding high-probability convergence results obtained via the Markov's inequality also suffer from negative-power dependence on the confidence level.} Next, the authors of \citep{zhang2020why} establish several expectation convergence guarantees for \algname{clipped-SGD} and prove their optimality in the non-convex case under the assumption that the central $\alpha$-moment of the stochastic gradient is uniformly bounded, where $\alpha \in (1,2]$. It turns out that \algname{clipped-SGD} is able to converge even when $\alpha < 2$, whereas vanilla \algname{SGD} can diverge in this setting.

\subsection{Paper Organization}
We present and discuss the simplified versions of our main results on \algname{clipped-SSTM} (Theorems~\ref{thm:main_result_clipped_SSTM_main} and \ref{thm:main_result_clipped_SSTM_str_cvx_main}) and \algname{clipped-SGD} (Theorems~\ref{thm:main_result_clipped_SGD_main} and \ref{thm:main_result_clipped_SGD_str_cvx_main}) in Sections~\ref{sec:clipped_SSTM_main} and \ref{sec:clipped_SGD_main} respectively. The detailed statements of the main results, complete proofs, and the corollaries for unit batch sizes (Corollaries~\ref{cor:SSTM_cvx_small_batch} and \ref{cor:SGD_cvx_small_batch}) are given in Sections~\ref{sec:clipped_SSTM_appendix} and \ref{sec:clipped_SGD_appendix}. Finally, Section~\ref{sec:experiments} contains the results of our numerical experiments. In Appendix~\ref{sec:basic_facts}, we give some useful auxiliary lemmas and prove few technical results. In Appendix~\ref{sec:extra_experiments}, we provide a detailed description of the setup for numerical experiments and extra numerical results. 

\section{Clipped Stochastic Similar Triangles Method}\label{sec:clipped_SSTM_main}
In this section, we propose a novel variation of Clipped Stochastic Similar Triangles Method \citep{gorbunov2020clipped_sstm} adjusted to the class of objectives with H\"older continuous gradients (\algname{clipped-SSTM}, see Algorithm~\ref{alg:clipped-SSTM}).

The method is based on the clipping of the stochastic gradients:
\begin{equation}
    \clip(\nabla f(x,\Bxi), \lambda) = \min\left\{1,\frac{\lambda}{\|\nabla f(x,\Bxi)\|_2}\right\}\nabla f(x,\Bxi) \label{eq:clip_operator}
\end{equation}
where $\nabla f(x,\Bxi)=\frac{1}{m}\sum_{i=1}^m\nabla f(x,\xi_i)$ is a mini-batched stochastic gradient \revision{and for shortness we denote by $\Bxi$ the collection $\{ \xi_i\}_{i=1}^m$ of samples.} Gradient clipping ensures that the resulting vector has a norm bounded by the clipping level $\lambda$. Since the clipped stochastic gradient cannot have an arbitrary large norm, the clipping helps to avoid unstable behavior of the method when the noise is heavy-tailed, and the clipping level $\lambda$ is properly adjusted.

However, unlike the stochastic gradient, the clipped stochastic gradient is a \textit{biased} estimate of $\nabla f(x)$: the smaller the clipping level, the larger the bias. The biasedness of the clipped stochastic gradient complicates the analysis of the method. On the other hand, to circumvent the negative effect of the heavy-tailed noise on the high-probability convergence, one should choose $\lambda$ to be not too large. Therefore, the question on the appropriate choice of the clipping level is highly non-trivial.

Fortunately, there exists a simple but insightful observation that helps us to obtain the right formula for the clipping level $\lambda_{k}$ \revision{at iteration $k$} in \algname{clipped-SSTM}: if $\lambda_k$ is chosen in such a way that $\|\nabla f(x^{k})\|_2 \le \nicefrac{\lambda_k}{2}$ with high probability, then for the realizations $\nabla f(x^{k+1},\Bxi^{k})$ of the mini-batched stochastic gradient such that $\|\nabla f(x^{k+1},\Bxi^{k}) - \nabla f(x^{k+1})\|_2 \le \nicefrac{\lambda_k}{2}$ the clipping is an identity operator. Next, if the probability mass of such realizations is big enough, then the bias of the clipped stochastic gradient is properly \revision{bounded, which helps derive
the needed} convergence guarantees. It turns out that the choice $\lambda_k \sim \nicefrac{1}{\alpha_k}$ ensures the method convergence with the needed rate and high enough probability.

\begin{algorithm}[h]
\caption{Clipped Stochastic Similar Triangles Method (\algname{clipped-SSTM}): case $\nu \in [0,1]$}
\label{alg:clipped-SSTM}   
\begin{algorithmic}[1]
\Require starting point $x^0$, number of iterations $N$, batch sizes $\{m_k\}_{k=1}^N $, stepsize parameter $\alpha$, clipping parameter $B$, H\"older exponent $\nu \in [0,1]$.
\State Set $A_0 = \alpha_0 = 0$, $y^0 = z^0 = x^0$
\For{\revision{$k=0,1,\ldots, N-1$}}
\State Set $\alpha_{k+1} = \alpha (k+1)^{\frac{2\nu}{1+\nu}}$, $A_{k+1} = A_k + \alpha_{k+1}$, $\lambda_{k+1} = \frac{B}{\alpha_{k+1}}$
\State $x^{k+1} = \nicefrac{(A_k y^k + \alpha_{k+1} z^k)}{A_{k+1}}$
\State Draw mini-batch $m_k$ of fresh i.i.d.\ samples $\xi_1^k,\ldots,\xi_{m_{k}}^k$ and compute $\nabla f(x^{k+1},\Bxi^k) = \frac{1}{m_k}\sum_{i=1}^{m_k}\nabla f(x^{k+1},\xi_i^k)$
\State Compute $\tnabla f(x^{k+1},\Bxi^k) = \clip(\nabla f(x^{k+1},\Bxi^k),\lambda_{k+1})$ using \eqref{eq:clip_operator} 
\State $z^{k+1} = z^k - \alpha_{k+1}\tnabla f(x^{k+1},\Bxi^k)$
\State $y^{k+1} = \nicefrac{(A_k y^k + \alpha_{k+1} z^{k+1})}{A_{k+1}}$
\EndFor
\Ensure $y^N$ 
\end{algorithmic}
\end{algorithm}

Guided by this observation, we derive the precise expressions for all the parameters of \algname{clipped-SSTM} and derive high-probability complexity bounds for the method. Below, we provide a simplified version of the main result for \algname{clipped-SSTM} in the convex case. The complete formulation and the full proof of the theorem are deferred to Section~\ref{sec:clipped_SSTM_cvx_appendix} (see Theorem~\ref{thm:main_result_clipped_SSTM}).
\begin{theorem}[Simplified version of Theorem~\ref{thm:main_result_clipped_SSTM}]\label{thm:main_result_clipped_SSTM_main}
    Assume that function $f$ is convex, its stochastic gradient and its gradient satisfy \eqref{eq:bounded_variance_clipped_SSTM} and \eqref{eq:holder_def} respectively with $\sigma > 0$, $\nu \in [0,1]$, $M_\nu > 0$ on $Q = B_{3R_0}(x^*) = \{x\in\R^n\mid \|x-x^*\|_2 \le 3R_0\}$, where $R_0 \ge \|x^0 - x^*\|_2$. Then there exists such a choice of parameters that \algname{clipped-SSTM} achieves $f(y^N) - f(x^*) \le \varepsilon$ with probability at least $1-\beta$ after $\cO\left(D\ln^{\frac{2(1+\nu)}{1+3\nu}}\frac{D}{\beta}\right)$ iterations with $D = \frac{M_\nu^{\frac{2}{1+3\nu}}R_0^{\frac{2(1+\nu)}{1+3\nu}}}{\varepsilon^{\frac{2}{1+3\nu}}}$ and requires
    \begin{equation}
        \cO\left(\max\left\{D\ln^{\frac{2(1+\nu)}{1+3\nu}}\frac{D}{\beta},\frac{\sigma^2R_0^2}{\varepsilon^2}\ln\frac{D}{\beta}\right\}\right)\text{ oracle calls.} \label{eq:clipped_SSTM_oracle_complexity_main}
    \end{equation}
\end{theorem}

The obtained result has only logarithmic dependence on the confidence level $\beta$ and optimal dependence on the accuracy $\varepsilon$ up to logarithmic factors \revision{\citep{guzman2015lower, lan2012optimal}} and matches the state-of-the-art results in the light-tailed case \revision{\citep{dvurechensky2016stochastic, ghadimi2012optimal, nemirovski2009robust}} for all $\nu \in [0,1]$. Moreover, the complexity bounds from \revision{\citep{dvurechensky2016stochastic, ghadimi2012optimal, nemirovski2009robust}} are proportional to $\cO\left(\ln^2 \tfrac{1}{\beta}\right)$ (neglecting the dependence on $M_\nu$ and $R_0$), while our bound has better dependence on the power of the logarithm when $\nu > 0$. In particular, when $\nu = 1$, our bound is proportional to $\cO\left(\ln\tfrac{1}{\sqrt{\varepsilon}\beta}\right)$. When $\beta$ is small enough (of the same order with $\varepsilon$ or smaller), our logarithmic factor is much smaller than $\cO\left(\ln^2 \tfrac{1}{\beta}\right)$.

Next, we emphasize that our result does not require $f$ to have $(\nu,M_\nu)$-H\"older continuous gradient and the variance of the stochastic gradient to be uniformly bounded \textit{on the whole space}. To achieve this, we prove that for the proposed choice of parameters the iterates of \algname{clipped-SSTM} stay inside the ball $B_{3R_0} = \{x\in\R^n\mid \|x-x^*\|_2 \le 3R_0\}$ with probability at least $1 - \beta$, and, as a consequence, it is sufficient to assume that \eqref{eq:bounded_variance_clipped_SSTM} and \eqref{eq:holder_def} hold only inside this ball. In particular, this means that the better starting point leads not only to the reduction of $R_0$, but also it can reduce $M_\nu$ and $\sigma$. Moreover, our result is applicable to a much wider class of functions than the standard class of functions with H\"older continuous gradients in $\R^n$, e.g., to the problems with polynomial growth in both the gradient and the variance of the stochastic estimator.

For the strongly convex problems, we consider a restarted version of Algorithm~\ref{alg:clipped-SSTM} (\algname{R-clipped-SSTM}, see Algorithm~\ref{alg:R-clipped-SSTM}) and derive high-probability complexity result for this version.
Below we provide a simplified version of the result. The complete formulation and the full proof of the theorem are deferred to Section~\ref{sec:clipped_SSTM_str_cvx_appendix} (see Theorem~\ref{thm:main_result_clipped_SSTM_str_cvx}).
\begin{theorem}[Simplified version of Theorem~\ref{thm:main_result_clipped_SSTM_str_cvx}]\label{thm:main_result_clipped_SSTM_str_cvx_main}
    Assume that function $f$ is $\mu$-strongly convex,  its stochastic gradient and its gradient satisfy \eqref{eq:bounded_variance_clipped_SSTM} and \eqref{eq:holder_def} respectively with $\sigma > 0$, $\nu \in [0,1]$, $M_\nu > 0$ on $Q = B_{3R_0}(x^*) = \{x\in\R^n\mid \|x-x^*\|_2 \le 3R_0\}$, where $R_0 \ge \|x^0 - x^*\|_2$. Then there exists such a choice of parameters that \algname{R-clipped-SSTM} achieves $f(\hat{x}^\tau) - f(x^*) \le \varepsilon$ with probability at least $1-\beta$ after
    \begin{equation}
        \hat{N} = \revision{\cO}\left(D\ln^{\frac{2(1+\nu)}{1+3\nu}}\frac{D}{\beta}\right),\; D = \max\left\{\left(\frac{M_\nu}{\mu R_0^{1-\nu}}\right)^{\frac{2}{1+3\nu}}\ln\frac{\mu R_0^2}{\varepsilon}, \left(\frac{M_\nu^2}{\mu^{1+\nu}\varepsilon^{1-\nu}}\right)^{\frac{1}{1+3\nu}}\right\} \label{eq:clipped_SSTM_iter_complexity_str_cvx_main}
    \end{equation}
    iterations of Algorithm~\ref{alg:clipped-SSTM} in total and requires
    \begin{equation}
        \revision{\cO}\left(\max\left\{D\ln^{\frac{2(1+\nu)}{1+3\nu}}\frac{D}{\beta},\frac{\sigma^2}{\mu\varepsilon}\ln\frac{D}{\beta}\right\}\right)\text{ oracle calls.} \label{eq:clipped_SSTM_oracle_complexity_str_cvx_main}
    \end{equation}
\end{theorem}
Again, the obtained result has only logarithmic dependence on the confidence level $\beta$ and, as our result in the convex case, it has optimal dependence on the accuracy $\varepsilon$ up to logarithmic factors depending on $\beta$ \revision{\citep{guzman2015lower, lan2012optimal}} for all $\nu \in [0,1]$.

\begin{algorithm}[h]
\caption{Restarted \algname{clipped-SSTM} (\algname{R-clipped-SSTM}): case $\nu \in [0,1]$}
\label{alg:R-clipped-SSTM}   
\begin{algorithmic}[1]
\Require starting point $x^0$, number of restarts $\tau$, number of steps of \algname{clipped-SSTM} in restarts $\{N_t\}_{t=1}^{\tau}$, batch sizes $\{m_{k}^1\}_{k=1}^{N_1-1}, \{m_{k}^2\}_{k=1}^{N_2-1}, \ldots, \{m_{k}^{\tau}\}_{k=1}^{N_{\tau}-1}$, stepsize parameters $\{\alpha^t\}_{t=1}^\tau$, clipping parameters $\{B_t\}_{t=1}^\tau$, H\"older exponent $\nu \in [0,1]$.
\State $\hat{x}^0 = x^0$
\For{$t=1,\ldots, \tau$}
\State Run \algname{clipped-SSTM} (Algorithm~\ref{alg:clipped-SSTM}) for $N_t$ iterations with batch sizes $\{m_k^t\}_{k=1}^{N_t-1}$, stepsize parameter $\alpha_{t}$, clipping parameter $B_t$, and starting point $\hat{x}^{t-1}$. Define the output of \algname{clipped-SSTM} by $\hat{x}^{t}$.
\EndFor
\Ensure $\hat{x}^\tau$ 
\end{algorithmic}
\end{algorithm}

\section{SGD with Clipping}\label{sec:clipped_SGD_main}

In this section, we present a new variant of \algname{clipped-SGD} \citep{pascanu2013difficulty} properly adjusted to the class of objectives with $(\nu, M_\nu)$-H\"older continuous gradients (see Algorithm~\ref{alg:clipped-SGD}). 

We emphasize that as for \algname{clipped-SSTM} we use clipping level $\lambda$ inversely proportional to the stepsize \revision{$\alpha$}. Below, we provide a simplified version of the main result for \algname{clipped-SGD} in the convex case. The complete formulation and the full proof of the theorem are deferred to Section~\ref{sec:clipped_SGD_cvx_appendix} (see Theorem~\ref{thm:main_result_clipped_SGD}).
\begin{theorem}[Simplified version of Theorem~\ref{thm:main_result_clipped_SGD}]\label{thm:main_result_clipped_SGD_main}
    Assume that function $f$ is convex,  its stochastic gradient and its gradient satisfy \eqref{eq:bounded_variance_clipped_SSTM} and \eqref{eq:holder_def} respectively with $\sigma > 0$, $\nu \in [0,1]$, $M_\nu > 0$ on $Q = B_{7R_0}(x^*) = \{x\in\R^n\mid \|x-x^*\|_2 \le 7R_0\}$, where $R_0 \ge \|x^0 - x^*\|_2$. Then there exists such a choice of parameters that \algname{clipped-SGD} achieves $f(\Bar{x}^N) - f(x^*) \le \varepsilon$ with probability at least $1-\beta$ after
    \begin{equation}
        \cO\left(\max\left\{D^2,D^{1+\nu}\ln\frac{D^2 + D^{1+\nu}}{\beta}\right\}\right),\quad D = \frac{M_\nu^{\frac{1}{1+\nu}}R_0}{\varepsilon^{\frac{1}{1+\nu}}} \label{eq:clipped_SGD_iter_complexity_main}
    \end{equation}
    iterations and requires
    \begin{equation}
       \cO\left(\max\left\{D^2,\max\left\{D^{1+\nu},\frac{\sigma^2R_0^2}{\varepsilon^2}\right\}\ln\frac{D^2 + D^{1+\nu}}{\beta}\right\}\right) \text{ oracle calls.} \label{eq:clipped_SGD_oracle_complexity_main}
    \end{equation}
\end{theorem}
\begin{algorithm}[h]
\caption{Clipped Stochastic Gradient Descent (\algname{clipped-SGD}): case $\nu \in [0,1]$}
\label{alg:clipped-SGD}   
\begin{algorithmic}[1]
\Require starting point $x^0$, number of iterations $N$, batch size $m$, stepsize $\gamma$, clipping parameter $B > 0$.
\For{\revision{$k=0,1,\ldots, N-1$}}
\State Draw mini-batch of $m$ fresh i.i.d.\ samples $\xi_1^k,\ldots,\xi_{m}^k$ and compute $\nabla f(x^{k+1},\Bxi^k) = \frac{1}{m}\sum_{i=1}^{m}\nabla f(x^{k+1},\xi_i^k)$
\State Compute $\tnabla f(x^{k},\Bxi^k) = \clip(\nabla f(x^{k},\Bxi^k),\lambda)$ using \eqref{eq:clip_operator} with $\lambda = \nicefrac{B}{\gamma}$
\State $x^{k+1} = x^k - \gamma\tnabla f(x^{k},\Bxi^k)$
\EndFor
\Ensure $\Bar{x}^N = \frac{1}{N}\sum_{k=0}^{N-1}x^k$ 
\end{algorithmic}
\end{algorithm}
As all our results in the paper, this result for \algname{clipped-SGD} has two important features: 1) the dependence on the confidence level $\beta$ is logarithmic and 2) H\"older continuity and uniformly bounded variance assumptions are required only on the ball $B_{7R_0}(x^*)$ centered at the solution. Moreover, up to the difference in the expressions under the logarithm, the dependence on $\varepsilon$ in the result for \algname{clipped-SGD} is the same as in the tightest known results for non-accelerated \algname{SGD}-type methods \citep{devolder2013exactness, guigues2017non}. Finally, we emphasize that for $\nu < 1$ the logarithmic factors appearing in the complexity bound for \algname{clipped-SSTM} are worse than the corresponding factor in the complexity bound for \algname{clipped-SGD}. Therefore, \algname{clipped-SGD} has the best-known high-probability complexity results in the case when $\nu = 0$ and $f$ is convex. Furthermore, when $\nu =0$, our result has $\cO(\ln\tfrac{1}{\varepsilon^{2}\beta})$ logarithmic factor, while the best-known high-probability results under ``light tails'' assumption are proportional to $\cO(\ln^2\tfrac{1}{\beta})$ \revision{\citep{dvurechensky2016stochastic, nemirovski2009robust}}. When $\beta$ is small enough (of the same order with $\varepsilon$ or smaller), our logarithmic factor is much smaller than $\cO\left(\ln^2 \tfrac{1}{\beta}\right)$.

For the strongly convex problems, we consider a restarted version of Algorithm~\ref{alg:clipped-SGD} (\algname{R-clipped-SGD}, see Algorithm~\ref{alg:R-clipped-SGD}) and derive high-probability complexity result for this version.
\begin{algorithm}[h]
\caption{Restarted \algname{clipped-SGD} (\algname{R-clipped-SGD}): case $\nu \in [0,1]$}
\label{alg:R-clipped-SGD}   
\begin{algorithmic}[1]
\Require starting point $x^0$, number of restarts $\tau$, number of steps of \algname{clipped-SGD} in restarts $\{N_t\}_{t=1}^{\tau}$, batch sizes $\{m_{t}\}_{k=1}^{\tau}$, stepsizes $\{\gamma_t\}_{t=1}^\tau$, clipping parameters $\{B_t\}_{t=1}^\tau$
\State $\hat{x}^0 = x^0$
\For{$t=1,\ldots, \tau$}
\State Run \algname{clipped-SGD} (Algorithm~\ref{alg:clipped-SGD}) for $N_t$ iterations with batch size $m_t$, stepsize $\gamma_{t}$, clipping parameter $B_t$, and starting point $\hat{x}^{t-1}$. Define the output of \algname{clipped-SGD} by $\hat{x}^{t}$.
\EndFor
\Ensure $\hat{x}^\tau$ 
\end{algorithmic}
\end{algorithm}
Below we provide a simplified version of the result. The complete formulation and the full proof of the theorem are deferred to Section~\ref{sec:clipped_SGD_str_cvx_appendix} (see Theorem~\ref{thm:main_result_clipped_SGD_str_cvx}).
\begin{theorem}[Simplified version of Theorem~\ref{thm:main_result_clipped_SGD_str_cvx}]\label{thm:main_result_clipped_SGD_str_cvx_main}
    Assume that function $f$ is $\mu$-strongly convex, its stochastic gradient and its gradient satisfy \eqref{eq:bounded_variance_clipped_SSTM} and \eqref{eq:holder_def} respectively with $\sigma > 0$, $\nu \in [0,1]$, $M_\nu > 0$ on $Q = B_{7R_0}(x^*) = \{x\in\R^n\mid \|x-x^*\|_2 \le 7R_0\}$, where $R_0 \ge \|x^0 - x^*\|_2$. Then there exists such a choice of parameters that \algname{R-clipped-SGD} achieves $f(\Bar{x}^N) - f(x^*) \le \varepsilon$ with probability at least $1-\beta$ after  
    \begin{equation}
        \cO\left(\max\left\{D_1^{\frac{2}{1+\nu}}\ln\frac{\mu R_0^2}{\varepsilon}, D_2^{\frac{2}{1+\nu}}, \max\left\{D_1\ln\frac{\mu R_0^2}{\varepsilon},D_2\right\}\ln\frac{D}{\beta} \right\}\right) \notag
    \end{equation}
    iterations of Algorithm~\ref{alg:clipped-SGD} in total and requires
    \begin{equation*}
        \cO\left(\max\left\{D_1^{\frac{2}{1+\nu}}\ln\frac{\mu R_0^2}{\varepsilon}, D_2^{\frac{2}{1+\nu}}, \max\left\{D_1\ln\frac{\mu R_0^2}{\varepsilon},D_2, \frac{\sigma^2}{\mu\varepsilon}\right\}\ln\frac{D}{\beta} \right\}\right) \; 
    \end{equation*}
    oracle calls, where
    \begin{equation*}
        D_1 = \frac{M_\nu}{\mu R_0^{1-\nu}},\quad D_2 = \frac{M_\nu}{\mu^{\frac{1+\nu}{2}}\varepsilon^{\frac{1-\nu}{2}}},\quad D = (D_1^{\frac{2}{1+\nu}}+D_1)\ln\frac{\mu R_0^2}{\varepsilon} + D_2 + D_2^{\frac{2}{1+\nu}}.
    \end{equation*}
\end{theorem}
As in the convex case, for $\nu < 1$, the log-factors appearing in the complexity bound for \algname{R-clipped-SSTM} are worse than the corresponding factor in the bound for \algname{R-clipped-SGD}. Thus, \algname{R-clipped-SGD} has the best-known high-probability complexity results for strongly convex $f$ and $\nu = 0$.

\section{Clipped Similar Triangles Method: Missing Details and Proofs}\label{sec:clipped_SSTM_appendix}

\subsection{Convergence in the Convex Case}\label{sec:clipped_SSTM_cvx_appendix}
In this section, we provide the full proof of Theorem~\ref{thm:main_result_clipped_SSTM_main} together with a complete statement of the result.
\subsubsection{Two lemmas}
The analysis of \algname{clipped-SSTM} consists of \revision{three} main steps. The first one is an ``optimization lemma'' -- a modification of a standard lemma for Similar Triangles Method (see \citep{gasnikov2016universal} and Lemma~F.4 from \citep{gorbunov2020clipped_sstm}). This result helps to estimate the progress of the method after $N$ iterations.
\begin{lemma}\label{lem:main_opt_lemma_clipped_SSTM}
    Let $f$ be a convex function  with a minimum at some\footnote{ Our proofs are valid for any solution $x^*$ and, for example, one can take as $x^*$ the closest solution to the starting point $x^0$.}  point $x^*$, its gradient be $(\nu, M_\nu)$-H\"older continuous on a ball $B_{3R_0}(x^*)$, where $R_0 \ge \|x^0 - x^*\|_2$, and let stepsize parameter  \revision{$\alpha$ have the form $\alpha=\frac{ (\nicefrac{\varepsilon}{2})^{\frac{1-\nu}{1+\nu}}}{2^{\frac{2\nu}{1+\nu}}aM_\nu^{\frac{2}{1+\nu}}}$, where $a\ge 1$.}
    If $x^{k}, y^k, z^k \in B_{3R_0}(x^*)$ for all $k = 0,1,\ldots,N$, $N \ge 0$, then after $N$ iterations of \algname{clipped-SSTM} for all $z\in \R^n$ we have
    \begin{eqnarray}
        &&A_N\left(f(y^N) - f(z)\right) \le \frac{1}{2}\|z^0 - z\|_2^2 - \frac{1}{2}\|z^{N} - z\|_2^2 + \sum\limits_{k=0}^{N-1}\alpha_{k+1}\left\la \theta_{k+1}, z - z^{k}\right\ra\notag\\
        &&\quad + \sum\limits_{k=0}^{N-1}\alpha_{k+1}^2\left\|\theta_{k+1}\right\|_2^2 + \sum\limits_{k=0}^{N-1}\alpha_{k+1}^2\left\la\theta_{k+1},\nabla f(x^{k+1})\right\ra   + \frac{A_{N}\varepsilon}{4},\label{eq:main_opt_lemma_clipped_SSTM}\\
        &&\theta_{k+1} \eqdef \tnabla f(x^{k+1},\Bxi^k) - \nabla f(x^{k+1}).\label{eq:theta_k+1_def_clipped_SSTM}
    \end{eqnarray}
\end{lemma}
\begin{proof}
    Consider an arbitrary $k \in \{0,1,\ldots,N-1\}$. Using \\$z^{k+1} = z^k - \alpha_{k+1}\tnabla f(x^{k+1},\Bxi^k)$ we get that for all $z\in \R^n$
    \begin{eqnarray}
        &&\alpha_{k+1}\left\la\tnabla f(x^{k+1},\Bxi^k), z^k - z \right\ra \notag \\
        &=& \alpha_{k+1}\left\la \tnabla f(x^{k+1},\Bxi^{k}), z^k - z^{k+1}\right\ra + \alpha_{k+1}\left\la \tnabla f(x^{k+1},\Bxi^{k}), z^{k+1} - z\right\ra\notag\\
        &=& \alpha_{k+1}\left\la \tnabla f(x^{k+1},\Bxi^{k}), z^k - z^{k+1}\right\ra + \left\la z^{k+1}-z^k, z - z^{k+1}\right\ra \notag\\
        &\overset{\eqref{eq:inner_product_representation}}{=}& \alpha_{k+1}\left\la \tnabla f(x^{k+1},\Bxi^{k}), z^k - z^{k+1}\right\ra - \frac{1}{2}\|z^k - z^{k+1}\|_2^2 \notag\\
        &&\quad+ \frac{1}{2}\|z^k - z\|_2^2 - \frac{1}{2}\|z^{k+1}-z\|_2^2.\label{eq:main_olc_clipped_SSTM_technical_2}
    \end{eqnarray}
    Next, we notice that
    \begin{eqnarray}
         y^{k+1} &=& \frac{A_k y^k + \alpha_{k+1}z^{k+1}}{A_{k+1}} = \frac{A_k y^k + \alpha_{k+1}z^{k}}{A_{k+1}} + \frac{\alpha_{k+1}}{A_{k+1}}\left(z^{k+1}-z^k\right)\notag\\
         &=& x^{k+1} + \frac{\alpha_{k+1}}{A_{k+1}}\left(z^{k+1}-z^k\right) \label{eq:main_olc_clipped_SSTM_technical_3}
    \end{eqnarray}
    implying
    \begin{eqnarray}
        &\alpha_{k+1}&\left\la\tnabla f(x^{k+1},\Bxi^k), z^k - z \right\ra \notag \\
        &\overset{\eqref{eq:theta_k+1_def_clipped_SSTM},\eqref{eq:main_olc_clipped_SSTM_technical_2}}{\le} &\alpha_{k+1}\left\la \nabla f(x^{k+1}), z^{k} - z^{k+1}\right\ra - \frac{1}{2}\|z^k - z^{k+1}\|_2^2 \notag\\
        &&\quad + \alpha_{k+1}\left\la \theta_{k+1}, z^{k} - z^{k+1}\right\ra  + \frac{1}{2}\|z^k - z\|_2^2 - \frac{1}{2}\|z^{k+1}-z\|_2^2\notag\\
        &\overset{\eqref{eq:main_olc_clipped_SSTM_technical_3}}{=}& A_{k+1}\left\la \nabla f(x^{k+1}), x^{k+1} - y^{k+1}\right\ra - \frac{1}{2}\|z^k - z^{k+1}\|_2^2 \notag\\
        &&\quad + \alpha_{k+1}\left\la \theta_{k+1}, z^{k} - z^{k+1}\right\ra   + \frac{1}{2}\|z^k - z\|_2^2 - \frac{1}{2}\|z^{k+1}-z\|_2^2\notag\\
        &\overset{\eqref{eq:holder_cor2}}{\le}& A_{k+1}\left(f(x^{k+1}) - f(y^{k+1})\right)   + \frac{A_{k+1}L_{k+1}}{2}\|x^{k+1}-y^{k+1}\|_2^2\notag\\
        &&\quad + \frac{\alpha_{k+1}\varepsilon}{4} - \frac{1}{2}\|z^k - z^{k+1}\|_2^2  + \alpha_{k+1}\left\la \theta_{k+1}, z^{k} - z^{k+1}\right\ra\notag\\
        &&\quad + \frac{1}{2}\|z^k - z\|_2^2 - \frac{1}{2}\|z^{k+1}-z\|_2^2\notag\\
        &\overset{\eqref{eq:main_olc_clipped_SSTM_technical_3}}{=}& A_{k+1}\left(f(x^{k+1}) - f(y^{k+1})\right)  + \frac{1}{2}\left(\frac{\alpha_{k+1}^2L_{k+1}}{A_{k+1}} - 1\right)\|z^k - z^{k+1}\|_2^2\notag\\
        &&\quad + \alpha_{k+1}\left\la \theta_{k+1}, z^{k} - z^{k+1}\right\ra   + \frac{1}{2}\|z^k - z\|_2^2 - \frac{1}{2}\|z^{k+1}-z\|_2^2 + \frac{\alpha_{k+1}\varepsilon}{4},\notag
    \end{eqnarray}
    where in the third inequality we used $x^{k+1}, y^{k+1} \in B_{3R_0}(x^*)$ and \eqref{eq:holder_cor2} with $\delta = \frac{\alpha_{k+1}}{2A_{k+1}}\varepsilon$ and $L(\delta,\nu) = L_{k+1} = \left(\frac{2A_{k+1}}{\varepsilon\alpha_{k+1}}\right)^{\frac{1-\nu}{1+\nu}}M_\nu^{\frac{2}{1+\nu}}$. Since $A_{k+1} \ge aL_{k+1}\alpha_{k+1}^2$ (Lemma~\ref{lem:alpha_k_A_K_lemma}) and $a\ge 1$ we can continue our derivations as follows:
    \begin{eqnarray}
        &\alpha_{k+1}&\left\la\tnabla f(x^{k+1},\Bxi^k), z^k - z \right\ra \notag \\
        &\le& A_{k+1}\left(f(x^{k+1}) - f(y^{k+1})\right) + \alpha_{k+1}\left\la \theta_{k+1}, z^{k} - z^{k+1}\right\ra\notag\\
        &&\quad + \frac{1}{2}\|z^k - z\|_2^2 - \frac{1}{2}\|z^{k+1}-z\|_2^2 + \frac{\alpha_{k+1}\varepsilon}{4}. \label{eq:main_olc_clipped_SSTM_technical_4}
    \end{eqnarray}
    Next, due to the convexity of $f$, we have
    \begin{eqnarray}
        \left\la\tnabla f(x^{k+1},\Bxi^k), y^k - x^{k+1}\right\ra &\overset{\eqref{eq:theta_k+1_def_clipped_SSTM}}{=}& \left\la\nabla f(x^{k+1}), y^k - x^{k+1}\right\ra + \left\la \theta_{k+1}, y^k - x^{k+1}\right\ra\notag\\
        &\le& f(y^k) - f(x^{k+1}) + \left\la \theta_{k+1}, y^k - x^{k+1}\right\ra.\label{eq:main_olc_clipped_SSTM_technical_5}
    \end{eqnarray}
    By definition of $x^{k+1}$ we have $x^{k+1} = \frac{A_k y^k + \alpha_{k+1} z^k}{A_{k+1}}$ implying
    \begin{equation}
        \alpha_{k+1}\left(x^{k+1} - z^k\right) = A_k\left(y^k - x^{k+1}\right)\label{eq:main_olc_clipped_SSTM_technical_6}
    \end{equation}
    since $A_{k+1} = A_k + \alpha_{k+1}$. Putting all together, we derive that
    \begin{eqnarray*}
        &\alpha_{k+1}&\left\la\tnabla f(x^{k+1},\Bxi^k), x^{k+1} - z \right\ra \notag \\
        &=& \alpha_{k+1}\left\la\tnabla f(x^{k+1},\Bxi^k), x^{k+1} - z^k \right\ra  + \alpha_{k+1}\left\la\tnabla f(x^{k+1},\Bxi^k), z^k - z \right\ra\\
        &\overset{\eqref{eq:main_olc_clipped_SSTM_technical_6}}{=}& A_{k}\left\la\tnabla f(x^{k+1},\Bxi^k), y^{k} - x^{k+1} \right\ra + \alpha_{k+1}\left\la\tnabla f(x^{k+1},\Bxi^k), z^k - z \right\ra\\
        &\overset{\eqref{eq:main_olc_clipped_SSTM_technical_5},\eqref{eq:main_olc_clipped_SSTM_technical_4}}{\le}& A_k\left(f(y^k)-f(x^{k+1})\right) + A_k\left\la \theta_{k+1}, y^k - x^{k+1}\right\ra\\
        &&\quad + A_{k+1}\left(f(x^{k+1}) - f(y^{k+1})\right)\notag  + \alpha_{k+1}\left\la \theta_{k+1}, z^{k} - z^{k+1}\right\ra \\
        &&\quad + \frac{1}{2}\|z^k - z\|_2^2 - \frac{1}{2}\|z^{k+1}-z\|_2^2 + \frac{\alpha_{k+1}\varepsilon}{4}\\
        &\overset{\eqref{eq:main_olc_clipped_SSTM_technical_6}}{=}& A_kf(y^k) - A_{k+1}f(y^{k+1}) + \alpha_{k+1}\left\la \theta_{k+1}, x^{k+1}-z^k\right\ra\\
        &&\quad + \alpha_{k+1}f(x^{k+1}) + \alpha_{k+1}\left\la \theta_{k+1}, z^{k} - z^{k+1}\right\ra\\
        &&\quad + \frac{1}{2}\|z^k - z\|_2^2 - \frac{1}{2}\|z^{k+1}-z\|_2^2 + \frac{\alpha_{k+1}\varepsilon}{4}\\
        &=& A_kf(y^k) - A_{k+1}f(y^{k+1}) + \alpha_{k+1}f(x^{k+1})\\
        &&\quad + \alpha_{k+1}\left\la \theta_{k+1}, x^{k+1} - z^{k+1}\right\ra\\
        &&\quad + \frac{1}{2}\|z^k - z\|_2^2 - \frac{1}{2}\|z^{k+1}-z\|_2^2 + \frac{\alpha_{k+1}\varepsilon}{4}.
    \end{eqnarray*}
    Rearranging the terms, we get
    \begin{eqnarray*}
        &A_{k+1}&f(y^{k+1}) - A_kf(y^k) \\
        &\le& \alpha_{k+1}\left(f(x^{k+1}) + \left\la\tnabla f(x^{k+1},\Bxi^k), z-x^{k+1}\right\ra\right) + \frac{1}{2}\|z^k - z\|_2^2\\
        &&\quad - \frac{1}{2}\|z^{k+1} - z\|_2^2 + \alpha_{k+1}\left\la \theta_{k+1}, x^{k+1} - z^{k+1}\right\ra + \frac{\alpha_{k+1}\varepsilon}{4}\\
        &\overset{\eqref{eq:theta_k+1_def_clipped_SSTM}}{=}& \alpha_{k+1}\left(f(x^{k+1}) + \left\la\nabla f(x^{k+1}), z-x^{k+1}\right\ra\right)\\
        &&\quad + \alpha_{k+1}\left\la\theta_{k+1}, z-x^{k+1}\right\ra + \frac{1}{2}\|z^k - z\|_2^2 - \frac{1}{2}\|z^{k+1} - z\|_2^2\\
        &&\quad + \alpha_{k+1}\left\la \theta_{k+1}, x^{k+1} - z^{k+1}\right\ra + \frac{\alpha_{k+1}\varepsilon}{4}\\
        &\le& \alpha_{k+1}f(z) + \frac{1}{2}\|z^k - z\|_2^2 - \frac{1}{2}\|z^{k+1} - z\|_2^2   + \alpha_{k+1}\left\la \theta_{k+1}, z - z^{k+1}\right\ra + \frac{\alpha_{k+1}\varepsilon}{4}
    \end{eqnarray*}
    where in the last inequality, we use the convexity of $f$. Taking into account $A_0 = \alpha_0 = 0$ and $A_{N} = \sum_{k=0}^{N-1}\alpha_{k+1}$ we sum up these inequalities for \revision{$k=0,1,\ldots,N-1$} and get
    \begin{eqnarray*}
        &A_N&f(y^N) \\
        &\le& A_N f(z) + \frac{1}{2}\|z^0 - z\|_2^2 - \frac{1}{2}\|z^{N} - z\|_2^2 + \sum\limits_{k=0}^{N-1}\alpha_{k+1}\left\la \theta_{k+1}, z - z^{k+1}\right\ra+ \frac{A_{N}\varepsilon}{4}\\
        &=& A_N f(z) + \frac{1}{2}\|z^0 - z\|_2^2 - \frac{1}{2}\|z^{N} - z\|_2^2 + \sum\limits_{k=0}^{N-1}\alpha_{k+1}\left\la \theta_{k+1}, z - z^{k}\right\ra \\
        &&\quad + \sum\limits_{k=0}^{N-1}\alpha_{k+1}^2\left\la\theta_{k+1},\tnabla f(x^{k+1},\Bxi^k) \right\ra + \frac{A_{N}\varepsilon}{4}\\
        &\overset{\eqref{eq:theta_k+1_def_clipped_SSTM}}{=}& A_N f(z) + \frac{1}{2}\|z^0 - z\|_2^2 - \frac{1}{2}\|z^{N} - z\|_2^2 + \sum\limits_{k=0}^{N-1}\alpha_{k+1}\left\la \theta_{k+1}, z - z^{k}\right\ra\\
        &&\quad + \sum\limits_{k=0}^{N-1}\alpha_{k+1}^2\left\|\theta_{k+1}\right\|_2^2 + \sum\limits_{k=0}^{N-1}\alpha_{k+1}^2\left\la\theta_{k+1},\nabla f(x^{k+1})\right\ra + \frac{A_{N}\varepsilon}{4}
    \end{eqnarray*}
    that concludes the proof. \qed
\end{proof}

From Lemma~\ref{lem:alpha_k_A_K_lemma} we know that 
$
    A_N \sim \frac{N^{\frac{1+3\nu}{1+\nu}}\varepsilon^{\frac{1-\nu}{1+\nu}}}{M_\nu^{\frac{2}{1+\nu}}}.
$ 
Therefore, in view of Lemma~\ref{lem:main_opt_lemma_clipped_SSTM} (inequality \eqref{eq:main_opt_lemma_clipped_SSTM} with $z = x^*$), to derive the desired complexity bound from Theorem~\ref{thm:main_result_clipped_SSTM_main} it is sufficient to show that
\begin{equation*}
    \sum\limits_{k=0}^{N-1} \left(\alpha_{k+1}\left\la \theta_{k+1}, z - z^{k}\right\ra + 
    \alpha_{k+1}^2\left\|\theta_{k+1}\right\|_2^2 + 
    \alpha_{k+1}^2\left\la\theta_{k+1},\nabla f(x^{k+1})\right\ra\right) + \frac{A_{N}\varepsilon}{4} \lesssim R_0^2
\end{equation*}
with probability at least $1-\beta$. One possible way to achieve this goal is to apply some concentration inequality to these three sums. Since we use clipped stochastic gradients, under a proper choice of the clipping parameter, random vector $\theta_{k+1} = \tnabla f(x^{k+1},\Bxi^k) - \nabla f(x^{k+1})$ is bounded in $\ell_2$-norm by $2\lambda_{k+1}$ with high probability as well. Taking into account the assumption on the stochastic gradients (see \eqref{eq:bounded_variance_clipped_SSTM}), it is natural to apply Bernstein's inequality (see Lemma~\ref{lem:Bernstein_ineq}). Despite the seeming simplicity, this part of the proof is the trickiest one.

First of all, it is useful to derive tight enough upper bounds for bias, variance, and distortion of $\tnabla f(x^{k+1},\Bxi^k)$ -- this is the second step of the whole proof. Fortunately, Lemma~F.5 from \cite{gorbunov2020clipped_sstm} does exactly what we need in our proof and holds without any changes.
\begin{lemma}[Lemma~F.5 from \cite{gorbunov2020clipped_sstm}.]\label{lem:main_stoch_lemma_clipped_SSTM}
    For all $k\ge 0$, the following inequality holds:
    \begin{equation}
        \left\|\tnabla f(x^{k+1},\Bxi^{k}) - \EE_{\Bxi^k}\left[\tnabla f(x^{k+1},\Bxi^{k})\right]\right\|_2 \le 2\lambda_{k+1}.\label{eq:magnitude_bound_clipped_SSTM}
    \end{equation}
    Moreover, if the stochastic gradient satisfies \eqref{eq:bounded_variance_clipped_SSTM} on $Q = B_{3R_0}(x^*)$ and \\ $\|\nabla f(x^{k+1})\|_2 \le \frac{\lambda_{k+1}}{2}$ for some $k\ge 0$, then for this $k$ we have:
    \begin{eqnarray}
        \left\|\EE_{\Bxi^k}\left[\tnabla f(x^{k+1},\Bxi^{k})\right] - \nabla f(x^{k+1})\right\|_2 &\le& \frac{4\sigma^2}{m_k\lambda_{k+1}},\label{eq:bias_bound_clipped_SSTM}\\
        \EE_{\Bxi^k}\left[\left\|\tnabla f(x^{k+1},\Bxi^{k}) - \nabla f(x^{k+1})\right\|_2^2\right] &\le& \frac{18\sigma^2}{m_k},\label{eq:distortion_bound_clipped_SSTM}\\
        \EE_{\Bxi^k}\left[\left\|\tnabla f(x^{k+1},\Bxi^{k}) - \EE_{\Bxi^k}\left[\tnabla f(x^{k+1},\Bxi^{k})\right]\right\|_2^2\right] &\le& \frac{18\sigma^2}{m_k}. \label{eq:variance_bound_clipped_SSTM}
    \end{eqnarray}
\end{lemma}

\subsubsection{Proof of the Main Result}
The final,  third step of the proof consists of providing explicit formulas and bounds for the parameters of the method and derivation of the desired result using induction and Bernstein's inequality. Below, we provide the complete statement of Theorem~\ref{thm:main_result_clipped_SSTM_main}. 
\begin{theorem}\label{thm:main_result_clipped_SSTM}
    Assume that function $f$ is convex,  achieves minimum value at some\footnote{ Our proofs are valid for any solution $x^*$ and, for example, one can take $x^*$ as the closest solution to the starting point $x^0$. } $x^*$, its stochastic gradient and its gradient satisfy \eqref{eq:bounded_variance_clipped_SSTM} and \eqref{eq:holder_def} respectively with $\sigma > 0$, $\nu \in [0,1]$, $M_\nu > 0$ on $Q = B_{3R_0}(x^*)$, where $R_0 \ge \|x^0 - x^*\|_2$. 
    \revision{Let $\beta \in (0,1)$ and $N\ge 1$ be arbitrary such that 
    \begin{equation}
        \ln\frac{4N}{\beta} \ge 2 \label{eq:beta_N_condition_clipped_SSTM}
    \end{equation}
    and let the parameters of \algname{clipped-SSTM} satisfy\footnote{\revision{The choice of the parameters (in this and the following results) is dictated by the need to estimate and control the stochastic error in the proofs. If some of the parameters (such as $\nu, R_0, M_\nu, \sigma$) are unknown, one can directly tune parameters $\alpha, a, m_k$. To satisfy \eqref{eq:epsilon_clipped_SSTM} and \eqref{eq:epsilon_clipped_SSTM_2} it sufficient to choose sufficiently large $a$ (or, alternatively, sufficiently small $\varepsilon$).}}
    \begin{equation}
        \alpha = \frac{
        \varepsilon^{\frac{1-\nu}{1+\nu}}}{2aM_\nu^{\frac{2}{1+\nu}}},\quad m_k = \max\left\{1, \frac{20736N\sigma^2\alpha_{k+1}^2\ln \frac{4N}{\beta}}{C^2R_0^2}\right\},\label{eq:bathces_clipped_SSTM}
    \end{equation}
    \begin{equation}
        B = \frac{CR_0}{16\ln\frac{4N}{\beta}},\quad a \ge 16384\ln^2\frac{4N}{\beta}, \label{eq:B_a_parameters_clipped_SSTM}
    \end{equation}
    \begin{equation}
        \varepsilon^{\frac{1-\nu}{1+\nu}} \le \frac{aCM_\nu^{\frac{1-\nu}{1+\nu}}R_0^{1-\nu}}{16\ln\frac{4N}{\beta}},\quad \varepsilon \le \frac{2^{\frac{1+\nu}{2}}a^{\frac{1+\nu}{2}}C^{1+\nu}R_0^{1+\nu}M_\nu}{100^{\frac{1+3\nu}{2}}}, \label{eq:epsilon_clipped_SSTM}
    \end{equation}
    \begin{eqnarray}
         \varepsilon^{\frac{1-\nu}{1+3\nu}} &\le& \min\Bigg\{\frac{a^{\frac{2+3\nu-\nu^2}{2(1+3\nu)}}}{2^{4+4\nu+\frac{3+8\nu-5\nu^2-6\nu^3}{(1+\nu)(1+3\nu)}}\ln\frac{4N}{\beta}},\notag\\
         &&\quad\quad\quad\frac{a^{\frac{(1+\nu)^2}{1+3\nu}}}{2^{5+8\nu+\frac{2+9\nu+7\nu^2-3\nu^3 + \nu^4}{(1+\nu)(1+3\nu)}}\ln^{1+\nu}\frac{4N}{\beta}}\Bigg\}C^{\frac{1-\nu^2}{1+3\nu}}R_0^{\frac{1-\nu^2}{1+3\nu}}M_\nu^{\frac{1-\nu}{1+3\nu}},\label{eq:epsilon_clipped_SSTM_2} 
    \end{eqnarray}
    \begin{eqnarray}
        \varepsilon &\leq& \frac{2^{\frac{1+\nu}{2}}a^{\frac{1+\nu}{2}}C^{1+\nu}R_0^{1+\nu}M_\nu}{N^{\frac{1+3\nu}{2}}} + \frac{1}{N^{\frac{1+3\nu}{2}}}.\label{eq:epsilon_clipped_SSTM_3}
    \end{eqnarray}
    Then, after $N$ iterations of \algname{clipped-SSTM}, with probability at least $1-\beta$, it holds that
    \begin{equation}
        f(y^N) - f(x^*) \le \frac{4aC^2R_0^2M_\nu^{\frac{2}{1+\nu}}}{N^{\frac{1+3\nu}{1+\nu}}\varepsilon^{\frac{1-\nu}{1+\nu}}}, \label{eq:main_result_clipped_SSTM}
    \end{equation}
    where
    \begin{equation}
        C = \sqrt{7}. \label{eq:C_definition_clipped_SSTM}
    \end{equation}
    }

    In other words, if we choose $a = 16384\ln^2\frac{4N}{\beta}$, then the method achieves $f(y^N) - f(x^*) \le \varepsilon$ with probability at least $1-\beta$ \revision{
    after $\cO\left(D\ln^{\frac{2(1+\nu)}{1+3\nu}}\frac{D}{\beta}\right)$ iterations with $D = \frac{M_\nu^{\frac{2}{1+3\nu}}R_0^{\frac{2(1+\nu)}{1+3\nu}}}{\varepsilon^{\frac{2}{1+3\nu}}}$ and requires
    \begin{equation}
        \cO\left(\max\left\{D\ln^{\frac{2(1+\nu)}{1+3\nu}}\frac{D}{\beta},\frac{\sigma^2R_0^2}{\varepsilon^2}\ln\frac{D}{\beta}\right\}\right)\text{ oracle calls.} \label{eq:clipped_SSTM_oracle_complexity}
    \end{equation}}    
\end{theorem}
\begin{proof}
    \revision{The proof of this result (and the following ones) is induction-based: we will show by induction that the iterates stay in some bounded ball around $x^*$ with high probability. This will allow us to apply Bernstein-type concentration inequality to estimate the stochastic sums appearing in the upper bounds.}

    First of all, we notice that for each $k\ge 0$ iterates $x^{k+1}, z^k, y^k$ lie in the ball $B_{\widetilde{R}_k}(x^*)$, where $R_k = \|z^k - x^*\|_2$, $\widetilde{R}_0 = R_0$, $\widetilde{R}_{k+1} = \max\{\widetilde{R}_k, R_{k+1}\}$. We prove it using induction. Since $y^0 = z^0 = x^0$, $\widetilde{R}_0 = R_0 \ge \|z^0 - x^*\|_2$ and $x^1 = \frac{A_0y^0 + \alpha_1z^0}{A_1} = z^0$ we have that $x^{1}, z^0, y^0 \in B_{\widetilde{R}_0}(x^*)$. Next, assume that $x^{l}, z^{l-1}, y^{l-1} \in B_{\widetilde{R}_{l-1}}(x^*)$ for some $l\ge 1$. By definitions of $R_l$ and $\widetilde{R}_l$ we have that $z^l \in B_{R_{l}}(x^*)\subseteq B_{\widetilde{R}_{l}}(x^*)$. Since $y^l$ is a convex combination of $y^{l-1}\in B_{\widetilde{R}_{l-1}}(x^*)\subseteq B_{\widetilde{R}_{l}}(x^*)$\revision{, and} $z^l\in B_{\widetilde{R}_{l}}(x^*)$\revision{, and} $B_{\widetilde{R}_{l}}(x^*)$ is a convex set we conclude that $y^l \in B_{\widetilde{R}_{l}}(x^*)$. Finally, since $x^{l+1}$ is a convex combination of $y^l$ and $z^l$ we have that $x^{l+1}$ lies in $B_{\widetilde{R}_{l}}(x^*)$ as well.

    Next, our goal is to prove via induction that for all $k=0,1,\ldots, N$ \revision{we have $\PP\{E_k\} = 1 - \frac{k\beta}{N}$ for probability event $E_k$ defined as follows:}

    \revision{\begin{center}
        Event $E_k$:
    \end{center}
    \noindent\fbox{%
    \parbox{\textwidth}{%
    Inequalities
    \begin{eqnarray}
         R_t^2 &\le& R_0^2 + 2\sum\limits_{l=0}^{t-1}\alpha_{l+1}\left\la\theta_{l+1}, x^* - z^l\right\ra + 2\sum\limits_{l=0}^{t-1}\alpha_{l+1}^2\left\la\theta_{l+1}, \nabla f(x^{l+1})\right\ra \notag\\
         &&\quad + 2\sum\limits_{l=0}^{t-1}\alpha_{k+1}^2\|\theta_{l+1}\|_2^2 + \frac{A_{N}\varepsilon}{2} \le C^2R_0^2 \label{eq:main_thm_clipped_SSTM_technical_2}
    \end{eqnarray}
    hold for $t=0,1,\ldots,k$ simultaneously where $C$ is defined in \eqref{eq:C_definition_clipped_SSTM}.
    }
    }}
    
    For $t = 0$ inequality \eqref{eq:main_thm_clipped_SSTM_technical_2} holds with probability $1$ since $C \ge 1$, hence $\PP\{E_0\} = 1$. Next, assume that for some $k = T-1 \le N-1$ we have $\PP\{E_k\} = \PP\{E_{T-1}\} \ge 1 - \frac{(T-1)\beta}{N}$. Let us prove that $\PP\{E_{T}\} \ge 1 - \frac{T\beta}{N}$. First of all, since $R_{T-1}$ implies $R_t \le CR_0$ for all $t = 0,1,\ldots,T-1$ we have that $\widetilde{R}_{T-1} \le CR_0$, and, as a consequence, $z^{T-1} \in B_{CR_0}(x^*)$. Therefore, probability event $E_{T-1}$ implies
    \begin{eqnarray*}
         & & \hspace{-2em}\|z^{T} - x^*\|_2 \\
         &=& \|z^{T-1} - x^* - \alpha_{T}\tnabla f(x^{T},\Bxi^{T-1})\|_2 \le \|z^{T-1} - x^*\|_2 + \alpha_T \|\tnabla f(x^{T},\Bxi^{T-1})\|_2\\
         &\le& CR_0 + \alpha_T\lambda_T = \left(1 + \frac{1}{16\ln\frac{4N}{\beta}}\right)CR_0 \overset{\eqref{eq:beta_N_condition_clipped_SSTM},\eqref{eq:C_definition_clipped_SSTM}}{\le} \left(1 + \frac{1}{32}\right)\sqrt{7} R_0 \le 3R_0,
    \end{eqnarray*}
    hence $\widetilde{R}_T \le 3R_0$. Then, one can apply Lemma~\ref{lem:main_opt_lemma_clipped_SSTM} and get that probability event $E_{T-1}$ implies
    \begin{eqnarray}
        A_t\left(f(y^t) - f(x^*)\right) &\le& \frac{1}{2}\|z^0 - x^*\|_2^2 - \frac{1}{2}\|z^{t} - x^*\|_2^2 + \sum\limits_{k=0}^{t-1}\alpha_{k+1}\left\la \theta_{k+1}, x^* - z^{k}\right\ra\notag\\
        &&\hspace{-3em} + \sum\limits_{k=0}^{t-1}\alpha_{k+1}^2\left\|\theta_{k+1}\right\|_2^2 + \sum\limits_{k=0}^{t-1}\alpha_{k+1}^2\left\la\theta_{k+1},\nabla f(x^{k+1})\right\ra + \frac{A_{t}\varepsilon}{4},\label{eq:main_thm_clipped_SSTM_technical_0}\\
        \theta_{k+1} &\eqdef& \tnabla f(x^{k+1},\Bxi^k) - \nabla f(x^{k+1})\label{eq:main_thm_clipped_SSTM_technical_theta}
    \end{eqnarray}
    for all \revision{$t = 0,1,\ldots,T$}. Taking into account that $f(y^t) - f(x^*) \ge 0$ for all $y^t$ we derive that probability event $E_{T-1}$ implies
    \begin{eqnarray}
         R_t^2 &\le& R_0^2 + 2\sum\limits_{l=0}^{t-1}\alpha_{l+1}\left\la\theta_{l+1}, x^* - z^l\right\ra + 2\sum\limits_{l=0}^{t-1}\alpha_{l+1}^2\left\la\theta_{l+1}, \nabla f(x^{l+1})\right\ra\notag\\
         &&\quad + 2\sum\limits_{l=0}^{t-1}\alpha_{l+1}^2\|\theta_{l+1}\|_2^2 + \frac{A_{t}\varepsilon}{2}. \label{eq:main_thm_clipped_SSTM_technical_1}
    \end{eqnarray}
    for all $t = 0,1, \ldots, T$.
    
    The rest of the proof is based on the refined analysis of inequality \eqref{eq:main_thm_clipped_SSTM_technical_1}. First of all, when $\nu = 0$ \revision{from \eqref{eq:holder_gradient_bound}} for all $t\ge 0$ we have
    \begin{eqnarray}
        \left\|\nabla f(x^{t+1})\right\|_2 &\le& M_0 \revision{\overset{\eqref{eq:B_a_parameters_clipped_SSTM}}{=}} \frac{16M_0B\ln\frac{4N}{\beta}}{CR_0} \revision{\overset{\eqref{eq:epsilon_clipped_SSTM}}{\le}} \frac{aM_0^2B}{\varepsilon} = \frac{B}{2\alpha_{t+1}} = \frac{\lambda_{t+1}}{2}. \notag
    \end{eqnarray}
    Next, we prove that $\|\nabla f(x^{t+1})\|_2 \le \frac{\lambda_{t+1}}{2}$ when $\nu > 0$. For $t = 0$ we have
    \begin{eqnarray}
        \|\nabla f(x^1)\|_2 &=& \|\nabla f(z^0)\|_2 \overset{\eqref{eq:holder_def}}{\le} M_\nu\|z^0 - x^*\|_2^{\nu} \le M_\nu R_0^\nu \revision{\overset{\eqref{eq:epsilon_clipped_SSTM}}{\leq} \frac{2^{\frac{1-\nu}{1+\nu}}aC R_0 M_\nu^{\frac{2}{1+\nu}}}{32\varepsilon^{\frac{1-\nu}{1+\nu}}\ln\frac{4N}{\beta}}}\notag \\
        &\revision{\overset{\eqref{eq:B_a_parameters_clipped_SSTM}, \eqref{eq:epsilon_clipped_SSTM}}{\le}}& \frac{B}{2\alpha_1} = \frac{\lambda_1}{2}.\notag
    \end{eqnarray}
    For $0< t \le T-1$ probability event $E_{T-1}$ implies
    \begin{eqnarray}
        &&\hspace{-7em}\|\nabla f(x^{t+1})\|_2\notag\\ &\le& \|\nabla f(x^{t+1}) - \nabla f(y^t)\|_2 + \|\nabla f(y^{t})\|_2    \notag\\
        &\overset{\eqref{eq:holder_def}, \revision{\text{ Lemma \ref{lem:gradient_bound}}}}{\le}& M_\nu \|x^{t+1}-y^t\|_2^\nu + \left(\frac{1+\nu}{\nu}\right)^{\frac{\nu}{1+\nu}}M_\nu^{\frac{1}{1+\nu}} \left(f(y^t) - f(x^*)\right)^{\frac{\nu}{1+\nu}} \notag\\
        &\revision{\overset{\eqref{eq:main_opt_lemma_clipped_SSTM}, \eqref{eq:main_olc_clipped_SSTM_technical_6},\eqref{eq:main_thm_clipped_SSTM_technical_2}}{\le}}& M_\nu \left(\frac{\alpha_{t+1}}{A_t}\right)^\nu\|x^{t+1}-z^t\|_2^\nu + \left(\frac{1+\nu}{\nu}\right)^{\frac{\nu}{1+\nu}}M_\nu^{\frac{1}{1+\nu}} \left(\frac{C^2R_0^2}{2A_t}\right)^{\frac{\nu}{1+\nu}}\notag\\
        &=& \frac{\lambda_{t+1}}{2}\cdot\underbrace{\frac{2M_\nu}{\lambda_{t+1}} \left(\frac{\alpha_{t+1}}{A_t}\right)^\nu\|x^{t+1}-z^t\|_2^\nu}_{D_1} \notag\\
        &&\quad + \frac{\lambda_{t+1}}{2}\cdot \underbrace{\left(\frac{1+\nu}{\nu}\right)^{\frac{\nu}{1+\nu}}\frac{2M_\nu^{\frac{1}{1+\nu}}}{\lambda_{t+1}} \left(\frac{C^2R_0^2}{2A_t}\right)^{\frac{\nu}{1+\nu}}}_{D_2}.\notag
    \end{eqnarray}
    Next, we show that $D_1 + D_2 \le 1$. Using the definition of $\lambda_{t+1}$\revision{, the definition of $\alpha_{t+1} = \alpha (t+1)^{\frac{2\nu}{1+\nu}}$}, triangle inequality $\|x^{t+1} - z^t\|_2 \le \|x^{t+1} - x^*\|_2 + \|z^t - x^*\|_2 \le 2CR_0$, and lower bound \eqref{eq:A_k_lower_bound} for $A_t$ (see Lemma~\ref{lem:alpha_k_A_K_lemma}) we derive
    \begin{eqnarray*}
        D_1 &=& \frac{\revision{2^{\nu+5}}M_\nu\alpha_{t+1}^{1+\nu}\ln\frac{4N}{\beta}}{C^{1-\nu}R_0^{1-\nu}A_t^\nu} \revision{\overset{\eqref{eq:bathces_clipped_SSTM}}{=}} \frac{\revision{2^{\nu+5}}M_\nu (t+1)^{2\nu}(\nicefrac{\varepsilon}{2})^{1-\nu}\ln\frac{4N}{\beta}}{2^{2\nu}a^{1+\nu}C^{1-\nu}R_0^{1-\nu}M_\nu^2A_t^\nu}\\
        &\overset{\eqref{eq:A_k_lower_bound}}{\le}& \frac{\revision{2^{4}}(t+1)^{2\nu}\varepsilon^{1-\nu}\ln\frac{4N}{\beta}}{a^{1+\nu}C^{1-\nu}R_0^{1-\nu}M_\nu}\cdot \frac{2^{\frac{(1+3\nu)\nu}{1+\nu}}a^\nu M_\nu^{\frac{2\nu}{1+\nu}}}{t^{\frac{(1+3\nu)\nu}{1+\nu}}(\nicefrac{\varepsilon}{2})^{\frac{\nu(1-\nu)}{1+\nu}}}\\
        &=& \frac{(t+1)^{2\nu}}{t^{\frac{\nu(1+3\nu)}{1+\nu}}}\cdot \frac{\revision{2^{4+2\nu}}\varepsilon^{\frac{1-\nu}{1+\nu}}\ln\frac{4N}{\beta}}{aM_\nu^{\frac{1-\nu}{1+\nu}}C^{1-\nu}R_0^{1-\nu}} ~\revision{\overset{\frac{t+1}{t} \leq 2}{\le}}~ \frac{\revision{2^{4+4\nu}}t^{\frac{\nu(1-\nu)}{1+\nu}}\varepsilon^{\frac{1-\nu}{1+\nu}}\ln\frac{4N}{\beta}}{aM_\nu^{\frac{1-\nu}{1+\nu}}C^{1-\nu}R_0^{1-\nu}} \\
        &\revision{\overset{t\leq N-1, \eqref{eq:epsilon_clipped_SSTM_3}}{\le}}& \frac{\revision{2^{4+4\nu}}\varepsilon^{\frac{1-\nu}{1+\nu}}\ln\frac{4N}{\beta}}{aM_\nu^{\frac{1-\nu}{1+\nu}}C^{1-\nu}R_0^{1-\nu}}\cdot \frac{2^{\frac{2\nu(1-\nu)(1+2\nu)}{(1+\nu)(1+3\nu)}}a^{\frac{\nu(1-\nu)}{1+3\nu}}C^{\frac{2\nu(1-\nu)}{1+3\nu}}R_0^{\frac{2\nu(1-\nu)}{1+3\nu}}M_\nu^{\frac{2\nu(1-\nu)}{(1+\nu)(1+3\nu)}}}{\varepsilon^{\frac{2\nu(1-\nu)}{(1+\nu)(1+3\nu)}}}\\
        &=& \frac{\revision{2^{4+4\nu+\frac{2\nu(1-\nu)(1+2\nu)}{(1+\nu)(1+3\nu)}}}\varepsilon^{\frac{1-\nu}{1+3\nu}}\ln\frac{4N}{\beta}}{a^{\frac{(1+\nu)^2}{1+3\nu}}M_\nu^{\frac{1-\nu}{1+3\nu}}C^{\frac{(1-\nu)(1+\nu)}{1+3\nu}}R_0^{\frac{(1-\nu)(1+\nu)}{1+3\nu}}} \overset{\eqref{eq:epsilon_clipped_SSTM_2}}{\le} \frac{1}{2^{\frac{3+6\nu - 7\nu^2-2\nu^3}{(1+\nu)(1+3\nu)}} a^{\frac{\nu}{2}}}.
    \end{eqnarray*}
    Applying the same inequalities and $\left(\frac{1+\nu}{\nu}\right)^{\frac{\nu}{1+\nu}} \le 2$ we estimate $D_2$:
    \begin{eqnarray*}
        D_2 &=& \left(\frac{1+\nu}{\nu}\right)^{\frac{\nu}{1+\nu}} \frac{2^{\revision{5}-\frac{\nu}{1+\nu}}M_\nu^{\frac{1}{1+\nu}}\alpha_{t+1}\ln\frac{4N}{\beta}}{C^{\frac{1-\nu}{1+\nu}}R_0^{\frac{1-\nu}{1+\nu}}A_t^{\frac{\nu}{1+\nu}}} \\&\le& 2\cdot\frac{2^{\revision{5}-\frac{\nu}{1+\nu}}M_\nu^{\frac{1}{1+\nu}}\ln\frac{4N}{\beta}}{C^{\frac{1-\nu}{1+\nu}}R_0^{\frac{1-\nu}{1+\nu}}A_t^{\frac{\nu}{1+\nu}}}\cdot\frac{(t+1)^{\frac{2\nu}{1+\nu}}(\nicefrac{\varepsilon}{2})^{\frac{1-\nu}{1+\nu}}}{2^{\frac{2\nu}{1+\nu}}a M_\nu^{\frac{2}{1+\nu}}} \\
        &\le& \frac{2^{\revision{5}-\frac{\nu}{1+\nu}}\cdot 2^{\frac{2\nu}{1+\nu}}t^{\frac{2\nu}{1+\nu}}\varepsilon^{\frac{1-\nu}{1+\nu}}\ln\frac{4N}{\beta}}{aC^{\frac{1-\nu}{1+\nu}}R_0^{\frac{1-\nu}{1+\nu}}M_\nu^{\frac{1}{1+\nu}}A_t^{\frac{\nu}{1+\nu}}}\\
        &\overset{\eqref{eq:A_k_lower_bound}}{\le}& \frac{2^{\revision{5}+\frac{\nu}{1+\nu}}t^{\frac{2\nu}{1+\nu}}\varepsilon^{\frac{1-\nu}{1+\nu}}\ln\frac{4N}{\beta}}{aC^{\frac{1-\nu}{1+\nu}}R_0^{\frac{1-\nu}{1+\nu}}M_\nu^{\frac{1}{1+\nu}}} \cdot \frac{2^{\frac{\nu(1+3\nu)}{(1+\nu)^2}}a^{\frac{\nu}{1+\nu}}M_\nu^{\frac{2\nu}{(1+\nu)^2}}}{t^{\frac{\nu(1+3\nu)}{(1+\nu)^2}}(\nicefrac{\varepsilon}{2})^{\frac{\nu(1-\nu)}{(1+\nu)^2}}}\\
        &=& \frac{2^{\revision{5}+\frac{3\nu}{1+\nu}}t^{\frac{\nu(1-\nu)}{(1+\nu)^2}}\varepsilon^{\frac{1-\nu}{(1+\nu)^2}}\ln\frac{4N}{\beta}}{a^{\frac{1}{1+\nu}}C^{\frac{1-\nu}{1+\nu}}R_0^{\frac{1-\nu}{1+\nu}}M_\nu^{\frac{1-\nu}{(1+\nu)^2}}}\\
        &\revision{\overset{t\leq N-1, \eqref{eq:epsilon_clipped_SSTM_3}}{\le}}& \frac{2^{\revision{5}+\frac{3\nu}{1+\nu}}\varepsilon^{\frac{1-\nu}{(1+\nu)^2}}\ln\frac{4N}{\beta}}{a^{\frac{1}{1+\nu}}C^{\frac{1-\nu}{1+\nu}}R_0^{\frac{1-\nu}{1+\nu}}M_\nu^{\frac{1-\nu}{(1+\nu)^2}}} \\
        && \hspace{1em}\cdot \frac{2^{\frac{2\nu(1+2\nu)(1-\nu)}{(1+\nu)^2(1+3\nu)}}a^{\frac{\nu(1-\nu)}{(1+\nu)(1+3\nu)}}C^{\frac{2\nu(1-\nu)}{(1+\nu)(1+3\nu)}}R_0^{\frac{2\nu(1-\nu)}{(1+\nu)(1+3\nu)}}M_\nu^{\frac{2\nu(1-\nu)}{(1+\nu)^2(1+3\nu)}}}{\varepsilon^{\frac{2\nu(1-\nu)}{(1+\nu)^2(1+3\nu)}}}\\
        &=& \frac{2^{\revision{5}+\frac{3\nu}{1+\nu}+\frac{2\nu(1+2\nu)(1-\nu)}{(1+\nu)^2(1+3\nu)}}\varepsilon^{\frac{1-\nu}{(1+\nu)(1+3\nu)}}\ln\frac{4N}{\beta}}{a^{\frac{1+\nu}{1+3\nu}}C^{\frac{1-\nu}{1+3\nu}}R_0^{\frac{1-\nu}{1+3\nu}}M_\nu^{\frac{1-\nu}{(1+\nu)(1+3\nu)}}} \overset{\eqref{eq:epsilon_clipped_SSTM_2}}{\le} \frac{1}{2^{\frac{2+5\nu+\nu^3}{(1+\nu)^2(1+3\nu)}}}.
    \end{eqnarray*}
    Combining the upper bounds for $D_1$ and $D_2$ we get
    \begin{eqnarray*}
        D_1 + D_2 &\le& \frac{1}{2^{\frac{3+6\nu - 7\nu^2-2\nu^3}{(1+\nu)(1+3\nu)}} a^{\frac{\nu}{2}}} + \frac{1}{2^{\frac{2+5\nu+\nu^3}{(1+\nu)^2(1+3\nu)}}}.
    \end{eqnarray*}
    Since $\frac{2+5\nu+\nu^3}{(1+\nu)^2(1+3\nu)}$ is a decreasing function of $\nu$ for $\nu\in[0,1]$ we continue as
    \begin{eqnarray*}
        D_1 + D_2 &\le& \frac{1}{2^{\frac{3+6\nu - 7\nu^2-2\nu^3}{(1+\nu)(1+3\nu)}} a^{\frac{\nu}{2}}} + \frac{1}{\sqrt{2}}.
    \end{eqnarray*}
    Next, we use \revision{$a \overset{\eqref{eq:B_a_parameters_clipped_SSTM}}{\ge} 16384\ln^2\frac{4N}{\beta} \overset{\eqref{eq:beta_N_condition_clipped_SSTM}}{\ge} 2^{10}$} and obtain
    \begin{eqnarray*}
        D_1 + D_2 &\le& \frac{1}{2^{\frac{3+11\nu + 13\nu^2 + 13\nu^3}{(1+\nu)(1+3\nu)}}} + \frac{1}{\sqrt{2}}.
    \end{eqnarray*}
    One can numerically verify that $\frac{1}{2^{\frac{3+11\nu + 13\nu^2 + 13\nu^3}{(1+\nu)(1+3\nu)}}} + \frac{1}{\sqrt{2}}$ is smaller than $1$ for $\nu\in[0,1]$. Putting all together, we conclude that probability event $E_{T-1}$ implies
    \begin{equation}
        \|\nabla f(x^{t+1})\|_2 \le \frac{\lambda_{t+1}}{2} \label{eq:main_thm_clipped_SSTM_technical_4}
    \end{equation}
    for all $t = 0,1,\ldots, T-1$. Having inequality \eqref{eq:main_thm_clipped_SSTM_technical_4} in hand, we show in the rest of the proof that \eqref{eq:main_thm_clipped_SSTM_technical_2} holds for $t = T$ with large enough probability. First of all, we introduce new random variables:
    \begin{equation}
        \eta_l = \begin{cases}x^* - z^l,&\text{if } \|x^* - z^l\|_2 \le CR_0,\\ 0,&\text{otherwise,} \end{cases},  \zeta_l = \begin{cases}\nabla f(x^{l+1}),&\text{if } \|\nabla f(x^{l+1})\|_2 \le \frac{B}{2\alpha_{l+1}},\\ 0,&\text{otherwise,} \end{cases} \label{eq:main_thm_clipped_SSTM_technical_4_1}
    \end{equation}
    for $l=0,1,\ldots, T-1$. Note that these random variables are bounded with probability $1$, i.e.\ with probability $1$, we have
    \begin{equation}
        \|\eta_l\|_2 \le CR_0\quad \text{and}\quad \|\zeta_l\|_2 \le \frac{B}{2\alpha_{l+1}}.\label{eq:main_thm_clipped_SSTM_technical_4_2}
    \end{equation}
    Secondly, we use the introduced notation and get that $E_{T-1}$ implies 
    \begin{eqnarray*}
         &&\hspace{-7em}R_T^2 \\ &\overset{\eqref{eq:main_thm_clipped_SSTM_technical_1},\eqref{eq:main_thm_clipped_SSTM_technical_2},\eqref{eq:main_thm_clipped_SSTM_technical_4},\eqref{eq:main_thm_clipped_SSTM_technical_4_1}}{\le}& R_0^2 + 2\sum\limits_{l=0}^{T-1}\alpha_{l+1}\left\la\theta_{l+1}, \eta_l\right\ra + 2\sum\limits_{l=0}^{T-1}\alpha_{l+1}^2\|\theta_{l+1}\|_2^2\notag\\
         &&\quad + 2\sum\limits_{l=0}^{T-1}\alpha_{l+1}^2\left\la\theta_{l+1}, \zeta_l\right\ra + \frac{A_{N}\varepsilon}{2}\notag\\
        &=& R_0^2 + \sum\limits_{l=0}^{T-1}\alpha_{l+1}\left\la\theta_{l+1}, 2\eta_l + 2\alpha_{l+1}\zeta_l\right\ra + 2\sum\limits_{l=0}^{T-1}\alpha_{l+1}^2\|\theta_{l+1}\|_2^2 + \frac{A_{N}\varepsilon}{2}.
    \end{eqnarray*}
    Finally, we do some preliminaries in order to apply Bernstein's inequality (see Lemma~\ref{lem:Bernstein_ineq}) and obtain that $E_{T-1}$ implies
    \begin{eqnarray}
        R_T^2 &\overset{\eqref{eq:squared_norm_sum}}{\le}& R_0^2 + \underbrace{\sum\limits_{l=0}^{T-1}\alpha_{l+1}\left\la\theta_{l+1}^u, 2\eta_l + 2\alpha_{l+1}\zeta_l\right\ra}_{\circledOne}+ \underbrace{\sum\limits_{l=0}^{T-1}\alpha_{l+1}\left\la\theta_{l+1}^b, 2\eta_l + 2\alpha_{l+1}\zeta_l\right\ra}_{\circledTwo}\notag\\
        &&\quad + \underbrace{\sum\limits_{l=0}^{T-1}4\alpha_{l+1}^2\left(\|\theta_{l+1}^u\|_2^2 - \EE_{\Bxi^l}\left[\|\theta_{l+1}^u\|_2^2\right]\right)}_{\circledThree} + \underbrace{\sum\limits_{l=0}^{T-1}4\alpha_{l+1}^2\EE_{\Bxi^l}\left[\|\theta_{l+1}^u\|_2^2\right]}_{\circledFour}\notag\\
        &&\quad + \underbrace{\sum\limits_{l=0}^{T-1}4\alpha_{l+1}^2\|\theta_{l+1}^b\|_2^2}_{\circledFive} + \frac{A_{N}\varepsilon}{2}\label{eq:main_thm_clipped_SSTM_technical_5}
    \end{eqnarray}
    where we introduce new notations:
    \begin{equation}
        \theta_{l+1}^u \eqdef \tnabla f(x^{l+1},\Bxi^l) - \EE_{\Bxi^l}\left[\tnabla f(x^{l+1},\Bxi^l)\right],  \theta_{l+1}^b \eqdef \EE_{\Bxi^l}\left[\tnabla f(x^{l+1},\Bxi^l)\right] - \nabla f(x^{l+1}),\label{eq:main_thm_clipped_SSTM_technical_6}
    \end{equation}
    \begin{equation}
        \theta_{l+1} \overset{\eqref{eq:theta_k+1_def_clipped_SSTM}}{=} \theta_{l+1}^u +\theta_{l+1}^b.\notag
    \end{equation}
    It remains to provide tight upper bounds for $\circledOne$, $\circledTwo$, $\circledThree$, $\circledFour$ and $\circledFive$, i.e.\ in the remaining part of the proof we show that $\circledOne+\circledTwo+\circledThree+\circledFour+\circledFive \le \delta C^2R_0^2$ for some $\delta < 1$.
    
    \textbf{Upper bound for $\circledOne$.} First of all, since $\EE_{\Bxi^l}[\theta_{l+1}^u] = 0$ \revision{and random variables $\eta_l,\zeta_l$ are independent from $\Bxi^l$ (which is a collection of i.i.d. samples)} summands in $\circledOne$ are conditionally unbiased:
    \begin{equation*}
        \EE_{\Bxi^l}\left[\alpha_{l+1}\left\la\theta_{l+1}^u, 2\eta_l + 2\alpha_{l+1}\zeta_l\right\ra\right] = 0.
    \end{equation*}
    Secondly, these summands are bounded with probability $1$:
    \begin{eqnarray*}
        &&\hspace{-5em}\left|\alpha_{l+1}\left\la\theta_{l+1}^u, 2\eta_l + 2\alpha_{l+1}\zeta_l\right\ra\right|\\
        &\le& \alpha_{l+1}\|\theta_{l+1}^u\|_2\left\|2\eta_l + 2\alpha_{l+1}\zeta_l\right\|_2\\
        &\overset{\eqref{eq:magnitude_bound_clipped_SSTM},\eqref{eq:main_thm_clipped_SSTM_technical_4_2}}{\le}& 2\alpha_{l+1}\lambda_{l+1}\left(2CR_0 + B\right) = 2B(2CR_0+B)\\
        &=& \left(1 + \frac{1}{32\ln\frac{4N}{\beta}}\right)\frac{C^2R_0^2}{4\ln\frac{4N}{\beta}} \overset{\eqref{eq:beta_N_condition_clipped_SSTM}}{\le} \left(1 + \frac{1}{64}\right)\frac{C^2R_0^2}{4\ln\frac{4N}{\beta}}.
    \end{eqnarray*}
    Finally, one can bound conditional variances $\sigma_l^2 \eqdef \EE_{\Bxi^l}\left[\alpha_{l+1}^2\left\la\theta_{l+1}^u, 2\eta_l + 2\alpha_{l+1}\zeta_l\right\ra^2\right]$ in the following way:
    \begin{eqnarray}
         \sigma_l^2 &\le& \EE_{\Bxi^l}\left[\alpha_{l+1}^2\left\|\theta_{l+1}^u\right\|_2^2 \left\|2\eta_l + 2\alpha_{l+1}\zeta_l\right\|_2^2\right]\notag\\
         &\overset{\eqref{eq:main_thm_clipped_SSTM_technical_4_2}}{\le}&\alpha_{l+1}^2\EE_{\Bxi^l}\left[\left\|\theta_{l+1}^u\right\|_2^2\right](2CR_0+B)^2 \notag \\
         &=& 4\alpha_{l+1}^2\EE_{\Bxi^l}\left[\left\|\theta_{l+1}^u\right\|_2^2\right]\left(1 + \frac{1}{32\ln\frac{4N}{\beta}}\right)^2C^2R_0^2\notag\\
         &\overset{\eqref{eq:beta_N_condition_clipped_SSTM}}{\le}& 4\alpha_{l+1}^2\EE_{\Bxi^l}\left[\left\|\theta_{l+1}^u\right\|_2^2\right]\left(1 + \frac{1}{64}\right)^2C^2R_0^2 , \label{eq:main_thm_clipped_SSTM_technical_7}
    \end{eqnarray}
    i.e., $\sigma_l^2$ is finite due to finiteness of $\|\theta_{l+1}^u\|_2$ (see Lemma~\ref{lem:main_stoch_lemma_clipped_SSTM}). In other words, sequence $\left\{\alpha_{l+1}\left\la\theta_{l+1}^u, 2\eta_l + 2\alpha_{l+1}\zeta_l\right\ra\right\}_{l\ge 0}$ is a bounded martingale difference sequence with bounded conditional variances $\{\sigma_l^2\}_{l\ge 0}$. Thus, we can apply Bernstein's inequality, i.e.\ we apply Lemma~\ref{lem:Bernstein_ineq} with $X_l = \alpha_{l+1}\left\la\theta_{l+1}^u, 2\eta_l + 2\alpha_{l+1}\zeta_l\right\ra$, $c = \left(1 + \frac{1}{64}\right)\frac{C^2R_0^2}{4\ln\frac{4N}{\beta}}$ and $F = \frac{c^2\ln\frac{4N}{\beta}}{18}$ \revision{(The choice of $F$ will be clarified below.)} and get that for all $b > 0$
    \begin{equation*}
        \PP\left\{\left|\sum\limits_{l=0}^{T-1}X_l\right| > b\text{ and } \sum\limits_{l=0}^{T-1}\sigma_l^2 \le F\right\} \le 2\exp\left(-\frac{b^2}{2F + \nicefrac{2cb}{3}}\right)
    \end{equation*}
    or, equivalently, with probability at least $1 - 2\exp\left(-\frac{b^2}{2F + \nicefrac{2cb}{3}}\right)$
    \begin{equation*}
        \text{either } \sum\limits_{l=0}^{T-1}\sigma_l^2 > F \quad \text{or} \quad \underbrace{\left|\sum\limits_{l=0}^{T-1}X_l\right|}_{|\circledOne|} \le b.
    \end{equation*}
    Let us now choose $b$ in such a way that $2\exp\left(-\frac{b^2}{2F + \nicefrac{2cb}{3}}\right) = \frac{\beta}{2N}$. This implies that $b$ is the positive root of the quadratic equation \begin{equation*}
        b^2 - \frac{2c\ln\frac{4N}{\beta}}{3}b - 2F\ln\frac{4N}{\beta} = 0,
    \end{equation*}
    hence
    \begin{eqnarray*}
         b &=& \frac{c\ln\frac{4N}{\beta}}{3} + \sqrt{\frac{c^2\ln^2\frac{4N}{\beta}}{9} + 2F\ln\frac{4N}{\beta}}\revision{=}\frac{c\ln\frac{4N}{\beta}}{3} + \sqrt{\frac{2c^2\ln^2\frac{4N}{\beta}}{9}}\\
         &=& \frac{1+\sqrt{2}}{3}c\ln\frac{4N}{\beta} \le c\ln\frac{4N}{\beta} = \left(1 + \frac{1}{64}\right)\frac{C^2R_0^2}{4} = \left(\frac{1}{4} + \frac{1}{256}\right)C^2R_0^2.
    \end{eqnarray*}
    That is, with probability at least $1 - \frac{\beta}{2N}$
     \begin{equation*}
        \underbrace{\text{either } \sum\limits_{l=0}^{T-1}\sigma_l^2 > F \quad \text{or} \quad \left|\circledOne\right| \le \left(\frac{1}{4} + \frac{1}{256}\right)C^2R_0^2}_{\text{probability event } E_{\circledOne}}.
    \end{equation*}
    Here and below, we notice that the conditions of Lemma~\ref{lem:main_stoch_lemma_clipped_SSTM} hold when $E_{T-1}$ holds, since event $E_{T-1}$ implies that $x^0, x^1, \ldots, x^T$ lie in $B_{3R_0}(x^*)$. Therefore, probability event $E_{T-1}$ implies that
    \begin{eqnarray*}
        \sum\limits_{l=0}^{T-1}\sigma_l^2  &\overset{\eqref{eq:main_thm_clipped_SSTM_technical_7}}{\le}& 4\left(1 + \frac{1}{64}\right)^2C^2R_0^2\sum\limits_{l=0}^{T-1}\alpha_{l+1}^2\EE_{\Bxi^l}\left[\left\|\theta_{l+1}^u\right\|_2^2\right]\\
        &\overset{\eqref{eq:variance_bound_clipped_SSTM},\eqref{eq:main_thm_clipped_SSTM_technical_4}}{\le}& 72\left(1 + \frac{1}{64}\right)^2\sigma^2C^2R_0^2\sum\limits_{l=0}^{T-1}\frac{\alpha_{l+1}^2}{m_l}\\
        &\overset{\eqref{eq:bathces_clipped_SSTM}}{\le}& \frac{\left(1 + \frac{1}{64}\right)^2C^4R_0^4}{288\ln\frac{4N}{\beta}}\sum\limits_{l=0}^{T-1}\frac{1}{N}\\
        &\overset{T \le N}{\le}& \frac{\left(1 + \frac{1}{64}\right)^2C^4R_0^4}{288\ln\frac{4N}{\beta}} =  \frac{c^2\ln\frac{4N}{\beta}}{18} = F.
    \end{eqnarray*}
    
    \textbf{Upper bound for $\circledTwo$.} The probability event $E_{T-1}$ implies
    \begin{eqnarray*}
        && \hspace{-5em}\alpha_{l+1}\left\la\theta_{l+1}^b, 2\eta_l + 2\alpha_{l+1}\zeta_l\right\ra \\
        &\le& \alpha_{l+1}\left\|\theta_{l+1}^b\right\|_2\left\|2\eta_l + 2\alpha_{l+1}\zeta_l\right\|_2\\
        &\overset{\eqref{eq:bias_bound_clipped_SSTM},\eqref{eq:main_thm_clipped_SSTM_technical_4_2}}{\le}& \alpha_{l+1}\cdot\frac{4\sigma^2}{m_l\lambda_{l+1}}\left(2CR_0 + B\right)\\
        &=& \frac{4\sigma^2\alpha_{l+1}^2}{m_l}\left(1 + \frac{2CR_0}{B}\right) = \frac{4\sigma^2\alpha_{l+1}^2\left(1 + 32\ln\frac{4N}{\beta}\right)}{m_l}\\
        &\overset{\eqref{eq:bathces_clipped_SSTM}}{\le}& \frac{4\left(\frac{1}{\ln\frac{4N}{\beta}} + 32\right)C^2R_0^2}{20736N} \overset{\eqref{eq:beta_N_condition_clipped_SSTM}}{\le} \frac{11C^2R_0^2}{1728N}.
    \end{eqnarray*}
    This implies that
    \begin{eqnarray*}
        \circledTwo &=& \sum\limits_{l=0}^{T-1}\alpha_{l+1}\left\la\theta_{l+1}^b,2\eta_l + 2\alpha_{l+1}\zeta_l\right\ra \overset{T\le N}{\le} \frac{11C^2R_0^2}{1728}.
    \end{eqnarray*}
    
    \textbf{Upper bound for $\circledThree$.} We derive the upper bound for $\circledThree$ using the same technique as for $\circledOne$. First of all, we notice that the summands in $\circledThree$ are conditionally unbiased:
    \begin{equation*}
        \EE_{\Bxi^l}\left[4\alpha_{l+1}^2\left(\|\theta_{l+1}^u\|_2^2 - \EE_{\Bxi^l}\left[\|\theta_{l+1}^u\|_2^2\right]\right)\right] = 0.
    \end{equation*}
    Secondly, the summands are bounded with probability $1$:
    \begin{eqnarray}
        \hspace{-2em} \left|4\alpha_{l+1}^2\left(\|\theta_{l+1}^u\|_2^2 - \EE_{\Bxi^l}\left[\|\theta_{l+1}^u\|_2^2\right]\right)\right| &\le& 4\alpha_{l+1}^2\left(\|\theta_{l+1}^u\|_2^2 + \EE_{\Bxi^l}\left[\|\theta_{l+1}^u\|_2^2\right]\right)\notag\\
       &\overset{\eqref{eq:magnitude_bound_clipped_SSTM}}{\le}& 4\alpha_{l+1}^2\left(4\lambda_{l+1}^2 + 4\lambda_{l+1}^2\right)\notag\\
          &=& 32B^2 = \frac{C^2R_0^2}{8\ln^2\frac{4N}{\beta}} \overset{\eqref{eq:beta_N_condition_clipped_SSTM}}{\le} \frac{C^2R_0^2}{16\ln\frac{4N}{\beta}} \eqdef c_1.\label{eq:main_thm_clipped_SSTM_technical_8}
    \end{eqnarray}
    Finally, one can bound conditional variances \\$\hat \sigma_l^2 \eqdef \EE_{\Bxi^l}\left[\left|4\alpha_{l+1}^2\left(\|\theta_{l+1}^u\|_2^2 - \EE_{\Bxi^l}\left[\|\theta_{l+1}^u\|_2^2\right]\right)\right|^2\right]$ in the following way:
    \begin{eqnarray}
         \hat \sigma_l^2 &\overset{\eqref{eq:main_thm_clipped_SSTM_technical_8}}{\le}& c_1\EE_{\Bxi^l}\left[\left|4\alpha_{l+1}^2\left(\|\theta_{l+1}^u\|_2^2 - \EE_{\Bxi^l}\left[\|\theta_{l+1}^u\|_2^2\right]\right)\right|\right]\notag\\
         &\le& 4c_1\alpha_{l+1}^2\EE_{\Bxi^l}\left[\|\theta_{l+1}^u\|_2^2 + \EE_{\Bxi^l}\left[\|\theta_{l+1}^u\|_2^2\right]\right] = 8c_1\alpha_{l+1}^2\EE_{\Bxi^l}\left[\|\theta_{l+1}^u\|_2^2\right],\label{eq:main_thm_clipped_SSTM_technical_9}
    \end{eqnarray}
     i.e., $\hat \sigma_l^2$ is finite due to finiteness of $\|\theta_{l+1}^u\|_2$ (see Lemma~\ref{lem:main_stoch_lemma_clipped_SSTM}). In other words, sequence $\left\{4\alpha_{l+1}^2\left(\|\theta_{l+1}^u\|_2^2 - \EE_{\Bxi^l}\left[\|\theta_{l+1}^u\|_2^2\right]\right)\right\}_{l\ge 0}$ is bounded martingale difference sequence with bounded conditional variances $\{\hat \sigma_l^2\}_{l\ge 0}$. Therefore, we can apply Bernstein's inequality, i.e.\ we apply Lemma~\ref{lem:Bernstein_ineq} with $X_l = \hat X_l = 4\alpha_{l+1}^2\left(\|\theta_{l+1}^u\|_2^2 - \EE_{\Bxi^l}\left[\|\theta_{l+1}^u\|_2^2\right]\right)$, $c = c_1 = \frac{C^2R_0^2}{16\ln\frac{4N}{\beta}}$ and $F = F_1 = \frac{c_1^2\ln\frac{4N}{\beta}}{18}$ and get that for all $b > 0$
    \begin{equation*}
        \PP\left\{\left|\sum\limits_{l=0}^{T-1}\hat X_l\right| > b\text{ and } \sum\limits_{l=0}^{T-1}\hat \sigma_l^2 \le F_1\right\} \le 2\exp\left(-\frac{b^2}{2F_1 + \nicefrac{2c_1b}{3}}\right)
    \end{equation*}
    or, equivalently, with probability at least $1 - 2\exp\left(-\frac{b^2}{2F_1 + \nicefrac{2c_1b}{3}}\right)$
    \begin{equation*}
        \text{either } \sum\limits_{l=0}^{T-1}\hat\sigma_l^2 > F_1 \quad \text{or} \quad \underbrace{\left|\sum\limits_{l=0}^{T-1}\hat X_l\right|}_{|\circledThree|} \le b.
    \end{equation*}
    As in our derivations of the upper bound for $\circledOne$ we choose such $b$ that $2\exp\left(-\frac{b^2}{2F_1 + \nicefrac{2c_1b}{3}}\right) = \frac{\beta}{2N}$, i.e.,
    \begin{eqnarray*}
         b &=& \frac{c_1\ln\frac{4N}{\beta}}{3} + \sqrt{\frac{c_1^2\ln^2\frac{4N}{\beta}}{9} + 2F_1\ln\frac{4N}{\beta}}\revision{=}\frac{1+\sqrt{2}}{3}c_1\ln\frac{4N}{\beta} \le  \frac{C^2R_0^2}{16}.
    \end{eqnarray*}
    That is, with probability at least $1 - \frac{\beta}{2N}$
     \begin{equation*}
        \underbrace{\text{either } \sum\limits_{l=0}^{T-1}\hat\sigma_l^2 > F_1 \quad \text{or} \quad \left|\circledThree\right| \le \frac{C^2R_0^2}{16}}_{\text{probability event } E_{\circledThree}}.
    \end{equation*}
    Next, we notice that probability event $E_{T-1}$ implies that
    \begin{eqnarray*}
        \sum\limits_{l=0}^{T-1}\hat\sigma_l^2 &\overset{\eqref{eq:main_thm_clipped_SSTM_technical_9}}{\le}& 8c_1\sum\limits_{l=0}^{T-1}\alpha_{l+1}^2\EE_{\Bxi^l}\left[\left\|\theta_{l+1}^u\right\|_2^2\right]\\
        &\overset{\eqref{eq:variance_bound_clipped_SSTM},\eqref{eq:main_thm_clipped_SSTM_technical_4}}{\le}& \frac{9\sigma^2C^2R_0^2}{\ln\frac{4N}{\beta}}\sum\limits_{l=0}^{T-1}\frac{\alpha_{l+1}^2}{m_l} \overset{\eqref{eq:bathces_clipped_SSTM}}{\le} \frac{C^4R_0^4}{2304\ln^2\frac{4N}{\beta}}\sum\limits_{l=0}^{T-1}\frac{1}{N}\\
        &\overset{T\le N}{\le}& \frac{C^4R_0^4}{2304\ln^2\frac{4N}{\beta}} \overset{\eqref{eq:beta_N_condition_clipped_SSTM}}{\le} \frac{C^4R_0^4}{4608\ln\frac{4N}{\beta}} = \frac{c_1^2\ln\frac{4N}{\beta}}{18} = F_1.
    \end{eqnarray*}
    
    \textbf{Upper bound for $\circledFour$.} The probability event $E_{T-1}$ implies
    \begin{eqnarray*}
        \circledFour &=& \sum\limits_{l=0}^{T-1}4\alpha_{l+1}^2\EE_{\Bxi^l}\left[\|\theta_{l+1}^u\|_2^2\right] \overset{\eqref{eq:variance_bound_clipped_SSTM},\eqref{eq:main_thm_clipped_SSTM_technical_4}}{\le} \sum\limits_{l=0}^{T-1}\frac{72\alpha_{l+1}^2\sigma^2}{m_l}\overset{\eqref{eq:bathces_clipped_SSTM}}{\le} \sum\limits_{l=0}^{T-1}\frac{C^2R_0^2}{288N\ln\frac{4N}{\beta}}\\
        &\overset{T\le N}{\le}& \frac{C^2R_0^2}{288\ln\frac{4N}{\beta}} \overset{\eqref{eq:beta_N_condition_clipped_SSTM}}{\le} \frac{C^2R_0^2}{576}.
    \end{eqnarray*}
    
    \textbf{Upper bound for $\circledFive$.} Again, we use corollaries of probability event $E_{T-1}$:
    \begin{eqnarray*}
        \circledFive &=& \sum\limits_{l=0}^{T-1}4\alpha_{l+1}^2\|\theta_{l+1}^b\|_2^2 \overset{\eqref{eq:bias_bound_clipped_SSTM},\eqref{eq:main_thm_clipped_SSTM_technical_4}}{\le} \sum\limits_{l=0}^{T-1}\frac{64\alpha_{l+1}^2\sigma^4}{m_l^2\lambda_{l+1}^2} = \frac{64\sigma^4}{B^2}\sum\limits_{l=0}^{T-1}\frac{\alpha_{l+1}^4}{m_l^2}\\
        &\overset{\eqref{eq:bathces_clipped_SSTM},\eqref{eq:B_a_parameters_clipped_SSTM}}{\le}& \frac{256\cdot 64\sigma^4\ln^2\frac{4N}{\beta}}{C^2R_0^2}\sum\limits_{l=0}^{T-1}\frac{C^4R_0^4}{20736^2N^2\sigma^4\ln^2\frac{4N}{\beta}} \overset{T\le N}{\le} \frac{C^2R_0^2}{26244}.
    \end{eqnarray*}
    Now we summarize all bounds that we have: probability event $E_{T-1}$ implies
    \revision{\begin{gather*}
        R_T^2 
        \overset{\eqref{eq:main_thm_clipped_SSTM_technical_5}}{\le} R_0^2 + \circledOne + \circledTwo + \circledThree + \circledFour + \circledFive + \frac{A_N\varepsilon}{2},\\
        \revision{\text{where\quad}} \circledTwo \le \frac{11C^2R_0^2}{1728},\quad \circledFour \le \frac{CR_0^2}{576},\quad \circledFive \le \frac{C^2R_0^2}{26244},\\
        \sum\limits_{l=0}^{T-1}\sigma_l^2 \le F,\quad \sum\limits_{l=0}^{T-1}\hat\sigma_l^2 \le F_1
    \end{gather*}}
    and 
    \begin{equation*}
        \PP\{E_{T-1}\} \ge 1 - \frac{(T-1)\beta}{N},\quad \PP\{E_\circledOne\} \ge 1 - \frac{\beta}{2N},\quad \PP\{E_\circledThree\} \ge 1 - \frac{\beta}{2N},
    \end{equation*}
    where
    \begin{eqnarray*}
        E_{\circledOne} &=& \left\{\text{either } \sum\limits_{l=0}^{T-1}\sigma_l^2 > F \quad \text{or} \quad \left|\circledOne\right| \le \left(\frac{1}{4} + \frac{1}{256}\right)C^2R_0^2\right\},\\
        E_{\circledThree} &=& \left\{\text{either } \sum\limits_{l=0}^{T-1}\hat\sigma_l^2 > F_1 \quad \text{or} \quad \left|\circledThree\right| \le \frac{C^2R_0^2}{16}\right\}.
    \end{eqnarray*}
    Moreover, since $N \revision{\overset{\eqref{eq:epsilon_clipped_SSTM_3}}{\le}}\frac{2^{\frac{1+\nu}{1+3\nu}}a^{\frac{1+\nu}{1+3\nu}}C^{\frac{2(1+\nu)}{1+3\nu}}R_0^{\frac{2(1+\nu)}{1+3\nu}}M_\nu^{\frac{2}{1+3\nu}}}{\varepsilon^{\frac{2}{1+3\nu}}} + 1$ and \\$\varepsilon \overset{\eqref{eq:epsilon_clipped_SSTM}}{\le} \frac{2^{\frac{1+\nu}{2}}a^{\frac{1+\nu}{2}}C^{1+\nu}R_0^{1+\nu}M_\nu}{100^{\frac{1+3\nu}{2}}}$ we have
    \begin{eqnarray*}
          \frac{A_N\varepsilon}{2} &\overset{\eqref{eq:A_k_upper_bound}}{\le}& \frac{N^{\frac{1+3\nu}{1+\nu}}\varepsilon^{\frac{2}{1+\nu}}}{4aM_\nu^{\frac{2}{1+\nu}}} \revision{\overset{\eqref{eq:epsilon_clipped_SSTM_3}}{\le}} \left(\frac{2^{\frac{1+\nu}{1+3\nu}}a^{\frac{1+\nu}{1+3\nu}}C^{\frac{2(1+\nu)}{1+3\nu}}R_0^{\frac{2(1+\nu)}{1+3\nu}}M_\nu^{\frac{2}{1+3\nu}}}{\varepsilon^{\frac{2}{1+3\nu}}} + 1\right)^{\frac{1+3\nu}{1+\nu}}\frac{\varepsilon^{\frac{2}{1+\nu}}}{4aM_\nu^{\frac{2}{1+\nu}}}\\
          &\overset{\eqref{eq:epsilon_clipped_SSTM}}{\le}& \left(\frac{101}{100}\right)^{\frac{1+3\nu}{1+\nu}} \frac{C^2R_0^2}{2} \le \frac{10201C^2R_0^2}{20000}.
    \end{eqnarray*}
     Taking into account these inequalities we get that probability event $E_{T-1}\cap E_\circledOne \cap E_\circledThree$ implies
   \revision{\begin{eqnarray}
         R_T^2 
         &\le& \left(1 + \left(\frac{1}{4} + \frac{1}{256} + \frac{11}{1728} + \frac{1}{16} + \frac{1}{576} + \frac{1}{26244} + \frac{10201}{20000}\right)C^2\right)R_0^2\notag\\
         &\overset{\eqref{eq:C_definition_clipped_SSTM}}{\le}& C^2R_0^2.\label{eq:main_thm_clipped_SSTM_technical_10}
    \end{eqnarray}}
    Moreover, \revision{using the bound for the union}, we
derive
    \begin{equation}
        \PP\left\{E_{T-1}\cap E_\circledOne \cap E_\circledThree\right\} = 1 - \PP\left\{\overline{E}_{T-1}\cup\overline{E}_\circledOne\cup\overline{E}_\circledThree\right\} \ge 1 - \frac{T\beta}{N}.\label{eq:main_thm_clipped_SSTM_technical_11}
    \end{equation}
    That is, by definition of $E_T$ and $E_{T-1}$ we have \revision{proven} that
    \begin{eqnarray*}
        \PP\{E_T\} &\overset{\eqref{eq:main_thm_clipped_SSTM_technical_10}}{\ge}& \PP\left\{E_{T-1}\cap E_\circledOne \cap E_\circledThree\right\} \overset{\eqref{eq:main_thm_clipped_SSTM_technical_11}}{\ge} 1 - \frac{T\beta}{N},
    \end{eqnarray*}
    which implies that for all $k = 0,1,\ldots, N$ we have $\PP\{E_k\} \ge 1 - \frac{k\beta}{N}$. Then, for $k = N$ we have that with probability at least $1-\beta$
    \begin{eqnarray*}
        &A_N&\left(f(y^N) - f(x^*)\right) \\
        &\overset{\eqref{eq:main_thm_clipped_SSTM_technical_0}}{\le}& \frac{1}{2}\|z^0 - \revision{x^*}\|_2^2 - \frac{1}{2}\|z^{N} - \revision{x^*}\|_2^2 + \sum\limits_{k=0}^{N-1}\alpha_{k+1}\left\la \theta_{k+1}, \revision{x^*} - z^{k}\right\ra\notag\\
        &&\quad + \sum\limits_{k=0}^{N-1}\alpha_{k+1}^2\left\|\theta_{k+1}\right\|_2^2 + \sum\limits_{k=0}^{N-1}\alpha_{k+1}^2\left\la\theta_{k+1},\nabla f(x^{k+1})\right\ra + \frac{A_N\varepsilon}{4}\overset{\eqref{eq:main_thm_clipped_SSTM_technical_2}}{\le} \frac{C^2R_0^2}{2}.
    \end{eqnarray*}
    Since $A_{N} \overset{\eqref{eq:A_k_lower_bound}}{\ge} \frac{N^{\frac{1+3\nu}{1+\nu}}(\nicefrac{\varepsilon}{2})^{\frac{1-\nu}{1+\nu}}}{2^{\frac{1+3\nu}{1+\nu}}aM_\nu^{\frac{2}{1+\nu}}}$ we get that with probability at least $1-\beta$
    \begin{equation*}
        f(y^N) - f(x^*) \le \frac{4aC^2R_0^2M_\nu^{\frac{2}{1+\nu}}}{N^{\frac{1+3\nu}{1+\nu}}\varepsilon^{\frac{1-\nu}{1+\nu}}}.
    \end{equation*}
    In other words, {\tt clipped-SSTM} with $a = 16384\ln^2\frac{4N}{\beta}$ achieves $f(y^N) - f(x^*) \le \varepsilon$ with probability at least $1-\beta$ after $\cO\left(\frac{M_\nu^{\frac{2}{1+3\nu}}R_0^{\frac{2(1+\nu)}{1+3\nu}}}{\varepsilon^{\frac{2}{1+3\nu}}}\ln^{\frac{2(1+\nu)}{1+3\nu}}\frac{M_\nu^{\frac{2}{1+3\nu}}R_0^{\frac{2(1+\nu)}{1+3\nu}}}{\varepsilon^{\frac{2}{1+3\nu}}\beta}\right)$ iterations and requires
    \begin{eqnarray*}
         &&\sum\limits_{k=0}^{N-1}m_k \overset{\eqref{eq:bathces_clipped_SSTM}}{=} \sum\limits_{k=0}^{N-1}\cO\left(\max\left\{1,\frac{\sigma^2\alpha_{k+1}^2N\ln\frac{N}{\beta}}{R_0^2}\right\}\right)\\
         &=& \cO\left(\max\left\{N,\sum\limits_{k=0}^{N-1}\frac{\sigma^2(k+1)^{\frac{4\nu}{1+\nu}}\varepsilon^{\frac{2(1-\nu)}{1+\nu}}N\ln\frac{N}{\beta}}{M_\nu^{\frac{4}{1+\nu}}R_0^2a^2}\right\}\right)\\
         &\overset{\eqref{eq:B_a_parameters_clipped_SSTM}}{=}& \cO\left(\max\left\{N,\frac{\sigma^2\varepsilon^{\frac{2(1-\nu)}{1+\nu}}N^{\frac{2(1+3\nu)}{1+\nu}}}{M_\nu^{\frac{4}{1+\nu}}R_0^2\ln^3\frac{N}{\beta}}\right\}\right)\\
         &=& \cO\left(\max\left\{\frac{M_\nu^{\frac{2}{1+3\nu}}R_0^{\frac{2(1+\nu)}{1+3\nu}}}{\varepsilon^{\frac{2}{1+3\nu}}}\ln^{\frac{2(1+\nu)}{1+3\nu}}\frac{M_\nu^{\frac{2}{1+3\nu}}R_0^{\frac{2(1+\nu)}{1+3\nu}}}{\varepsilon^{\frac{2}{1+3\nu}}\beta},\frac{\sigma^2R_0^2}{\varepsilon^2}\ln\frac{M_\nu^{\frac{2}{1+3\nu}}R_0^{\frac{2(1+\nu)}{1+3\nu}}}{\varepsilon^{\frac{2}{1+3\nu}}\beta}\right\}\right).
    \end{eqnarray*}
    oracle calls\revision{, where in the last equality we substituted the number of iterations $N$ from the statement of the theorem}. \qed
\end{proof}

\subsubsection{On the Batchsizes and Numerical Constants}
The obtained complexity result is discussed in detail in Section~\ref{sec:clipped_SSTM_main}. Here, we discuss the choice of the parameters. For convenience, we provide all assumptions from Theorem~\ref{thm:main_result_clipped_SSTM} on the parameters below:
\begin{equation}
    \ln\frac{4N}{\beta} \ge 2 \label{eq:requirement_1}
\end{equation}
\begin{equation}
    \alpha = \frac{(\nicefrac{\varepsilon}{2})^{\frac{1-\nu}{1+\nu}}}{2^{\frac{2\nu}{1+\nu}}aM_\nu^{\frac{2}{1+\nu}}},\quad m_k = \max\left\{1, \frac{20736N\sigma^2\alpha_{k+1}^2\ln \frac{4N}{\beta}}{C^2R_0^2}\right\},\label{eq:requirement_2}
\end{equation}
\begin{equation}
     B = \frac{CR_0}{16\ln\frac{4N}{\beta}},\quad a \ge 16384\ln^2\frac{4N}{\beta}, \label{eq:requirement_3}
\end{equation}
\begin{equation}
        \varepsilon^{\frac{1-\nu}{1+\nu}} \le \frac{aCM_\nu^{\frac{1-\nu}{1+\nu}}R_0^{1-\nu}}{16\ln\frac{4N}{\beta}},\quad \varepsilon \le \frac{2^{\frac{1+\nu}{2}}a^{\frac{1+\nu}{2}}C^{1+\nu}R_0^{1+\nu}M_\nu}{100^{\frac{1+3\nu}{2}}}, \label{eq:requirement_4}
    \end{equation}
    \begin{eqnarray}
         \varepsilon^{\frac{1-\nu}{1+3\nu}} &\le& \min\Bigg\{\frac{a^{\frac{2+3\nu-\nu^2}{2(1+3\nu)}}}{2^{2+4\nu+\frac{3+8\nu-5\nu^2-6\nu^3}{(1+\nu)(1+3\nu)}}\ln\frac{4N}{\beta}},\notag\\
         &&\quad\quad\quad\quad\frac{a^{\frac{(1+\nu)^2}{1+3\nu}}}{2^{4+7\nu+\frac{2+7\nu+2\nu^2-3\nu^3}{(1+\nu)(1+3\nu)}}\ln^{1+\nu}\frac{4N}{\beta}}\Bigg\}C^{\frac{1-\nu^2}{1+3\nu}}R_0^{\frac{1-\nu^2}{1+3\nu}}M_\nu^{\frac{1-\nu}{1+3\nu}},\label{eq:requirement_4_1}
    \end{eqnarray}
\begin{equation}
    N = \left\lceil\frac{2^{\frac{1+\nu}{1+3\nu}}a^{\frac{1+\nu}{1+3\nu}}C^{\frac{2(1+\nu)}{1+3\nu}}R_0^{\frac{2(1+\nu)}{1+3\nu}}M_\nu^{\frac{2}{1+3\nu}}}{\varepsilon^{\frac{2}{1+3\nu}}}\right\rceil + 1,\quad C = \sqrt{7}. \label{eq:requirement_5}
\end{equation}
We emphasize that \eqref{eq:requirement_1}, \eqref{eq:requirement_4}, and \eqref{eq:requirement_4_1} are not restrictive at all since the target accuracy $\varepsilon$ and confidence level $\beta$ are often chosen to be small enough, whereas $a$ can be made large enough. 

\revision{One} can notice that the assumptions on parameter $a$ and batch size $m_k$ contain huge numerical constants (see \eqref{eq:requirement_2}-\eqref{eq:requirement_3}) that result in large numerical constants in the expression for the number of iterations $N$ and the total number of oracle calls required to guarantee accuracy $\varepsilon$ of the solution. However, for the sake of simplicity of the proofs, we do not try to provide an analysis with better dependence on the numerical constants. Moreover, the main goal of this paper is to derive improved high-probability complexity guarantees in terms of $\cO(\cdot)$-notation -- such guarantees are insensitive to numerical constants by definition.

Finally, \eqref{eq:requirement_2} implies that the batch size at iteration $k$ is
\begin{eqnarray*}
    m_k &=& \Theta\left(\max\left\{1, \frac{N\sigma^2(k+1)^{\frac{4\nu}{1+\nu}}\varepsilon^{\frac{2(1-\nu)}{1+\nu}}\ln\frac{N}{\beta}}{a^2 M_\nu^{\frac{4}{1+\nu}}R_0^2}\right\}\right)
\end{eqnarray*}
meaning that for $k \sim N$ and $a = \cO\left(\ln^2\frac{N}{\beta}\right)$ we have that the second term in the maximum is proportional to $N^{\frac{1+5\nu}{1+\nu}}\varepsilon^{\frac{2(1-\nu)}{1+\nu}}$. When $\nu$ is close to $1$ and $\sigma^2 \gg 0$, it implies that $m_k$ is huge for big enough $k$, making the method completely impractical. Fortunately, this issue can be easily solved without sacrificing the oracle complexity of the method: it is sufficient to choose large enough $a$.

\begin{corollary}\label{cor:SSTM_cvx_small_batch}
    Let the assumptions of Theorem~\ref{thm:main_result_clipped_SSTM} hold and
    \begin{equation}
        a = \max\left\{16384\ln^2\frac{4N}{\beta}, \frac{5184^{\frac{1+3\nu}{1+\nu}}\cdot 2^{\frac{2(1+5\nu)(1+2\nu)}{(1+\nu)^2}}\sigma^{\frac{2(1+3\nu)}{1+\nu}}C^{\frac{4\nu}{1+\nu}}R_0^{\frac{4\nu}{1+\nu}}\ln^{\frac{1+3\nu}{1+\nu}}\frac{4N}{\beta}}{M_\nu^{\frac{2}{1+\nu}}\varepsilon^{\frac{6\nu}{1+\nu}}}\right\}. \label{eq:refined_assumption_on_a_SSTM}
    \end{equation}
    Then for all $k=0,1,\ldots,N-1$ we have $m_k = 1$ and to achieve $f(y^N) - f(x^*) \le \varepsilon$ with probability at least $1-\beta$ \algname{clipped-SSTM} requires
    \begin{equation}
        \cO\left(\max\left\{\frac{M_\nu^{\frac{2}{1+3\nu}}R_0^{\frac{2(1+\nu)}{1+3\nu}}}{\varepsilon^{\frac{2}{1+3\nu}}}\ln^{\frac{2(1+\nu)}{1+3\nu}}\frac{M_\nu^{\frac{2}{1+3\nu}}R_0^{\frac{2(1+\nu)}{1+3\nu}}}{\varepsilon^{\frac{2}{1+3\nu}}\beta}, \frac{\sigma^2R_0^2}{\varepsilon^2}\ln\frac{\sigma^2R_0^2}{\varepsilon^2\beta}\right\}\right) \label{eq:SSTM_complexity_small_batch_cvx}
    \end{equation}
    iterations/oracle calls.
\end{corollary}
\begin{proof}
    We start with showing that for the new choice of $a$ we have $m_k = 1$ for all $k=0,1,\ldots,N-1$. Indeed, using the assumptions on the parameters from Theorem~\ref{thm:main_result_clipped_SSTM}, we derive
    \begin{eqnarray*}
        m_k &=& \max\left\{1, \frac{20736N\sigma^2\alpha_{k+1}^2\ln \frac{4N}{\beta}}{C^2R_0^2}\right\} \\
        &=& \max\left\{1, \frac{5184N\sigma^2 (k+1)^{\frac{4\nu}{1+\nu}}\varepsilon^{\frac{2(1-\nu)}{1+\nu}}}{a^2M_\nu^{\frac{4}{1+\nu}}C^2R_0^2}\right\}\\
        &\overset{k < N}{\le}& \max\left\{1, \frac{5184\sigma^2 N^{\frac{1+5\nu}{1+\nu}}\varepsilon^{\frac{2(1-\nu)}{1+\nu}}}{a^2M_\nu^{\frac{4}{1+\nu}}C^2R_0^2}\right\}\\
        &\overset{\eqref{eq:C_definition_clipped_SSTM}}{\le}& \max\left\{1, \frac{\revision{5184\cdot 2^{\frac{2(1+5\nu)(1+2\nu)}{(1+\nu)(1+3\nu)}}\sigma^2C^{\frac{4\nu}{1+3\nu}}R_0^{\frac{4\nu}{1+3\nu}}}}{a^{\frac{1+\nu}{1+3\nu}}M_\nu^{\frac{2}{1+3\nu}}\varepsilon^{\frac{6\nu}{1+3\nu}}}\right\} \overset{\eqref{eq:refined_assumption_on_a_SSTM}}{\le} 1.
    \end{eqnarray*}
    That is, with the choice of the stepsize parameter $a$ as in \eqref{eq:refined_assumption_on_a_SSTM}, the method uses unit batch sizes at each iteration. Therefore, iteration and oracle complexities coincide in this case. Next, we consider two possible situations.
    \begin{enumerate}
        \item If $a = 16384\ln^2\frac{4N}{\beta}$, then 
        \begin{eqnarray*}
             N &\overset{\eqref{eq:C_definition_clipped_SSTM}}{=}& \left\lceil\frac{2^{\frac{1+\nu}{1+3\nu}}a^{\frac{1+\nu}{1+3\nu}}C^{\frac{2(1+\nu)}{1+3\nu}}R_0^{\frac{2(1+\nu)}{1+3\nu}}M_\nu^{\frac{2}{1+3\nu}}}{\varepsilon^{\frac{2}{1+3\nu}}}\right\rceil + 1 \\
             &=&  \cO\left(\frac{M_\nu^{\frac{2}{1+3\nu}}R_0^{\frac{2(1+\nu)}{1+3\nu}}}{\varepsilon^{\frac{2}{1+3\nu}}}\ln^{\frac{2(1+\nu)}{1+3\nu}}\frac{N}{\beta}\right)\\
             &=& \cO\left(\frac{M_\nu^{\frac{2}{1+3\nu}}R_0^{\frac{2(1+\nu)}{1+3\nu}}}{\varepsilon^{\frac{2}{1+3\nu}}}\ln^{\frac{2(1+\nu)}{1+3\nu}}\frac{M_\nu^{\frac{2}{1+3\nu}}R_0^{\frac{2(1+\nu)}{1+3\nu}}}{\varepsilon^{\frac{2}{1+3\nu}}\beta}\right).
        \end{eqnarray*}
        \item If $a = \frac{5184^{\frac{1+3\nu}{1+\nu}}\cdot 2^{\frac{2(1+5\nu)(1+2\nu)}{(1+\nu)^2}}\sigma^{\frac{2(1+3\nu)}{1+\nu}}C^{\frac{4\nu}{1+\nu}}R_0^{\frac{4\nu}{1+\nu}}\ln^{\frac{1+3\nu}{1+\nu}}\frac{4N}{\beta}}{M_\nu^{\frac{2}{1+\nu}}\varepsilon^{\frac{6\nu}{1+\nu}}}$, then 
        \begin{eqnarray*}
             N &\overset{\eqref{eq:C_definition_clipped_SSTM}}{=}& \left\lceil\frac{2^{\frac{1+\nu}{1+3\nu}}a^{\frac{1+\nu}{1+3\nu}}C^{\frac{2(1+\nu)}{1+3\nu}}R_0^{\frac{2(1+\nu)}{1+3\nu}}M_\nu^{\frac{2}{1+3\nu}}}{\varepsilon^{\frac{2}{1+3\nu}}}\right\rceil + 1\\
             &=& \cO\left(\frac{M_\nu^{\frac{2}{1+3\nu}}R_0^{\frac{2(1+\nu)}{1+3\nu}}}{\varepsilon^{\frac{2}{1+3\nu}}}\cdot \frac{\sigma^2R_0^{\frac{4\nu}{1+3\nu}}\ln\frac{4N}{\beta}}{M_\nu^{\frac{2}{1+3\nu}}\varepsilon^{\frac{6\nu}{1+3\nu}}}\right) = \cO\left(\frac{\sigma^2R_0^2}{\varepsilon^2}\ln\frac{\sigma^2R_0^2}{\varepsilon^2\beta}\right).
        \end{eqnarray*}
    \end{enumerate}
    Putting all together, we derive \eqref{eq:SSTM_complexity_small_batch_cvx}. \qed
\end{proof}

\subsection{Convergence in the Strongly Convex Case}\label{sec:clipped_SSTM_str_cvx_appendix}
In this section, we provide the full proof of Theorem~\ref{thm:main_result_clipped_SSTM_str_cvx_main} together with a complete statement of the result.  Note that due to strong convexity, the solution $x^*$ is unique. 
\begin{theorem}\label{thm:main_result_clipped_SSTM_str_cvx}
    Assume that function $f$ is $\mu$-strongly convex,  its stochastic gradient and its gradient satisfy \eqref{eq:bounded_variance_clipped_SSTM} and \eqref{eq:holder_def} respectively with $\sigma > 0$, $\nu \in [0,1]$, $M_\nu > 0$ on $Q = B_{3R_0}(x^*)$, where $R_0 \ge \|x^0 - x^*\|_2$. Let $\varepsilon > 0$, $\beta \in (0,1)$ and for $t = 1,\ldots, \tau$
    \begin{equation}
        \revision{N_t = \left\lceil\frac{a_t^{\frac{1+\nu}{1+3\nu}}C^{\frac{2(1+\nu)}{1+3\nu}}R_0^{\frac{2(1+\nu)}{1+3\nu}}M_\nu^{\frac{2}{1+3\nu}}}{2^{\frac{(1+\nu)(t-2)}{1+3\nu}}\varepsilon_t^{\frac{2}{1+3\nu}}}\right\rceil + 1},\quad \varepsilon_t = \frac{\mu R_0^2}{2^{t+1}},\label{eq:N_t_eps_t_clipped_SSTM_str_cvx}
    \end{equation}
    \begin{equation}
        \tau = \left\lceil\log_2\frac{\mu R_0}{2\varepsilon}\right\rceil \revision{-1},\quad \ln\frac{4N_t \tau}{\beta} \ge 2,\quad C = \sqrt{7}, \label{eq:beta_N_condition_clipped_SSTM_str_cvx}
    \end{equation}
    \begin{equation}
        \revision{\alpha^t = \frac{\varepsilon_t^{\frac{1-\nu}{1+\nu}}}{2a_tM_\nu^{\frac{2}{1+\nu}}}},\quad m_k^t = \max\left\{1, \frac{20736\cdot 2^{t-1} N_t\sigma^2(\alpha_{k+1}^t)^2\ln \frac{4N_t\tau}{\beta}}{C^2R_0^2}\right\},\label{eq:bathces_clipped_SSTM_str_cvx}
    \end{equation}
    \begin{equation}
        \alpha_{k+1}^t = \alpha^t(k+1)^{\frac{2\nu}{1+\nu}},\quad \revision{B_t} = \frac{CR_0}{16\ln\frac{4N_t\tau}{\beta}},\quad a_t = 16384\ln^2\frac{4N_t\tau}{\beta}, \label{eq:B_a_parameters_clipped_SSTM_str_cvx}
    \end{equation}
   \begin{equation}
        \varepsilon_t^{\frac{1-\nu}{1+\nu}} \le \frac{a_tCM_\nu^{\frac{1-\nu}{1+\nu}}R_0^{1-\nu}}{16\cdot 2^{\frac{(1-\nu)(t-1)}{2}}\ln\frac{4N_t\tau}{\beta}},\quad \revision{\varepsilon_t \le \frac{a_t^{\frac{1+\nu}{2}}C^{1+\nu}R_0^{1+\nu}M_\nu}{100^{\frac{1+3\nu}{2}}\cdot 2^{\frac{(1+\nu)(t-2)}{2}}}}, \label{eq:epsilon_SSTM_str_cvx_1}
    \end{equation}
    \begin{eqnarray}
         \varepsilon_t^{\frac{1-\nu}{1+3\nu}} &\le& \min\Bigg\{\frac{a_t^{\frac{2+3\nu-\nu^2}{2(1+3\nu)}}}{2^{2+4\nu+\frac{3+8\nu-5\nu^2-6\nu^3}{(1+\nu)(1+3\nu)}}\ln\frac{4N_t\tau}{\beta}},\notag\\
         &&\quad\quad\quad\frac{a_t^{\frac{(1+\nu)^2}{1+3\nu}}}{2^{4+7\nu+\frac{2+7\nu+2\nu^2-3\nu^3}{(1+\nu)(1+3\nu)}}\ln^{1+\nu}\frac{4N_t\tau}{\beta}}\Bigg\}\frac{C^{\frac{1-\nu^2}{1+3\nu}}R_0^{\frac{1-\nu^2}{1+3\nu}}M_\nu^{\frac{1-\nu}{1+3\nu}}}{2^{\frac{(1-\nu^2)(t-1)}{2(1+3\nu)}}}.\label{eq:epsilon_SSTM_str_cvx_2}
    \end{eqnarray}
    Then, after $\tau$ restarts \algname{R-clipped-SSTM} produces $\hat x^\tau$ such that with probability at least $1-\beta$
    \begin{equation}
        f(\hat x^\tau) - f(x^*) \le \varepsilon. \label{eq:main_result_clipped_SSTM_str_cvx}
    \end{equation}
    That is, to achieve \eqref{eq:main_result_clipped_SSTM_str_cvx} with probability at least $1-\beta$ the method requires
    \begin{equation}
        \hat{N} = \cO\left(\max\left\{\left(\frac{M_\nu}{\mu R_0^{1-\nu}}\right)^{\frac{2}{1+3\nu}}\ln\frac{\mu R_0^2}{\varepsilon}, \left(\frac{M_\nu^2}{\mu^{1+\nu}\varepsilon^{1-\nu}}\right)^{\frac{1}{1+3\nu}}\right\}\ln^{\frac{2(1+\nu)}{1+3\nu}}\frac{M_\nu^{\frac{2}{1+3\nu}}\ln\frac{\mu R_0^2}{\varepsilon}}{\mu^{\frac{1+\nu}{1+3\nu}}\varepsilon^{\frac{1-\nu}{1+3\nu}}\beta}\right) \label{eq:clipped_SSTM_iter_complexity_str_cvx}
    \end{equation}
    iterations of Algorithm \ref{alg:clipped-SSTM} and
    \begin{equation}
        \cO\left(\max\left\{\hat{N}, \frac{\sigma^2}{\mu \varepsilon}\ln \frac{M_\nu^{\frac{2}{1+3\nu}}\ln\frac{\mu R_0^2}{\varepsilon}}{\mu^{\frac{1+\nu}{1+3\nu}}\varepsilon^{\frac{1-\nu}{1+3\nu}}\beta}\right\}\right)\text{ oracle calls.} \label{eq:clipped_SSTM_oracle_complexity_str_cvx}
    \end{equation}
\end{theorem}
\begin{proof}
    Applying \revision{the convergence rate result \eqref{eq:main_result_clipped_SSTM} in Theorem~\ref{thm:main_result_clipped_SSTM} together with our choice of the parameters of the first restart}, we obtain that with probability at least $1-\frac{\beta}{\tau}$ it holds that 
    $
        f(\hat{x}^1) - f(x^*) \le \frac{\mu R_0^2}{4}.
    $
    Since $f$ is $\mu$-strongly convex we have
        $\frac{\mu \|\hat{x}^1 - x^*\|_2^2}{2} \le f(\hat{x}^1) - f(x^*)$.
    Therefore, with probability at least $1-\frac{\beta}{\tau}$
    \begin{equation*}
        f(\hat{x}^1) - f(x^*) \le \frac{\mu R_0^2}{4},\quad \|\hat{x}^1 - x^*\|_2^2 \le \frac{R_0^2}{2}.
    \end{equation*}
    From mathematical induction and the union bound for probability events, it follows that \revision{the} inequalities
    \begin{equation*}
        f(\hat{x}^t) - f(x^*) \le \frac{\mu R_0^2}{2^{t+1}},\quad \|\hat{x}^t - x^*\|_2^2 \le \frac{R_0^2}{2^t}
    \end{equation*}
    hold simultaneously for $t = 1,\ldots, \tau$ with probability at least $1 - \beta$. Thus, it means that after $\tau = \left\lceil\log_2\frac{\mu R_0^2}{\varepsilon} \right\rceil - 1$ restarts \algname{R-clipped-SSTM} finds an $\varepsilon$-solution with probability at least $1-\beta$. The total number of iterations $\hat N$ is
    \begin{eqnarray*}
        &&\hspace{-1em}\sum\limits_{t=1}^\tau N_t\\
        &\revision{\overset{\eqref{eq:N_t_eps_t_clipped_SSTM_str_cvx}, \eqref{eq:B_a_parameters_clipped_SSTM_str_cvx}}{=}}& \cO\left(\sum\limits_{t=1}^\tau\frac{M_\nu^{\frac{2}{1+3\nu}}R_0^{\frac{2(1+\nu)}{1+3\nu}}}{2^{\frac{(1+\nu)t}{1+3\nu}}\varepsilon_t^{\frac{2}{1+3\nu}}}\ln^{\frac{2(1+\nu)}{1+3\nu}}\frac{M_\nu^{\frac{2}{1+3\nu}}R_0^{\frac{2(1+\nu)}{1+3\nu}}\tau}{2^{\frac{(1+\nu)t}{1+3\nu}}\varepsilon_t^{\frac{2}{1+3\nu}}\beta}\right)\\
        &=& \cO\left(\sum\limits_{t=1}^\tau\frac{M_\nu^{\frac{2}{1+3\nu}}R_0^{\frac{2(1+\nu)}{1+3\nu}}2^{\frac{2t}{1+3\nu}}}{2^{\frac{(1+\nu)t}{1+3\nu}}\mu^{\frac{2}{1+3\nu}}R_0^{\frac{4}{1+3\nu}}}\ln^{\frac{2(1+\nu)}{1+3\nu}}\frac{M_\nu^{\frac{2}{1+3\nu}}R_0^{\frac{2(1+\nu)}{1+3\nu}}2^{\frac{2t}{1+3\nu}}\tau}{2^{\frac{(1+\nu)t}{1+3\nu}}\mu^{\frac{2}{1+3\nu}}R_0^{\frac{4}{1+3\nu}}\beta}\right)\\
        &=& \cO\left(\sum\limits_{t=1}^\tau\frac{M_\nu^{\frac{2}{1+3\nu}}2^{\frac{(1-\nu)t}{1+3\nu}}}{\mu^{\frac{2}{1+3\nu}}R_0^{\frac{2(1-\nu)}{1+3\nu}}}\ln^{\frac{2(1+\nu)}{1+3\nu}}\frac{M_\nu^{\frac{2}{1+3\nu}}2^{\frac{(1-\nu)t}{1+3\nu}}\tau}{\mu^{\frac{2}{1+3\nu}}R_0^{\frac{2(1-\nu)}{1+3\nu}}\beta}\right)\\
        &=& \cO\left(\frac{M_\nu^{\frac{2}{1+3\nu}}\max\left\{\tau,2^{\frac{(1-\nu)\tau}{1+3\nu}}\right\}}{\mu^{\frac{2}{1+3\nu}}R_0^{\frac{2(1-\nu)}{1+3\nu}}}\ln^{\frac{2(1+\nu)}{1+3\nu}}\frac{M_\nu^{\frac{2}{1+3\nu}}2^{\frac{(1-\nu)\tau}{1+3\nu}}\tau}{\mu^{\frac{2}{1+3\nu}}R_0^{\frac{2(1-\nu)}{1+3\nu}}\beta}\right)\\
        &=& \cO\left(\max\left\{\left(\frac{M_\nu}{\mu R_0^{1-\nu}}\right)^{\frac{2}{1+3\nu}}\ln\frac{\mu R_0^2}{\varepsilon}, \left(\frac{M_\nu^2}{\mu^{1+\nu}\varepsilon^{1-\nu}}\right)^{\frac{1}{1+3\nu}}\right\}\ln^{\frac{2(1+\nu)}{1+3\nu}}\frac{M_\nu^{\frac{2}{1+3\nu}}\ln\frac{\mu R_0^2}{\varepsilon}}{\mu^{\frac{1+\nu}{1+3\nu}}\varepsilon^{\frac{1-\nu}{1+3\nu}}\beta}\right),
    \end{eqnarray*}
    and the total number of oracle calls equals
    \begin{eqnarray*}
         \sum\limits_{t=1}^\tau\sum\limits_{k=0}^{N_t-1}m_k^t &=& \cO\left(\max\left\{\sum\limits_{t=1}^\tau N_t, \sum\limits_{t=1}^\tau \frac{\sigma^2 R_0^2}{2^t \varepsilon_t^2}\ln \frac{M_\nu^{\frac{2}{1+3\nu}}2^{\frac{(1-\nu)t}{1+3\nu}}\tau}{\mu^{\frac{2}{1+3\nu}}R_0^{\frac{2(1-\nu)}{1+3\nu}}\beta}\right\}\right)\\
         &=& \cO\left(\max\left\{\hat{N}, \sum\limits_{t=1}^\tau \frac{\sigma^2 \cdot 2^{t}}{\mu^2 R_0^2}\ln \frac{M_\nu^{\frac{2}{1+3\nu}}2^{\frac{(1-\nu)\tau}{1+3\nu}}\tau}{\mu^{\frac{2}{1+3\nu}}R_0^{\frac{2(1-\nu)}{1+3\nu}}\beta}\right\}\right)\\
         &=& \cO\left(\max\left\{\hat{N}, \frac{\sigma^2}{\mu \varepsilon}\ln \frac{M_\nu^{\frac{2}{1+3\nu}}\ln\frac{\mu R_0^2}{\varepsilon}}{\mu^{\frac{1+\nu}{1+3\nu}}\varepsilon^{\frac{1-\nu}{1+3\nu}}\beta}\right\}\right). \qquad \qquad \qquad \qquad \qed
    \end{eqnarray*} 
\end{proof}
One can also derive a similar result for \algname{R-clipped-SSTM} when stepsize parameter $a$ is chosen as in Corollary~\ref{cor:SSTM_cvx_small_batch} for all restarts. In this case, on can choose unit batch sizes: $m_k^t = 1$ for all $k$ and $t$.

\section{SGD with Clipping: Missing Details and Proofs}\label{sec:clipped_SGD_appendix}

\subsection{Convex Case}\label{sec:clipped_SGD_cvx_appendix}
In this section, we provide a full statement of Theorem~\ref{thm:main_result_clipped_SGD_main} together with its proof. The proof is based on a similar idea as the proof of the complexity bounds for \algname{clipped-SSTM}.
\begin{theorem}\label{thm:main_result_clipped_SGD}
    Assume that \revision{the} function $f$ is convex,  achieves its minimum at a point $x^*$, and its stochastic gradient and its gradient satisfy \eqref{eq:bounded_variance_clipped_SSTM} and \eqref{eq:holder_def} respectively with $\sigma > 0$, $\nu \in [0,1]$, $M_\nu > 0$ on $Q = B_{7R_0}(x^*)$, where $R_0 \ge \|x^0 - x^*\|_2$. 
    Then, for all $\beta \in (0,1)$ and $N$ such that
    \begin{equation}
        \ln\frac{4N}{\beta} \ge 2, \label{eq:beta_N_condition_clipped_SGD}
    \end{equation}
    we have that after $N$ iterations of \algname{clipped-SGD} with
    \begin{equation}
        \lambda = \frac{R_0}{\gamma \ln\frac{4N}{\beta}} ,\quad m \ge \max\left\{1,\frac{81N\sigma^2}{\lambda^2\ln\frac{4N}{\beta}}\right\}\label{eq:bathces_clipped_SGD}
    \end{equation}
    and stepsize
    \begin{equation}
        \gamma \le \min \left\{\frac{\varepsilon^{\frac{1-\nu}{1+\nu}}}{8M_{\nu}^{\frac{2}{1+\nu}}}, \; \frac{R_0}{\sqrt{2N}\varepsilon^{\frac{\nu}{1+\nu}}M_{\nu}^{\frac{1}{1+\nu}}}, \; \frac{R_0^{1-\nu}}{2C^{\nu}M_{\nu}\ln\frac{4N}{\beta}}\right\},\label{eq:step_size_clipped_SGD}
    \end{equation}
    with probability at least $1-\beta$ it holds that
    \begin{equation}
        f(\Bar{x}^N) - f(x^*) \le \frac{C^2R_0^2}{\gamma N}, \label{eq:main_result_clipped_SGD}
    \end{equation}
    where $\Bar{x}^N = \frac{1}{N}\sum_{k=0}^{N-1}x^k$ and
    \begin{equation}
        C = 7. \label{eq:C_definition_clipped_SGD}
    \end{equation}
    In other words, \algname{clipped-SGD} with $\gamma = \min \left\{\frac{\varepsilon^{\frac{1-\nu}{1+\nu}}}{8M_{\nu}^{\frac{2}{1+\nu}}},  \frac{R_0}{\sqrt{2N}\varepsilon^{\frac{\nu}{1+\nu}}M_{\nu}^{\frac{1}{1+\nu}}},  \frac{R_0^{1-\nu}}{2C^{\nu}M_{\nu}\ln\frac{4N}{\beta}}\right\}$ achieves $f(\Bar{x}^N) - f(x^*) \le \varepsilon$ with probability at least $1-\beta$ after \\
    $\cO\left(\max\left\{\frac{M_\nu^{\frac{2}{1+\nu}}R_0^2}{\varepsilon^{\frac{2}{1+\nu}}},\frac{M_\nu R_0^{1+\nu}}{\varepsilon}\ln\frac{M_{\nu}R_0^{1+\nu}}{\varepsilon\beta}\right\}\right)$ iterations and requires
    \begin{equation}
       \cO\left(\max\left\{\frac{M_\nu^{\frac{2}{1+\nu}}R_0^2}{\varepsilon^{\frac{2}{1+\nu}}},\max\left\{\frac{M_\nu R_0^{1+\nu}}{\varepsilon},\frac{\sigma^2R_0^2}{\varepsilon^2}\right\}\ln\frac{M_{\nu}R_0^{1+\nu}}{\varepsilon\beta}\right\}\right) \label{eq:clipped_SGD_oracle_complexity}
    \end{equation}
    oracle calls.
\end{theorem}

\begin{proof}
    Since $f(x)$ is convex and its gradients satisfy \eqref{eq:holder_def}, we get the following inequality under assumption that $x^k \in B_{7R_0}(x^*)$:
    \begin{eqnarray*}
        &&\|x^{k+1} - x^*\|_2^2 = \|x^k - \gamma \tnabla f(x^k,\Bxi^k) - x^*\|_2^2\\
        &=&
        \|x^k - x^*\|^2_2 + \gamma^2\|\tnabla f(x^k,\Bxi^k)\|^2_2 - 2\gamma \left\la x^k-x^*, \tnabla f(x^k,\Bxi^k)\right\ra\\
        &\overset{\revision{\eqref{eq:theta_k+1_def_clipped_SSTM}}}{=}& 
        \|x^k-x^*\|^2_2 + \gamma^2 \|\nabla f(x^k) + \theta_k\|^2_2 - 2\gamma \left\la x^k-x^*, \nabla f(x^k) + \theta_k \right\ra\\
        &\overset{\eqref{eq:squared_norm_sum}}{\le}& \|x^k-x^*\|^2_2 + 2 \gamma^2 \|\nabla f(x^k)\|^2_2 + 2\gamma^2 \|\theta_k\|^2_2 - 2\gamma \left\la x^k-x^*, \nabla f(x^k) + \theta_k\right\ra\\
        &\overset{\eqref{eq:holder_gradient_bound_2}}{\le}& \|x^k-x^*\|^2_2 - 2 \gamma \left( 1 - 2 \gamma \left(\frac{1}{\varepsilon}\right)^{\frac{1-\nu}{1+\nu}}M_{\nu}^{\frac{2}{1+\nu}}\right) \left( f(x^k) - f(x^*) \right) + 2\gamma^2 \|\theta_k\|^2_2\\
        &&\quad- 2\gamma \left\la x^k-x^*, \theta_k\right\ra + 2 \gamma^2 \varepsilon^{\frac{2\nu}{1+\nu}}M_{\nu}^{\frac{2}{1+\nu}},
    \end{eqnarray*}
    where $\theta_k = \tnabla f(x^k, \Bxi^k) - \nabla f(x^k)$ and the last inequality follows from the convexity of $f$. Using notation $R_k \eqdef \|x^k - x^*\|_2$, $k > 0$ we derive that for all $k\ge 0$
    \begin{eqnarray*}
          R_{k+1}^2 &\le& R_k^2 - 2 \gamma \left( 1 - 2 \gamma \left(\frac{1}{\varepsilon}\right)^{\frac{1-\nu}{1+\nu}}M_{\nu}^{\frac{2}{1+\nu}}\right) \left( f(x^k) - f(x^*) \right) + 2\gamma^2 \|\theta_k\|^2_2\notag\\
          &&\quad - 2\gamma \left\la x^k-x^*, \theta_k\right\ra + 2 \gamma^2 \varepsilon^{\frac{2\nu}{1+\nu}}M_{\nu}^{\frac{2}{1+\nu}}
    \end{eqnarray*}
    under assumption that $x^k \in B_{7R_0}(x^*)$. Let us define $A = 2 \gamma \left( 1 - 2 \gamma \left(\frac{1}{\varepsilon}\right)^{\frac{1-\nu}{1+\nu}}M_{\nu}^{\frac{2}{1+\nu}}\right) \overset{\eqref{eq:step_size_clipped_SGD}}{\ge} \revision{2\gamma \left(1 - 2\cdot \frac{\varepsilon^{\frac{1-\nu}{1+\nu}}}{8M_{\nu}^{\frac{2}{1+\nu}}}\cdot \left(\frac{1}{\varepsilon}\right)^{\frac{1-\nu}{1+\nu}}M_{\nu}^{\frac{2}{1+\nu}} \right) = \frac{3}{2}\gamma \geq}~\gamma > 0$, then
    \begin{equation*}
        A \left( f(x^k) - f(x^*) \right) \le R_k^2 - R_{k+1}^2 + 2\gamma^2 \|\theta_k\|^2_2 - 2\gamma \left\la x^k-x^*, \theta_k\right\ra + 2 \gamma^2 \varepsilon^{\frac{2\nu}{1+\nu}}M_{\nu}^{\frac{2}{1+\nu}}
    \end{equation*}
    under assumption that $x^k \in B_{7R_0}(x^*)$. Summing up these inequalities for \revision{$k=0,1, \ldots, N-1 $}, we obtain
    \revision{\begin{eqnarray*}
        &&\frac{A}{N} \sum\limits_{k=0}^{N-1}\left[ f(x^k) - f(x^*) \right]\\
        &=& \frac{1}{N}\left( R_0^2 - R_{N}^2\right) + 2 \gamma^2 \varepsilon^{\frac{2\nu}{1+\nu}}M_{\nu}^{\frac{2}{1+\nu}} +  \frac{2\gamma^2}{N} \sum\limits_{k=0}^{N-1} \|\theta_k\|^2_2 - \frac{2\gamma}{N} \sum\limits_{k=0}^{N-1} \left\la x^k-x^*, \theta_k \right\ra
    \end{eqnarray*}}
    under assumption that $x^k \in B_{7R_0}(x^*)$. Noticing that for $\Bar{x}^N = \frac{1}{N}\sum\limits_{k=0}^{N-1}x^k$ Jensen's inequality gives $f(\Bar{x}^N) = f\left(\frac{1}{N}\sum\limits_{k=0}^{N-1}x^k\right) \le \frac{1}{N}\sum\limits_{k=0}^{N-1}f(x^k)$, we have
    \begin{eqnarray}
        A N \left( f(\Bar{x}^N) - f(x^*) \right) &\le& R_0^2 - R_{N}^2 + 2 \gamma^2 N \varepsilon^{\frac{2\nu}{1+\nu}}M_{\nu}^{\frac{2}{1+\nu}} + 2\gamma^2 \sum\limits_{k=0}^{N-1}\|\theta_k\|^2_2\notag\\
        &&\quad - 2\gamma \sum\limits_{k=0}^{N-1} \left\la x^k-x^*, \theta_k \right\ra
        \label{eq:main_thm_clipped_SGD_technical_0}
    \end{eqnarray}
    under assumption that $x^k \in B_{7R_0}(x^*)$ for $k = 0,1,\ldots,N-1$. Taking into account that $f(\Bar{x}^N) - f(x^*) \ge 0$ and changing the indices we get that for all $k = 0,1,\ldots,N$
    \begin{equation}
         R_k^2 \le R_{0}^2 + 2 \gamma^2 k \varepsilon^{\frac{2\nu}{1+\nu}}M_{\nu}^{\frac{2}{1+\nu}} + 2\gamma^2 \sum\limits_{l=0}^{k-1}\|\theta_l\|^2_2 - 2\gamma \sum\limits_{l=0}^{k-1} \left\la x^l-x^*, \revision{\theta_l} \right\ra.
         \label{eq:main_thm_clipped_SGD_technical_1}
    \end{equation}
    under assumption that $x^l \in B_{7R_0}(x^*)$ for $l = 0,1,\ldots,k-1$. The remaining part of the proof is based on the analysis of inequality \eqref{eq:main_thm_clipped_SGD_technical_1}. In particular, via induction we prove that for all $k=0,1,\ldots, N$ with probability at least $1 - \frac{k\beta}{N}$ \revision{we have $\PP\{E_k\} = 1 - \frac{k\beta}{N}$ for probability event $E_k$ defined as follows:}

    \revision{\begin{center}
        Event $E_k$:
    \end{center}
    \noindent\fbox{%
    \parbox{\textwidth}{%
    Inequalities
    \revision{\begin{eqnarray}
         R_t^2 &\overset{\eqref{eq:main_thm_clipped_SGD_technical_1}}{\le}& R_{0}^2 + 2 \gamma^2 t \varepsilon^{\frac{2\nu}{1+\nu}}M_{\nu}^{\frac{2}{1+\nu}} + 2\gamma^2 \sum\limits_{l=0}^{t-1}\|\theta_l\|^2_2 - 2\gamma \sum\limits_{l=0}^{t-1} \left\la x^l-x^*, \theta_l \right\ra \notag \\
         &\le& C^2R_0^2 \label{eq:main_thm_clipped_SGD_technical_2}
    \end{eqnarray}}
    hold for $t=0,1,\ldots,k$ simultaneously where $C$ is defined in \eqref{eq:C_definition_clipped_SGD}.
    }
    }}
    
    For $t = 0$ inequality \eqref{eq:main_thm_clipped_SGD_technical_2} holds with probability $1$ since $C \ge 1$. Next, assume that for some $k = T-1 \le N-1$ we have $\PP\{E_k\} = \PP\{E_{T-1}\} \ge 1 - \frac{(T-1)\beta}{N}$. Let us prove that $\PP\{E_{T}\} \ge 1 - \frac{T\beta}{N}$. First of all, probability event $E_{T-1}$ implies that $x^t \in B_{7R_0}(x^*)$ for $t=0,1,\ldots,T-1$, and, as a consequence, \eqref{eq:main_thm_clipped_SGD_technical_1} holds for $k = T$. Since $\nabla f(x)$ is $(\nu, M_\nu)$-H\"older continuous on $B_{7R_0}(x^*)$, we have that probability event $E_{T-1}$ implies
    \begin{eqnarray}
        \left\|\nabla f(x^{t})\right\|_2 \overset{\eqref{eq:holder_def}}{\le} M_{\nu} \|x^t-x^0\|^{\nu} \overset{\revision{\eqref{eq:main_thm_clipped_SGD_technical_2}}}\le M_{\nu}C^{\nu}R_0^{\nu} \overset{\eqref{eq:step_size_clipped_SGD}}{\le} \frac{\lambda}{2} \label{eq:main_thm_clipped_SGD_technical_4}
    \end{eqnarray}
    for \revision{$t=0,1,\ldots,T-1$}. Next, we introduce new random variables:
    \begin{equation}
        \eta_l = \begin{cases}x^* - x^l,&\text{if } \|x^* - x^l\|_2 \le CR_0,\\ 0,&\text{otherwise,} \end{cases}\quad \label{eq:main_thm_clipped_SGD_technical_4_1}
    \end{equation}
    for $l=0,1,\ldots, T-1$. Note that these random variables are bounded with probability $1$, i.e.\ with probability $1$, we have
    \begin{equation}
        \|\eta_l\|_2 \le CR_0.\label{eq:main_thm_clipped_SGD_technical_4_2}
    \end{equation}
    Using the introduced notation, we obtain that $E_{T-1}$ implies 
    \revision{\begin{eqnarray*}
         R_T^2 &\overset{\eqref{eq:main_thm_clipped_SGD_technical_2}, \eqref{eq:step_size_clipped_SGD}}{\leq}& R_{0}^2 + 2 \left(\frac{R_0}{\sqrt{2N}\varepsilon^{\frac{\nu}{1+\nu}}M_{\nu}^{\frac{1}{1+\nu}}}\right)^2 N \varepsilon^{\frac{2\nu}{1+\nu}}M_{\nu}^{\frac{2}{1+\nu}}\\
         &&\quad + 2\gamma^2 \sum\limits_{l=0}^{t-1}\|\theta_l\|^2_2 - 2\gamma \sum\limits_{l=0}^{t-1} \left\la x^l-x^*, \theta_l \right\ra\\
         &=& 2R_0^2 + 2\gamma^2 \sum\limits_{l=0}^{t-1}\|\theta_l\|^2_2 - 2\gamma \sum\limits_{l=0}^{t-1} \left\la x^l-x^*, \theta_l \right\ra\\
         &\overset{\eqref{eq:main_thm_clipped_SGD_technical_4_1},\eqref{eq:main_thm_clipped_SGD_technical_4_2}}{=}& 2R_0^2 + 2\gamma \sum\limits_{l=0}^{T-1}\left\la\theta_{l}, \eta_l\right\ra + 2\gamma^2\sum\limits_{l=0}^{T-1}\|\theta_{l}\|_2^2. \notag
    \end{eqnarray*}}
    Finally, we do some preliminaries in order to apply Bernstein's inequality (see Lemma~\ref{lem:Bernstein_ineq}) and obtain that $E_{T-1}$ implies
    \begin{eqnarray}
        R_T^2 &\overset{\eqref{eq:squared_norm_sum}}{\le}& 2R_0^2 + \underbrace{2\gamma\sum\limits_{l=0}^{T-1}\left\la\theta_{l}^u, \eta_l\right\ra}_{\circledOne}+ \underbrace{2\gamma\sum\limits_{l=0}^{T-1}\left\la\theta_{l}^b, \eta_l \right\ra}_{\circledTwo}\notag + \underbrace{4\gamma^2\sum\limits_{l=0}^{T-1}\left(\|\theta_{l}^u\|_2^2 - \EE_{\Bxi^l}\left[\|\theta_{l}^u\|_2^2\right]\right)}_{\circledThree}\\
        &&\quad  + \underbrace{4\gamma^2\sum\limits_{l=0}^{T-1}\EE_{\Bxi^l}\left[\|\theta_{l}^u\|_2^2\right]}_{\circledFour} + \underbrace{4\gamma^2\sum\limits_{l=0}^{T-1}\|\theta_{l}^b\|_2^2}_{\circledFive},\label{eq:main_thm_clipped_SGD_technical_5}
    \end{eqnarray}
    where we introduce new notations:
    \begin{equation}
        \theta_{l}^u \eqdef \tnabla f(x^{l},\Bxi^l) - \EE_{\Bxi^l}\left[\tnabla f(x^{l},\Bxi^l)\right],\quad \theta_{l}^b \eqdef \EE_{\Bxi^l}\left[\tnabla f(x^{l},\Bxi^l)\right] - \nabla f(x^{l}),\label{eq:main_thm_clipped_SGD_technical_6}
    \end{equation}
    \begin{equation}
        \theta_{l} = \theta_{l}^u +\theta_{l}^b.\notag
    \end{equation}
    It remains to provide tight upper bounds for $\circledOne$, $\circledTwo$, $\circledThree$, $\circledFour$ and $\circledFive$, i.e.\ in the remaining part of the proof we show that $\circledOne+\circledTwo+\circledThree+\circledFour+\circledFive \le \delta C^2R_0^2$ for some $\delta < 1$.
    
    
    \textbf{Upper bound for $\circledOne$.} First of all, since $\EE_{\Bxi^l}[\theta_{l}^u] = 0$ summands in $\circledOne$ are conditionally unbiased:
        $\EE_{\Bxi^l}\left[2\gamma \left\la \theta_l^u, \eta_l\right\ra\right] = 0$.
    Secondly, these summands are bounded with probability $1$:
        $\left|2\gamma \left\la \theta_l^u, \eta_l\right\ra\right| \le 2\gamma\|\theta_{l}^u\|_2\left\|\eta_l\right\|_2 \overset{\eqref{eq:magnitude_bound_clipped_SSTM},\eqref{eq:main_thm_clipped_SGD_technical_4_2}}{\le} 4\gamma\lambda CR_0$.
    Finally, one can bound conditional variances $\sigma_l^2 \eqdef \EE_{\Bxi^l}\left[4\gamma^2\left\la\theta_{l}^u, \eta_l\right\ra^2\right]$ in the following way:
         $\sigma_l^2 \le \EE_{\Bxi^l}\left[4\gamma^2\left\|\theta_{l}^u\right\|_2^2 \left\|\eta_l\right\|_2^2\right] \overset{\eqref{eq:main_thm_clipped_SGD_technical_4_2}}{\le} 4\gamma^2(CR_0)^2\EE_{\Bxi^l}\left[\left\|\theta_{l}^u\right\|_2^2\right]$,
    i.e., $\sigma_l^2$ is finite due to finiteness of $\|\theta_{l+1}^u\|_2$ (see Lemma~\ref{lem:main_stoch_lemma_clipped_SSTM}). In other words, sequence $\left\{2\gamma \left\la \theta_l^u, \eta_l\right\ra\right\}_{l\ge 0}$ is a bounded martingale difference sequence with bounded conditional variances $\{\sigma_l^2\}_{l\ge 0}$.  Therefore, we can apply Bernstein's inequality, i.e., we apply Lemma~\ref{lem:Bernstein_ineq} with $X_l = 2\gamma \left\la \theta_l^u, \eta_l\right\ra$, $c = 4\gamma\lambda CR_0$ and $F = \frac{c^2\ln\frac{4N}{\beta}}{6}$ and get that for all $b > 0$
    \begin{equation*}
        \PP\left\{\left|\sum\limits_{l=0}^{T-1}X_l\right| > b\text{ and } \sum\limits_{l=0}^{T-1}\sigma_l^2 \le F\right\} \le 2\exp\left(-\frac{b^2}{2F + \nicefrac{2cb}{3}}\right)
    \end{equation*}
    or, equivalently, with probability at least $1 - 2\exp\left(-\frac{b^2}{2F + \nicefrac{2cb}{3}}\right)$
    \begin{equation*}
        \text{either } \sum\limits_{l=0}^{T-1}\sigma_l^2 > F \quad \text{or} \quad \underbrace{\left|\sum\limits_{l=0}^{T-1}X_l\right|}_{|\circledOne|} \le b.
    \end{equation*}
    The choice of $F$ will be clarified \revision{further.} Let us now choose $b$ in such a way that $2\exp\left(-\frac{b^2}{2F + \nicefrac{2cb}{3}}\right) = \frac{\beta}{2N}$. This implies that $b$ is the positive root of the quadratic equation \begin{equation*}
        b^2 - \frac{2c\ln\frac{4N}{\beta}}{3}b - 2F\ln\frac{4N}{\beta} = 0,
    \end{equation*}
    hence
    \begin{eqnarray*}
         b &=& \frac{c\ln\frac{4N}{\beta}}{3} + \sqrt{\frac{c^2\ln^2\frac{4N}{\beta}}{9} + 2F\ln\frac{4N}{\beta}}=\frac{c\ln\frac{4N}{\beta}}{3} + \sqrt{\frac{4c^2\ln^2\frac{4N}{\beta}}{9}}\\
         &=& c\ln\frac{4N}{\beta} = 4\gamma\lambda CR_0\ln\frac{4N}{\beta}.
    \end{eqnarray*}
    That is, with probability at least $1 - \frac{\beta}{2N}$
     \begin{equation*}
        \underbrace{\text{either } \sum\limits_{l=0}^{T-1}\sigma_l^2 > F \quad \text{or} \quad \left|\circledOne\right| \le 4\gamma \lambda CR_0\ln\frac{4N}{\beta}}_{\text{probability event } E_{\circledOne}}.
    \end{equation*}
    Here and below, we notice that the conditions of Lemma~\ref{lem:main_stoch_lemma_clipped_SSTM} hold when $E_{T-1}$ holds, since event $E_{T-1}$ implies that $x^0, x^1, \ldots, x^T$ lie in $B_{7R_0}(x^*)$. Therefore, probability event $E_{T-1}$ implies that
    \begin{eqnarray*}
        \sum\limits_{l=0}^{T-1}\sigma_l^2 &\le& 4\gamma^2(CR_0)^2\sum\limits_{l=0}^{T-1}\EE_{\Bxi^l}\left[\|\theta_l^u\|_2^2\right] \overset{\eqref{eq:variance_bound_clipped_SSTM}}{\le} 72\gamma^2 (CR_0)^2 \sigma^2 \frac{T}{m}\\
        &\overset{T\le N}{\le}& 72\gamma^2 (CR_0)^2 \sigma^2 \frac{N}{m} \le \frac{c^2\ln\frac{4N}{\beta}}{6} = F,
    \end{eqnarray*}
    where the last inequality follows from $c = 4\gamma\lambda CR_0$ and simple arithmetic.
    
    
    \textbf{Upper bound for $\circledTwo$.} First of all, we notice that probability event $E_{T-1}$ implies
     \begin{eqnarray*}
        2\gamma\left\la\theta_{l}^b, \eta_l \right\ra &\le& 2\gamma\left\|\theta_{l}^b\right\|_2\left\|\eta_l\right\|_2\overset{\eqref{eq:bias_bound_clipped_SSTM},\eqref{eq:main_thm_clipped_SGD_technical_4_2}}{\le} 2\gamma\frac{4\sigma^2}{m \lambda}CR_0 = \frac{8\gamma \sigma^2 CR_0}{m \lambda}.
    \end{eqnarray*}
    This implies that
    \begin{eqnarray*}
        \circledTwo &=& 2\gamma\sum\limits_{l=0}^{T-1}\left\la\theta_{l}^b,\eta_l\right\ra \overset{T\le N}{\le} \frac{8\gamma \sigma^2 CR_0 N}{m \lambda}\overset{\eqref{eq:bathces_clipped_SGD}}{\le} \frac{8}{81} \lambda \gamma CR_0 \ln\frac{4N}{\beta}.
    \end{eqnarray*}
    
    
     \textbf{Upper bound for $\circledThree$.} We derive the upper bound for $\circledThree$ using the same technique as for $\circledOne$. First of all, we notice that the summands in $\circledThree$ are conditionally unbiased:
        $\EE_{\Bxi^l}\left[4\gamma^2\left(\|\theta_{l}^u\|_2^2 - \EE_{\Bxi^l}\left[\|\theta_{l}^u\|_2^2\right]\right)\right] = 0$.
    \revision{Second}, the summands are bounded with probability $1$:
    \begin{eqnarray}
        \left|4\gamma^2\left(\|\theta_{l}^u\|_2^2 - \EE_{\Bxi^l}\left[\|\theta_{l}^u\|_2^2\right]\right)\right| &\le& 4\gamma^2\left(\|\theta_{l}^u\|_2^2 + \EE_{\Bxi^l}\left[\|\theta_{l}^u\|_2^2\right]\right)\overset{\eqref{eq:magnitude_bound_clipped_SSTM}}{\le} 4\gamma^2\left(4\lambda^2 + 4\lambda^2\right)\notag\\
        &=& 32 \gamma^2 \lambda^2  \eqdef c_1.\label{eq:main_thm_clipped_SGD_technical_8}
    \end{eqnarray}
    Finally, one can bound conditional variances \\$\hat \sigma_l^2 \eqdef \EE_{\Bxi^l}\left[\left|4\gamma^2\left(\|\theta_{l}^u\|_2^2 - \EE_{\Bxi^l}\left[\|\theta_{l}^u\|_2^2\right]\right)\right|^2\right]$ in the following way:
    \begin{eqnarray}
         \hat \sigma_l^2 &\overset{\eqref{eq:main_thm_clipped_SGD_technical_8}}{\le}& c_1\EE_{\Bxi^l}\left[\left|4\gamma^2\left(\|\theta_{l}^u\|_2^2 - \EE_{\Bxi^l}\left[\|\theta_{l}^u\|_2^2\right]\right)\right|\right]\notag\\
         &\le& 4\gamma^2c_1\EE_{\Bxi^l}\left[\|\theta_{l}^u\|_2^2 + \EE_{\Bxi^l}\left[\|\theta_{l}^u\|_2^2\right]\right] = 8\gamma^2c_1\EE_{\Bxi^l}\left[\|\theta_{l}^u\|_2^2\right],\label{eq:main_thm_clipped_SGD_technical_9}
    \end{eqnarray}
    i.e., $\hat \sigma_l^2$ is finite due to finiteness of $\|\theta_{l+1}^u\|_2$ (see Lemma~\ref{lem:main_stoch_lemma_clipped_SSTM})\revision{.} In other words, sequence $\left\{4\gamma^2\left(\|\theta_{l}^u\|_2^2 - \EE_{\Bxi^l}\left[\|\theta_{l}^u\|_2^2\right]\right)\right\}_{l\ge 0}$ is a bounded martingale difference sequence with bounded conditional variances $\{\hat \sigma_l^2\}_{l\ge 0}$. Therefore, we can apply Bernstein's inequality, i.e.\ we apply Lemma~\ref{lem:Bernstein_ineq} with $X_l = \hat X_l = 4\gamma^2\left(\|\theta_{l}^u\|_2^2 - \EE_{\Bxi^l}\left[\|\theta_{l}^u\|_2^2\right]\right)$, $c = c_1 = 32 \gamma^2 \lambda^2$ and $F = F_1 = \frac{c_1^2\ln\frac{4N}{\beta}}{18}$ and get that for all $b > 0$
    \begin{equation*}
        \PP\left\{\left|\sum\limits_{l=0}^{T-1}\hat X_l\right| > b\text{ and } \sum\limits_{l=0}^{T-1}\hat \sigma_l^2 \le F_1\right\} \le 2\exp\left(-\frac{b^2}{2F_1 + \nicefrac{2c_1b}{3}}\right)
    \end{equation*}
    or, equivalently, with probability at least $1 - 2\exp\left(-\frac{b^2}{2F_1 + \nicefrac{2c_1b}{3}}\right)$
    \begin{equation*}
        \text{either } \sum\limits_{l=0}^{T-1}\hat\sigma_l^2 > F_1 \quad \text{or} \quad \underbrace{\left|\sum\limits_{l=0}^{T-1}\hat X_l\right|}_{|\circledThree|} \le b.
    \end{equation*}
    As in our derivations of the upper bound for $\circledOne$ we choose such $b$ that $2\exp\left(-\frac{b^2}{2F_1 + \nicefrac{2c_1b}{3}}\right) = \frac{\beta}{2N}$, i.e.,
    \begin{eqnarray*}
         b &=& \frac{c_1\ln\frac{4N}{\beta}}{3} + \sqrt{\frac{c_1^2\ln^2\frac{4N}{\beta}}{9} + 2F_1\ln\frac{4N}{\beta}} \le c_1\ln\frac{4N}{\beta} = 32 \gamma^2 \lambda^2 \ln\frac{4N}{\beta} .
    \end{eqnarray*}
    That is, with probability at least $1 - \frac{\beta}{2N}$
     \begin{equation*}
        \underbrace{\text{either } \sum\limits_{l=0}^{T-1}\hat\sigma_l^2 > F_1 \quad \text{or} \quad \left|\circledThree\right| \le 32 \gamma^2 \lambda^2 \ln\frac{4N}{\beta}}_{\text{probability event } E_{\circledThree}}.
    \end{equation*}
    Next, we notice that probability event $E_{T-1}$ implies that
    \begin{eqnarray*}
        \sum\limits_{l=0}^{T-1}\hat\sigma_l^2 &\overset{\eqref{eq:main_thm_clipped_SGD_technical_9}}{\le}& 8\gamma^2 c_1\sum\limits_{l=0}^{T-1}\EE_{\Bxi^l}\left[\left\|\theta_{l}^u\right\|_2^2\right]\overset{\eqref{eq:variance_bound_clipped_SSTM}}{\le} 144\gamma^2 c_1 \sigma^2\frac{T}{m}\\
        &\overset{T\le N}{\le}& 
        144\gamma^2 c_1 \sigma^2\frac{N}{m} \revision{\overset{\eqref{eq:bathces_clipped_SGD}}{\leq} \frac{144\gamma^2 \lambda^2 c_1 \ln \frac{4N}{\beta}}{81} = \frac{c_1^2\ln\frac{4N}{\beta}}{18} =}~ F_1.
    \end{eqnarray*}
    
    
     \textbf{Upper bound for $\circledFour$.} The probability event $E_{T-1}$ implies
    \begin{eqnarray*}
        \circledFour &=& 4\gamma^2\sum\limits_{l=0}^{T-1}\EE_{\Bxi^l}\left[\|\theta_{l}^u\|_2^2\right] \overset{\eqref{eq:variance_bound_clipped_SSTM}}{\le} 72\gamma^2\sigma^2\sum\limits_{l=0}^{T-1}\frac{1}{m} \overset{T\le N}{\le} \frac{72\gamma^2N\sigma^2}{m}\overset{\eqref{eq:bathces_clipped_SGD}}{\le} \frac{8}{9} \lambda^2 \gamma^2 \ln \frac{4N}{\beta}.
    \end{eqnarray*}
    
    
    \textbf{Upper bound for $\circledFive$.} Again, we use corollaries of probability event $E_{T-1}$:
    \begin{eqnarray*}
        \circledFive &=& 4\gamma^2\sum\limits_{l=0}^{T-1}\|\theta_{l}^b\|_2^2 \overset{\eqref{eq:bias_bound_clipped_SSTM}}{\le} 64 \gamma^2 \sigma^4\frac{T}{m^2\lambda^2} \overset{T\le N}{\le} 64 \gamma^2 \sigma^4\frac{N}{m^2\lambda^2}\overset{\eqref{eq:bathces_clipped_SGD}}{\le} \frac{64}{6561} \frac{\lambda^2 \gamma^2 \ln^2 \frac{4N}{\beta}}{N}.
    \end{eqnarray*}

    Now we summarize all bounds that we have: probability event $E_{T-1}$ implies
    \revision{\begin{gather*}
         R_T^2 
        \overset{\eqref{eq:main_thm_clipped_SGD_technical_5}}{\le} 2R_0^2 + \circledOne + \circledTwo + \circledThree + \circledFour + \circledFive,\\
        \circledTwo \le \frac{8}{81} \lambda \gamma CR_0 \ln\frac{4N}{\beta},\quad \circledFour \le \frac{8}{9} \lambda^2 \gamma^2 \ln \frac{4N}{\beta},\quad \circledFive \le \frac{64}{6561} \frac{\lambda^2 \gamma^2 \ln^2 \frac{4N}{\beta}}{N},\\
        \sum\limits_{l=0}^{T-1}\sigma_l^2 \le F,\quad \sum\limits_{l=0}^{T-1}\hat\sigma_l^2 \le F_1
    \end{gather*}}
    and        $\PP\{E_{T-1}\} \ge 1 - \frac{(T-1)\beta}{N},\quad \PP\{E_\circledOne\} \ge 1 - \frac{\beta}{2N},\quad \PP\{E_\circledThree\} \ge 1 - \frac{\beta}{2N}$,
    \begin{eqnarray*}
        \text{where} \;\; E_{\circledOne} &=& \left\{\text{either } \sum\limits_{l=0}^{T-1}\sigma_l^2 > F \quad \text{or} \quad \left|\circledOne\right| \le 4\gamma\lambda CR_0\ln\frac{4N}{\beta}\right\},\\
        E_{\circledThree} &=& \left\{\text{either } \sum\limits_{l=0}^{T-1}\hat\sigma_l^2 > F_1 \quad \text{or} \quad \left|\circledThree\right| \le 32 \gamma^2 \lambda^2 \ln\frac{4N}{\beta}\right\}.
    \end{eqnarray*}
    Taking into account these inequalities and our assumptions on $\lambda$ and $\gamma$ (see \eqref{eq:bathces_clipped_SGD} and \eqref{eq:step_size_clipped_SGD}) we get that probability event $E_{T-1}\cap E_\circledOne \cap E_\circledThree$ implies
   \revision{\begin{eqnarray}
         R_T^2 
         &\le& 2R_0^2 + \left(\frac{4}{7} + \frac{8}{567} + \frac{16}{49} + \frac{4}{441} + \frac{64}{321489}\right)C^2R_0^2 \overset{\eqref{eq:C_definition_clipped_SGD}}{\le} C^2R_0^2.\label{eq:main_thm_clipped_SGD_technical_10}
    \end{eqnarray}}
    Moreover, \revision{using the bound for the union,} we derive
    \begin{equation}
        \PP\left\{E_{T-1}\cap E_\circledOne \cap E_\circledThree\right\} = 1 - \PP\left\{\overline{E}_{T-1}\cup\overline{E}_\circledOne\cup\overline{E}_\circledThree\right\} \ge 1 - \frac{T\beta}{N}.\label{eq:main_thm_clipped_SGD_technical_11}
    \end{equation}
    That is, by definition of $E_T$ and $E_{T-1}$ we have \revision{proven} that
    \begin{eqnarray*}
        \PP\{E_T\} &\overset{\eqref{eq:main_thm_clipped_SGD_technical_10}}{\ge}& \PP\left\{E_{T-1}\cap E_\circledOne \cap E_\circledThree\right\} \overset{\eqref{eq:main_thm_clipped_SGD_technical_11}}{\ge} 1 - \frac{T\beta}{N},
    \end{eqnarray*}
    which implies that for all $k = 0,1,\ldots, N$ we have $\PP\{E_k\} \ge 1 - \frac{k\beta}{N}$. Then, for $k = N$ we have that with probability at least $1-\beta$
    \revision{\begin{eqnarray*}
        AN\left(f(\Bar{x}^N) - f(x^*)\right) &\overset{\eqref{eq:main_thm_clipped_SGD_technical_0}, \eqref{eq:step_size_clipped_SGD}}{\le}& R_0^2 + 2 \left(\frac{R_0}{\sqrt{2N}\varepsilon^{\frac{\nu}{1+\nu}}M_{\nu}^{\frac{1}{1+\nu}}}\right)^2 N \varepsilon^{\frac{2\nu}{1+\nu}}M_{\nu}^{\frac{2}{1+\nu}}\\
        &&\quad +  2\gamma^2 \sum\limits_{k=0}^{N-1}\|\theta_k\|^2_2 - 2\gamma \sum\limits_{k=0}^{N-1} \left\la x^k-x^*, \theta_k \right\ra \\
        &\leq& 2R_0^2 +  2\gamma^2 \sum\limits_{k=0}^{N-1}\|\theta_k\|^2_2 - 2\gamma \sum\limits_{k=0}^{N-1} \left\la x^k-x^*, \theta_k \right\ra\\
        &\overset{\eqref{eq:main_thm_clipped_SGD_technical_2}}{\le}& C^2R_0^2.
    \end{eqnarray*}}
    Since $A = 2 \gamma \left( 1 - 2 \gamma \left(\frac{1}{\varepsilon}\right)^{\frac{1-\nu}{1+\nu}}M_{\nu}^{\frac{2}{1+\nu}}\right)\overset{\eqref{eq:step_size_clipped_SGD}}{\ge} \gamma$
    we get that with probability at least $1-\beta$, it holds that 
        $f(\Bar{x}^N) - f(x^*) \le \frac{C^2R_0^2}{AN} \revision{\leq} \frac{C^2R_0^2}{\gamma N}$.
    When 
    \begin{equation*}
        \gamma = \min \left\{\frac{\varepsilon^{\frac{1-\nu}{1+\nu}}}{8M_{\nu}^{\frac{2}{1+\nu}}},  \frac{R_0}{\sqrt{2N}\varepsilon^{\frac{\nu}{1+\nu}}M_{\nu}^{\frac{1}{1+\nu}}},  \frac{R_0^{1-\nu}}{2C^{\nu}M_{\nu}\ln\frac{4N}{\beta}}\right\}
    \end{equation*}
    we have that with probability at least $1-\beta$
    \begin{eqnarray*}
        &&\hspace{-2em}f(\Bar{x}^N) - f(x^*) \\
        &\le& \max\left\{\frac{8C^2M_\nu^{\frac{2}{1+\nu}}R_0^2}{\varepsilon^{\frac{1-\nu}{1+\nu}} N}, \frac{\sqrt{2}C^2M_\nu^{\frac{1}{1+\nu}}R_0\varepsilon^{\frac{\nu}{1+\nu}}}{\sqrt{N}}, \frac{2C^{2+\nu}M_\nu R_0^{1+\nu}\ln\frac{4N}{\beta}}{N}\right\}.
    \end{eqnarray*}
    Next, we estimate the iteration and oracle complexities of the method and consider $3$ possible situations.
    \begin{enumerate}
        \item If $\gamma = \frac{\varepsilon^{\frac{1-\nu}{1+\nu}}}{8M_{\nu}^{\frac{2}{1+\nu}}}$, then with probability at least $1-\beta$
        \begin{equation*}
            f(\Bar{x}^N) - f(x^*) \le \frac{8C^2M_\nu^{\frac{2}{1+\nu}}R_0^2}{\varepsilon^{\frac{1-\nu}{1+\nu}} N}.
        \end{equation*}
        In other words, \algname{clipped-SGD} achieves $f(\Bar{x}^N) - f(x^*) \le \varepsilon$ with probability at least $1-\beta$ after
     \begin{equation*}
         \cO\left(\frac{M_\nu^{\frac{2}{1+\nu}}R_0^2}{\varepsilon^{\frac{2}{1+\nu}}}\right)
     \end{equation*}
     iterations and requires
    \begin{eqnarray*}
         Nm &\overset{\eqref{eq:bathces_clipped_SGD}}{=}&  \cO\left(\max\left\{N,\frac{N^2\sigma^2\gamma^2\ln\frac{N}{\beta}}{R_0^2}\right\}\right) \\
         &=&  \cO\left(\max\left\{N,\frac{N^2\varepsilon^{\frac{2(1-\nu)}{1+\nu}}\sigma^2\ln\frac{N}{\beta}}{M_\nu^{\frac{4}{1+\nu}}R_0^2}\right\}\right)\\
         &=& \cO\left(\max\left\{\frac{M_\nu^{\frac{2}{1+\nu}}R_0^2}{\varepsilon^{\frac{2}{1+\nu}}},\frac{\sigma^2R_0^2}{\varepsilon^2}\ln\frac{M_\nu^{\frac{2}{1+\nu}}R_0^2}{\varepsilon^{\frac{2}{1+\nu}}\beta}\right\}\right)
    \end{eqnarray*}
    oracle calls.
    
    \item If $\gamma = \frac{R_0}{\sqrt{2N}\varepsilon^{\frac{\nu}{1+\nu}}M_{\nu}^{\frac{1}{1+\nu}}}$, then with probability at least $1-\beta$
        \begin{equation*}
            f(\Bar{x}^N) - f(x^*) \le \frac{\sqrt{2}C^2M_\nu^{\frac{1}{1+\nu}}R_0\varepsilon^{\frac{\nu}{1+\nu}}}{\sqrt{N}}.
        \end{equation*}
        In other words, \algname{clipped-SGD} achieves $f(\Bar{x}^N) - f(x^*) \le \varepsilon$ with probability at least $1-\beta$ after
     \begin{equation*}
         \cO\left(\frac{M_\nu^{\frac{2}{1+\nu}}R_0^2}{\varepsilon^{\frac{2}{1+\nu}}}\right)
     \end{equation*}
     iterations and requires
    \begin{eqnarray*}
         Nm &\overset{\eqref{eq:bathces_clipped_SGD}}{=}&  \cO\left(\max\left\{N,\frac{N^2\sigma^2\gamma^2\ln\frac{N}{\beta}}{R_0^2}\right\}\right) =  \cO\left(\max\left\{N,\frac{N\sigma^2\ln\frac{N}{\beta}}{\varepsilon^{\frac{2\nu}{1+\nu}} M_\nu^{\frac{2}{1+\nu}}}\right\}\right)\\
         &=& \cO\left(\max\left\{\frac{M_\nu^{\frac{2}{1+\nu}}R_0^2}{\varepsilon^{\frac{2}{1+\nu}}},\frac{\sigma^2R_0^2}{\varepsilon^2}\ln\frac{M_\nu^{\frac{2}{1+\nu}}R_0^2}{\varepsilon^{\frac{2}{1+\nu}}\beta}\right\}\right)
    \end{eqnarray*}
    oracle calls.
    
    \item If $\gamma =  \frac{R_0^{1-\nu}}{2C^{\nu}M_{\nu}\ln\frac{4N}{\beta}}$, then with probability at least $1-\beta$
        \begin{equation*}
            f(\Bar{x}^N) - f(x^*) \le \frac{2C^{2+\nu}M_\nu R_0^{1+\nu}\ln\frac{4N}{\beta}}{N}.
        \end{equation*}
        In other words, \algname{clipped-SGD} achieves $f(\Bar{x}^N) - f(x^*) \le \varepsilon$ with probability at least $1-\beta$ after
     \begin{equation*}
         \cO\left(\frac{M_\nu R_0^{1+\nu}\ln\frac{M_\nu R_0^{1+\nu}}{\varepsilon\beta}}{\varepsilon}\right)
     \end{equation*}
     iterations and requires
    \begin{eqnarray*}
         Nm &\overset{\eqref{eq:bathces_clipped_SGD}}{=}&  \cO\left(\max\left\{N,\frac{N^2\sigma^2\gamma^2\ln\frac{N}{\beta}}{R_0^2}\right\}\right) =  \cO\left(\max\left\{N,\frac{N^2\sigma^2}{M_\nu^2 R_0^{2\nu}\ln\frac{N}{\beta}}\right\}\right)\\
         &=& \cO\left(\max\left\{\frac{M_\nu R_0^{1+\nu}}{\varepsilon},\frac{\sigma^2R_0^2}{\varepsilon^2}\right\}\ln\frac{M_\nu R_0^{1+\nu}}{\varepsilon\beta}\right)
    \end{eqnarray*}
    oracle calls.
    \end{enumerate}
    Putting all together and noticing that $\ln\frac{M_\nu^{\frac{2}{1+\nu}}R_0^2}{\varepsilon^{\frac{2}{1+\nu}}\beta} = \cO\left(\ln\frac{M_\nu R_0^{1+\nu}}{\varepsilon\beta}\right)$ we get the desired result. \qed
\end{proof}

As for \algname{clipped-SSTM}, it is possible to get rid of using large batch sizes without sacrificing the oracle complexity via a proper choice of $\gamma$.

\begin{corollary}\label{cor:SGD_cvx_small_batch}
    Let the assumptions of Theorem~\ref{thm:main_result_clipped_SGD} hold and
    \begin{equation}
        \gamma = \min \left\{\frac{\varepsilon^{\frac{1-\nu}{1+\nu}}}{8M_{\nu}^{\frac{2}{1+\nu}}},  \frac{R_0}{\sqrt{2N}\varepsilon^{\frac{\nu}{1+\nu}}M_{\nu}^{\frac{1}{1+\nu}}},  \frac{R_0^{1-\nu}}{2C^{\nu}M_{\nu}\ln\frac{4N}{\beta}}, \frac{R_0}{9\sigma \sqrt{N \ln\frac{4N}{\beta}}}\right\}. \label{eq:refined_assumption_on_gamma_SGD}
    \end{equation}
    Then for all $k=0,1,\ldots,N-1$ one can use $m = 1$ and to achieve $f(\Bar{x}^N) - f(x^*) \le \varepsilon$ with probability at least $1-\beta$ \algname{clipped-SGD} requires
    \begin{equation}
        \cO\left(\max\left\{\frac{M_\nu^{\frac{2}{1+\nu}}R_0^2}{\varepsilon^{\frac{2}{1+\nu}}},\max\left\{\frac{M_\nu R_0^{1+\nu}}{\varepsilon},\frac{\sigma^2R_0^2}{\varepsilon^2}\right\}\ln\left(\frac{M_{\nu}R_0^{1+\nu}}{\varepsilon\beta} + \frac{\sigma^2 R_0^2}{\varepsilon^2\beta}\right)\right\}\right) \label{eq:SGD_complexity_small_batch_cvx}
    \end{equation}
    iterations/oracle calls.
\end{corollary}
\begin{proof}
    First of all, we verify that $m = 1$ is a valid choice. The only assumption on $m$ is given in \eqref{eq:bathces_clipped_SGD}:
        $m \overset{\eqref{eq:bathces_clipped_SGD}}{\geq} \max\left\{1, \frac{81N\sigma^2}{\lambda^2\ln \tfrac{4N}{\beta}}\right\}.$
    Since $\gamma \leq \tfrac{R_0}{9\sigma \sqrt{N \ln\frac{4N}{\beta}}}$, we have
    \revision{\begin{eqnarray*}
         \max\left\{1, \frac{81N\sigma^2}{\lambda^2\ln \tfrac{4N}{\beta}}\right\} &\overset{\eqref{eq:bathces_clipped_SGD}}{=}& \max\left\{1, \frac{81\gamma^2\sigma^2N\ln\tfrac{4N}{\beta}}{R_0^2}\right\} \leq 1
    \end{eqnarray*}}
    Therefore, for $\gamma$ given in \eqref{eq:refined_assumption_on_gamma_SGD} one can use $m = 1$.
    
    Next, if the minimum in \eqref{eq:refined_assumption_on_gamma_SGD} is attained on any of the first three terms, then applying the derivations from the end of the proof of Theorem~\ref{thm:main_result_clipped_SGD}, we get that the method requires 
    \begin{equation}
       \cO\left(\max\left\{\frac{M_\nu^{\frac{2}{1+\nu}}R_0^2}{\varepsilon^{\frac{2}{1+\nu}}},\max\left\{\frac{M_\nu R_0^{1+\nu}}{\varepsilon},\frac{\sigma^2R_0^2}{\varepsilon^2}\right\}\ln\frac{M_{\nu}R_0^{1+\nu}}{\varepsilon\beta}\right\}\right) \notag
    \end{equation}
    iterations/oracle calls to achieve $f(\Bar{x}^N) - f(x^*) \le \varepsilon$ with probability at least $1-\beta$. If  $\gamma = \tfrac{R_0}{9\sigma \sqrt{N \ln\frac{4N}{\beta}}}$, then with probability at least $1-\beta$
        \begin{equation*}
            f(\Bar{x}^N) - f(x^*) \overset{\eqref{eq:main_result_clipped_SGD}}{\le} \frac{9 C^2 R_0 \sigma \sqrt{\ln\tfrac{4N}{\beta}}}{\sqrt{N}}.
        \end{equation*}
        In other words, \algname{clipped-SGD} achieves $f(\Bar{x}^N) - f(x^*) \le \varepsilon$ with probability at least $1-\beta$ after\footnote{\revision{To achieve $f(\Bar{x}^N) - f(x^*) \le \varepsilon$ it is sufficient to take $N$ such that $\frac{9 C^2 R_0 \sigma \sqrt{\ln\tfrac{4N}{\beta}}}{\sqrt{N}} \leq \varepsilon$. Solving this inequality w.r.t.\ $N$, we get that it is sufficient to take $N$ such that $N \geq \frac{81C^4\sigma^2 R_0^2 \ln \frac{4N}{\beta}}{\varepsilon^2}$, e.g., $N = \left\lceil\frac{162C^4\sigma^2 R_0^2 \ln \left(\frac{648C^4\sigma^2 R_0^2}{\varepsilon^2\beta}\right)}{\varepsilon^2}\right\rceil$ satisfies this inequality.}}
     \begin{equation*}
         \cO\left(\frac{\sigma^2 R_0^2\ln\frac{\sigma^2 R_0^2}{\varepsilon^2\beta}}{\varepsilon^2}\right)
     \end{equation*}
     iterations/oracle calls. Putting all together, we get the desired result. \qed
\end{proof}

\subsection{Strongly Convex Case}\label{sec:clipped_SGD_str_cvx_appendix}
In this section, we provide a full statement of Theorem~\ref{thm:main_result_clipped_SGD_str_cvx_main} together with its proof.  Note that due to strong convexity, the solution $x^*$ is unique. 
\begin{theorem}\label{thm:main_result_clipped_SGD_str_cvx}
    Assume that function $f$ is $\mu$-strongly convex, its stochastic gradient and its gradient satisfy \eqref{eq:bounded_variance_clipped_SSTM} and \eqref{eq:holder_def} respectively with $\sigma > 0$, $\nu \in [0,1]$, $M_\nu > 0$ on $Q = B_{7R_0}(x^*)$, where $R_0 \ge \|x^0 - x^*\|_2$. Let $\varepsilon > 0$, $\beta\in (0,1)$, and for all $t = 1,\ldots, \tau$\revision{, where $\tau = \left\lceil\log_2\frac{\mu R_0^2}{\varepsilon} \right\rceil - 1$,}
    \begin{equation*}
        N_t = \max\left\{\frac{2C^4M_\nu^{\frac{2}{1+\nu}}R_0^2}{2^t\varepsilon_t^{\frac{2}{1+\nu}}}, \frac{4C^{2+\nu}M_\nu R_0^{1+\nu}\ln\frac{16C^{2+\nu}M_\nu R_0^{1+\nu}}{2^{\frac{(1+\nu)t}{2}}\varepsilon_t\beta}}{2^{\frac{(1+\nu)t}{2}}\varepsilon_t}\right\},\quad \varepsilon_t = \frac{\mu R_0^2}{2^{t+1}},
    \end{equation*}
    \begin{equation*}
         \lambda_t = \frac{R_0}{2^{\frac{t}{2}}\gamma_t \ln\frac{4N_t\tau}{\beta}} ,\quad m_t \ge \max\left\{1,\frac{81N_t\sigma^2}{\lambda_t^2\ln\frac{4N_t\tau}{\beta}}\right\},\quad \ln\frac{4N_t\tau}{\beta} \ge 2,
    \end{equation*}
    \begin{equation*}
        \gamma_t = \min \left\{\frac{\varepsilon_t^{\frac{1-\nu}{1+\nu}}}{8M_{\nu}^{\frac{2}{1+\nu}}}, \; \frac{R_0}{2^{\frac{t}{2}}\sqrt{2N_t}\varepsilon_t^{\frac{\nu}{1+\nu}}M_{\nu}^{\frac{1}{1+\nu}}}, \; \frac{R_0^{1-\nu}}{2^{1+\frac{(1-\nu)t}{2}}C^{\nu}M_{\nu}\ln\frac{4N_t\tau}{\beta}}\right\}.
    \end{equation*}
    Then \algname{R-clipped-SGD} achieves $f(\Bar{x}^\tau) - f(x^*) \le \varepsilon$ with probability at least $1-\beta$ after  
    \begin{equation}
        \cO\left(\max\left\{D_1^{\frac{2}{1+\nu}}\ln\frac{\mu R_0^2}{\varepsilon}, D_2^{\frac{2}{1+\nu}}, \max\left\{D_1\ln\frac{\mu R_0^2}{\varepsilon},D_2\right\}\ln\frac{D}{\beta} \right\}\right) \notag
    \end{equation}
    iterations of Algorithm~\ref{alg:clipped-SGD} in total and requires
    \begin{equation}
        \cO\left(\max\left\{D_1^{\frac{2}{1+\nu}}\ln\frac{\mu R_0^2}{\varepsilon}, D_2^{\frac{2}{1+\nu}}, \max\left\{D_1\ln\frac{\mu R_0^2}{\varepsilon},D_2, \frac{\sigma^2}{\mu\varepsilon}\right\}\ln\frac{D}{\beta} \right\}\right) \label{eq:clipped_SGD_oracle_complexity_str_cvx}
    \end{equation}
    oracle calls, where
    \begin{equation*}
        D_1 = \frac{M_\nu}{\mu R_0^{1-\nu}},\quad D_2 = \frac{M_\nu}{\mu^{\frac{1+\nu}{2}}\varepsilon^{\frac{1-\nu}{2}}},\quad D = D_2\ln\frac{\mu R_0^2}{\varepsilon}.
    \end{equation*}
\end{theorem}
\begin{proof}
    Applying Theorem~\ref{thm:main_result_clipped_SGD}, we obtain that with probability at least $1-\frac{\beta}{\tau}$ it holds that 
        $f(\hat{x}^1) - f(x^*) \le \frac{\mu R_0^2}{4}$.
    Since $f$ is $\mu$-strongly convex we have
        $\frac{\mu \|\hat{x}^1 - x^*\|_2^2}{2} \le f(\hat{x}^1) - f(x^*)$.
    Therefore, with probability at least $1-\frac{\beta}{\tau}$
    \begin{equation*}
        f(\hat{x}^1) - f(x^*) \le \frac{\mu R_0^2}{4},\quad \|\hat{x}^1 - x^*\|_2^2 \le \frac{R_0^2}{2}.
    \end{equation*}
    From mathematical induction and the union bound for probability events, it follows that inequalities
        $f(\hat{x}^t) - f(x^*) \le \frac{\mu R_0^2}{2^{t+1}},\quad \|\hat{x}^t - x^*\|_2^2 \le \frac{R_0^2}{2^t}$
    hold simultaneously for $t = 1,\ldots, \tau$ with probability at least $1 - \beta$. In particular, it means that after $\tau = \left\lceil\log_2\frac{\mu R_0^2}{\varepsilon} \right\rceil - 1$ restarts \algname{R-clipped-SGD} finds an $\varepsilon$-solution with probability at least $1-\beta$. The total number of iterations $\hat N$ is
    \begin{eqnarray*}
        &&\sum\limits_{t=1}^\tau N_t = \cO\left(\sum\limits_{t=1}^\tau \max\left\{\frac{M_\nu^{\frac{2}{1+\nu}}R_0^2}{2^t\varepsilon_t^{\frac{2}{1+\nu}}},\frac{M_\nu R_0^{1+\nu}}{2^{\frac{(1+\nu)t}{2}}\varepsilon_t}\ln\frac{M_\nu R_0^{1+\nu}\tau}{2^{\frac{(1+\nu)t}{2}}\varepsilon_t\beta}\right\}\right)\\
        &=& \cO\left(\sum\limits_{t=1}^\tau \max\left\{\frac{M_\nu^{\frac{2}{1+\nu}}\cdot 2^{\frac{(1-\nu)t}{1+\nu}}}{\mu^{\frac{2}{1+\nu}}R_0^{\frac{2(1-\nu)}{1+\nu}}}, \frac{M_\nu\cdot 2^{\frac{(1-\nu)t}{2}}}{\mu R_0^{1-\nu}}\ln\frac{M_\nu\cdot 2^{\frac{(1-\nu)\tau}{2}}\tau}{\mu R_0^{1-\nu}\beta}\right\}\right)\\
        &=& \cO\left(\max\left\{\frac{M_\nu^{\frac{2}{1+\nu}}}{\mu^{\frac{2}{1+\nu}}R_0^{\frac{2(1-\nu)}{1+\nu}}}, \frac{M_\nu}{\mu R_0^{1-\nu}}\ln\frac{M_\nu \ln\frac{\mu R_0^2}{\varepsilon}}{\mu^{\frac{1+\nu}{2}} \varepsilon^{\frac{1-\nu}{2}}\beta}\right\}\cdot \left(\frac{\mu R_0^2}{\varepsilon}\right)^{\frac{1-\nu}{1+\nu}}\right)\\
        &=& \cO\left(\max\left\{D_1^{\frac{2}{1+\nu}}\ln\frac{\mu R_0^2}{\varepsilon}, D_2^{\frac{2}{1+\nu}}, \max\left\{D_1\ln\frac{\mu R_0^2}{\varepsilon},D_2\right\}\ln\frac{D}{\beta} \right\}\right),
    \end{eqnarray*}
    where 
    \begin{equation*}
        D_1 = \frac{M_\nu}{\mu R_0^{1-\nu}},\quad D_2 = \frac{M_\nu}{\mu^{\frac{1+\nu}{2}}\varepsilon^{\frac{1-\nu}{2}}},\quad D = D_2\ln\frac{\mu R_0^2}{\varepsilon}.
    \end{equation*}
    Finally, the total number of oracle calls equals
    \begin{eqnarray*}
         \sum\limits_{t=1}^\tau\sum\limits_{k=0}^{N_t-1}m_k^t &=& \cO\left(\max\left\{\sum\limits_{t=1}^\tau N_t, \sum\limits_{t=1}^\tau \frac{\sigma^2 R_0^2}{2^t \varepsilon_t^2} \ln\frac{M_{\nu}R_0^{1+\nu}\tau}{2^{\frac{(1+\nu)t}{2}}\varepsilon_t\beta}\right\}\right)\\
         &=& \cO\left(\max\left\{\hat{N}, \sum\limits_{t=1}^\tau \frac{\sigma^2 \cdot 2^{t}}{\mu^2 R_0^2}\ln\frac{D}{\beta}\right\}\right) = \cO\left(\max\left\{\hat{N}, \frac{\sigma^2}{\mu \varepsilon}\ln\frac{D}{\beta}\right\}\right). \qed
    \end{eqnarray*}
\end{proof}
One can also derive a similar result for \algname{R-clipped-SGD} when stepsize $\gamma$ is chosen as in Corollary~\ref{cor:SGD_cvx_small_batch} for all restarts. In this case, one can choose unit batch sizes: $m_t = 1$ for all $t$.

\section{Numerical Experiments}\label{sec:experiments}

\revision{In this section, we present the results of our numerical experiments in synthetic and real-world data. We defer additional details regarding the choice of parameters to Appendix~\ref{sec:extra_experiments}.}

\subsection{Experiments on Synthetic Data}
First of all, we tested the considered methods on the following problem, which corresponds to the linear regression with the noise having generalized Gaussian distribution (Example 4.4 from \citep{Chaux_2007}):
\begin{equation}
    \min\limits_{x \in \R^n}\left\{f_p(x) = \frac{1}{m}\|\mA x - y\|_p^p =: \frac{1}{m}\sum\limits_{i=1}^m f_{i,p}(x)\right\},\;\;   f_{i,p}(x) = |a_i^\top x - y_i|^{p}, \label{eq:generalized_lin_reg}
\end{equation}
where $\mA \in \R^{m\times n}$, $y \in \R^m$, $p \in [1,2]$, and $a_i^\top$ denotes the $i$-th row of matrix $\mA$. One can show that $f_p(x)$ is convex and has $(\nu, M_\nu)$-H\"older continuous (sub)gradient\footnote{\revision{For $p \in (1,2]$ function $f_{i,p}(x)$ is differentiable and $\nabla f_{i,p}(x) = p|a_i^\top x - y_i|^{p-1} \mathrm{sign}(a_i^\top x - y_i)a_i $ and for $p = 1$ it has subdifferential $\partial f_{i,p}(x) = \begin{cases}
    a_i,& \text{if } a_i^\top x - y_i > 0,\\ [-a_i, a_i],&\text{if } a_i^\top x - y_i = 0,\\ -a_i,& \text{if } a_i^\top x - y_i < 0.
\end{cases}$}} with $\nu = p-1$ and $M_\nu = \tfrac{2^{1-\nu}(1+\nu)}{m}\sum_{i=1}^m\|a_i\|_2^{1+\nu}$. Moreover, to rewrite the considered problem in the form \eqref{eq:main_problem}, we define $\xi$ as a random index having a uniform distribution on $\{1,\ldots,m\}$. Since the (sub)gradient of $f_i$ is bounded on any compact set and any $i \in \{1, \ldots, m\}$, the variance of the stochastic gradient can be uniformly upper bounded on any compact set as well. That is, problem \eqref{eq:generalized_lin_reg} fits the setup we consider in the theoretical analysis in the previous sections.

\begin{figure}[h]
  \centering
  \includegraphics[scale=0.27]{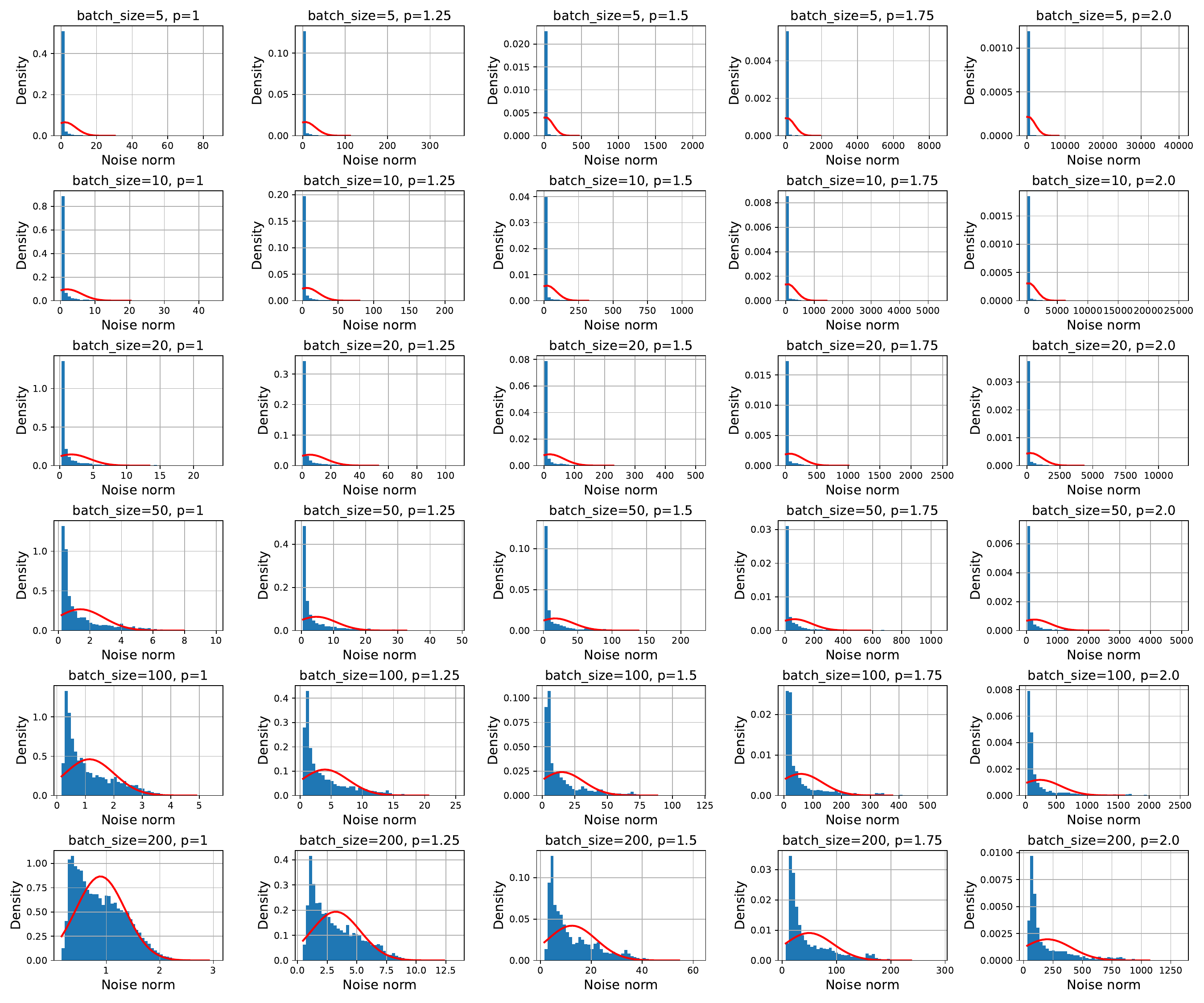}
  \caption{Noise distribution of the stochastic gradients for synthetic dataset, depending on batch size and $p$ of the loss function (\ref{eq:generalized_lin_reg}). \revision{Red lines: Gaussian probability density functions with means and variances empirically estimated by the samples.} The total number of batches for each graph is $5\cdot10^5$.}
  \label{figure:norm_diffs_distribution_synth}
\end{figure}

We generate matrix $\mA$ as follows: 1) \revision{assemble the} matrix $\revision{\mA_1} \in \R^{m\times n}$ from $mn$ i.i.d.\ samples from the standard Gaussian distribution $\cN(0,1)$, 2) multiply the rows of matrix $\revision{\mA_1}$ on i.i.d.\ samples from the Levy distribution with parameters $0$ and $0.5$ that are smaller than the threshold $t = 10^4$ (we redraw a sample if it is larger than $t$ to avoid numerical instabilities during the experiments)\revision{, denote the result by $\revision{\mA_2}$}, 3) divide the columns of \revision{$\mA_2$} by their empirical standard deviations (again due to numerical instabilities)\revision{, denote the result by $\mA$}, 4) split the dataset equally into the train and test sets and add i.i.d.\ samples from the Levy distribution with parameters $0$ and $0.5$ to the train part (we redraw a sample if it is larger than $t\cdot \alpha$ with $\alpha = 10$). Next, we generate $x = x_{\text{true}}$ as a random vector from the uniform distribution on the unit Euclidean sphere and set $y := y_{\text{true}} = \mA x_{\text{true}}$. We use $m = 10000$, $n = 200$ \revision{($5000$ for the train set and $5000$ for the test set)}. The starting point for the methods was generated from the uniform distribution on the Euclidean sphere with radius $10$. We also use $x_{\text{pred}}$ to denote the output of a method and $y_{\text{pred}} = \mA x_{\text{pred}}$ to denote the ``answer'' of the trained model.

The resulting problem has a heavy-tailed stochastic gradient noise. To illustrate this, for different values of $p$, we sample a large enough number of batched stochastic gradients $\nabla f_p(x^0,\Bxi_1),\ldots, \nabla f_p(x^0,\Bxi_K)$ with the batch size we use to run the methods and plot the histograms for $\|\nabla f_p(x^0,\Bxi_1) - \nabla f_p(x^0)\|_2,\ldots, \|\nabla f_p(x^0,\Bxi_K) - \nabla f_p(x^0)\|_2$, see Figure~\ref{figure:norm_diffs_distribution_synth}.


We compared 4 different methods on this problem with different $p$: \algname{Adam}, \algname{SGD}, \algname{clipped-SGD}, and \algname{clipped-SSTM}. The results w.r.t. the best relative loss achieved on the training dataset are reported in Figure~\ref{figure:results_best_train_synth}. In all our experiments, \algname{clipped-SSTM} performs significantly better than other tested methods for all values of $p$. We also observe that for $p < 1.5$ \algname{SGD} has a comparable or even faster convergence than \algname{clipped-SGD}, while for larger values of $p$ \algname{SGD} is much slower than \algname{clipped-SGD}. Taking into account the noise distributions reported in Figure~\ref{figure:norm_diffs_distribution_synth}, this behavior is expected since the stochastic gradient noise in the considered problem has heavier tails due to the specifics of the dataset generation. We also notice that, in this series of experiments, \algname{Adam} is never faster than \algname{clipped-SSTM} and, moreover, for $p \geq 1.5$ \algname{Adam} converges slower than \algname{clipped-SGD}. Additionally, we compared these methods w.r.t.\ the number of epochs needed to achieve a $2.0$ relative loss on the train, the results are reported in Appendix~\ref{sec:extra_experiments_synth_x2_train_speed}.

\subsection{Neural Networks Training}

In our experiments with the training of neural networks, we tested the performance of the methods on the following non-convex non-smooth problems\footnote{We conduct these experiments to illustrate that \algname{clipped-SSTM} and \algname{clipped-SGD} might be useful even for the problems that are not theoretically studied in this paper. Since \citep{gorbunov2020clipped_sstm} does not provide numerical experiments with \algname{clipped-SSTM} on the training of neural networks, our experiments are the first ones showing the behavior of \algname{clipped-SSTM} on the considered tasks.} (in both tasks, we use standard cross-entropy loss functions):
\begin{itemize}
    \item {\tt BERT} fine-tuning on {\tt CoLA} dataset \citep{warstadt2018neural}. We use pretrained {\tt BERT} from Transformers library \citep{wolf-etal-2020-transformers} ({\tt bert-base-uncased}) and freeze all layers except the last two linear ones.
    \item {\tt ResNet-18} training on {\tt ImageNet-100} (first 100 classes of {\tt ImageNet} \citep{ILSVRC15}).
\end{itemize}

First, we study the noise distribution for both problem\revision{s} as follows: at the starting point we sample large enough number of batched stochastic gradients $\nabla f(x^0,\Bxi_1),\ldots, \nabla f(x^0,\Bxi_K)$ with batch size $32$ and plot the histograms for $\|\nabla f(x^0,\Bxi_1) - \nabla f(x^0)\|_2,\ldots, \|\nabla f(x^0,\Bxi_K) - \nabla f(x^0)\|_2$, see Figure~\ref{figure:norm_diffs_distribution}. As one can see, the noise distribution for {\tt BERT} + {\tt CoLA} is substantially non-sub-Gaussian, whereas the distribution for {\tt ResNet-18} + {\tt Imagenet-100} is almost Gaussian. We observe a similar phenomenon for other points along the trajectories of the methods; see Appendix~\ref{sec:extra_experiments_noise_distr_evolution}.

\begin{figure}[h]
  \centering
  \includegraphics[scale=0.26]{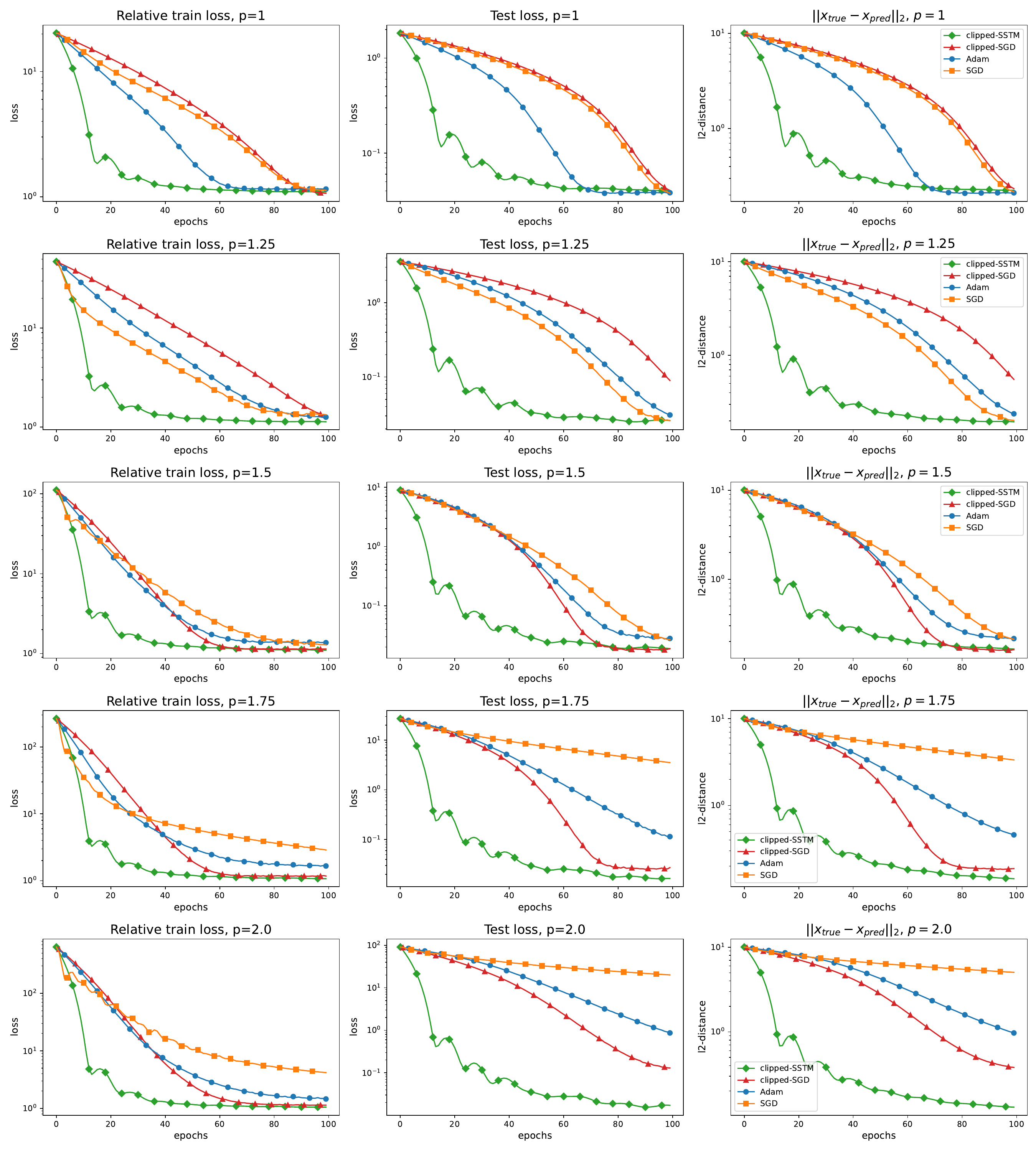}
  \caption{Results obtained for different $p$ by the best relative train loss achieved. To calculate relative loss, we use $f_p(x_{\text{pred}})/f_p(x_{\text{true}})$, where $f_p(x_{true})$ is non-zero because of the noise added to the train part of the dataset.}
  \label{figure:results_best_train_synth}
\end{figure}

\begin{figure}[h]
  \centering
  \includegraphics[scale=0.385]{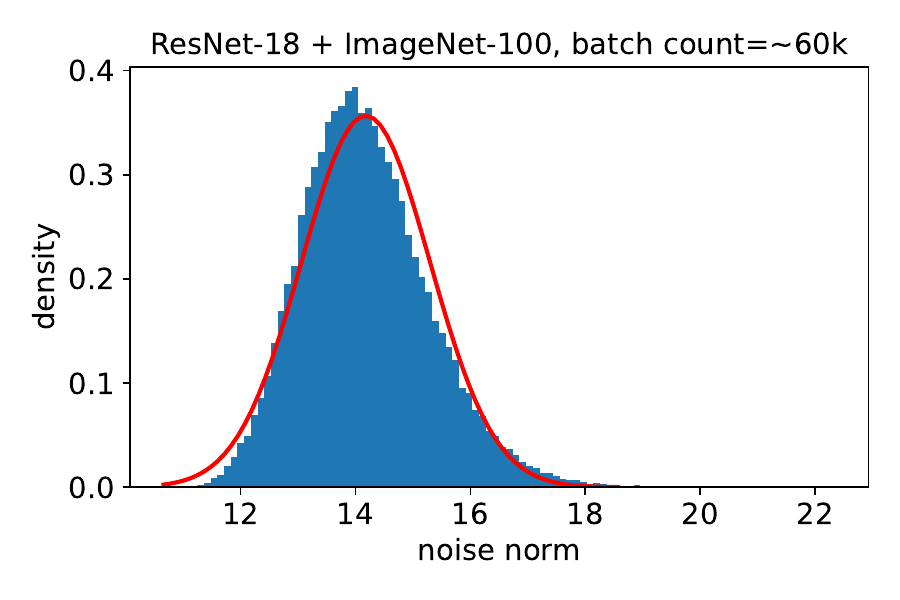}
  \includegraphics[scale=0.385]{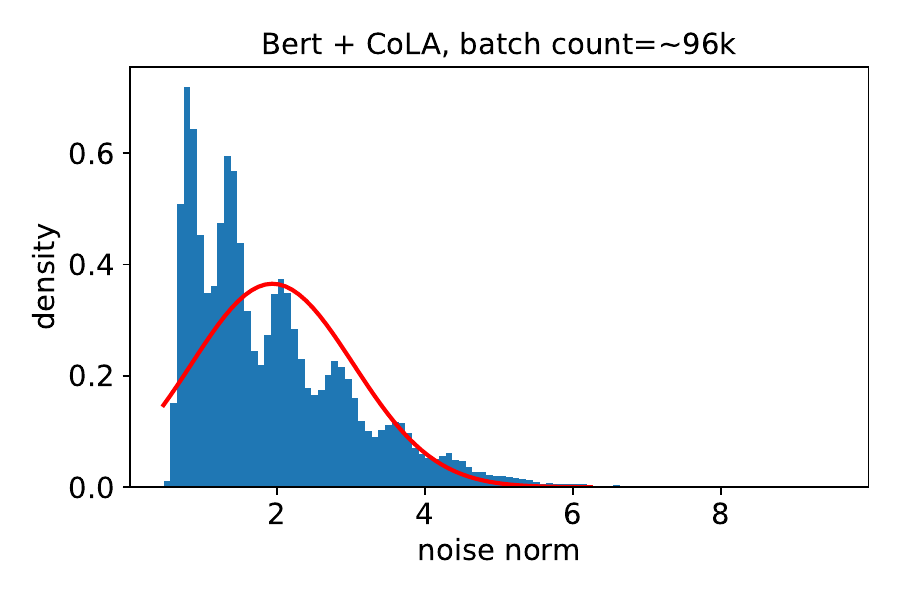}
  \caption{Noise distribution of the stochastic gradients for {\tt ResNet-18} on {\tt ImageNet-100} and {\tt BERT} fine-tuning on the {\tt CoLA} dataset before the training. Red lines: Gaussian probability density functions with means and variances empirically estimated by the samples. Batch count is the total number of samples used to build a histogram.}
  \label{figure:norm_diffs_distribution}
\end{figure}

Next, we compared \revision{four} different optimizers on these problems: \algname{Adam}, \algname{SGD} (with Momentum), \algname{clipped-SGD} (with Momentum and coordinate-wise\footnote{\revision{Following standard practice in the usage of clipping, we use coordinate-wise clipping in \algname{clipped-SGD} \cite{zhang2020why}. In the preliminary experiments, we also tried norm-clipping for \algname{clipped-SGD}, but it showed worse results than the coordinate-wise one. Our analysis can be generalized to the case of coordinate-wise clipping if we assume the boundedness of the coordinate-wise variance $\sigma_{c}^2$ of the stochastic gradients. Then, the result of Lemma~\ref{lem:main_stoch_lemma_clipped_SSTM} will hold with $\sigma^2 = n\sigma_c^2$, and the norm of the clipped vector will be bounded by $\sqrt{n}\lambda$. These changes will lead to the explicit dependence on the dimension in the complexity bounds, similarly to \cite{zhang2020why}.}} clipping) and \algname{clipped-SSTM} (with norm-clipping and $\nu=1$). The results are presented in Figure~\ref{figure:losses_graph_bert_resnet}. We observed that the noise distributions do not change significantly along the trajectories of the considered methods, see Appendix~\ref{sec:extra_experiments}. During the hyper-parameters search, we compared different batch sizes emulated via gradient accumulation (thus, we compare methods with different batch sizes by the number of base batches used). The base batch size was $32$ for both problems; stepsizes and clipping levels were tuned. One can find additional details regarding our experiments in Appendix~\ref{sec:extra_experiments}.

\begin{figure}[h]
  \centering
  \includegraphics[scale=0.33]{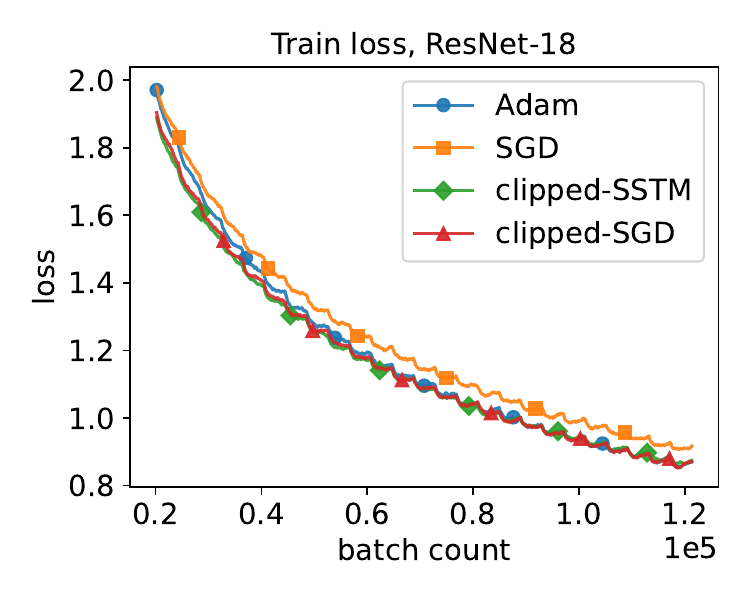}
  \includegraphics[scale=0.33]{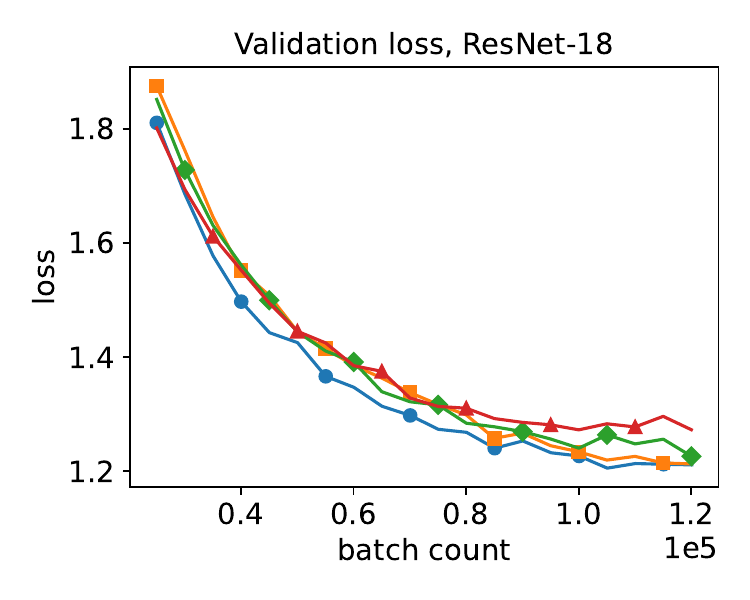}
  \includegraphics[scale=0.33]{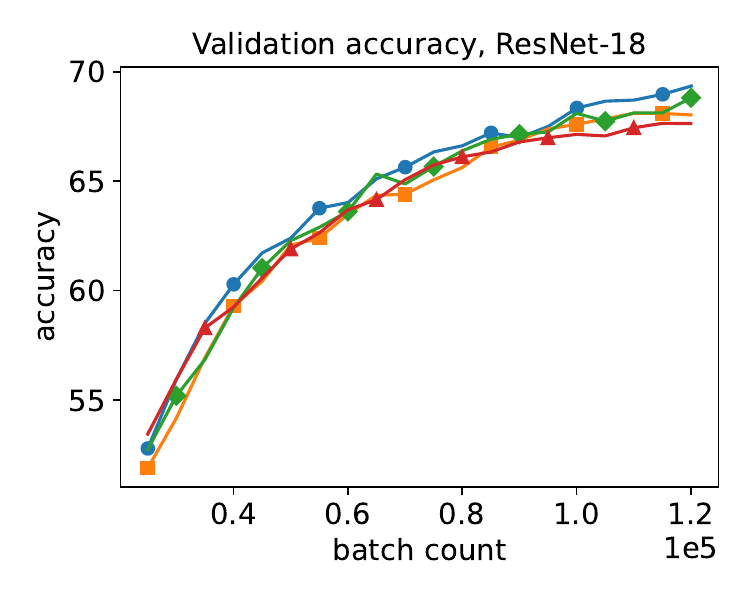}
  \includegraphics[scale=0.33]{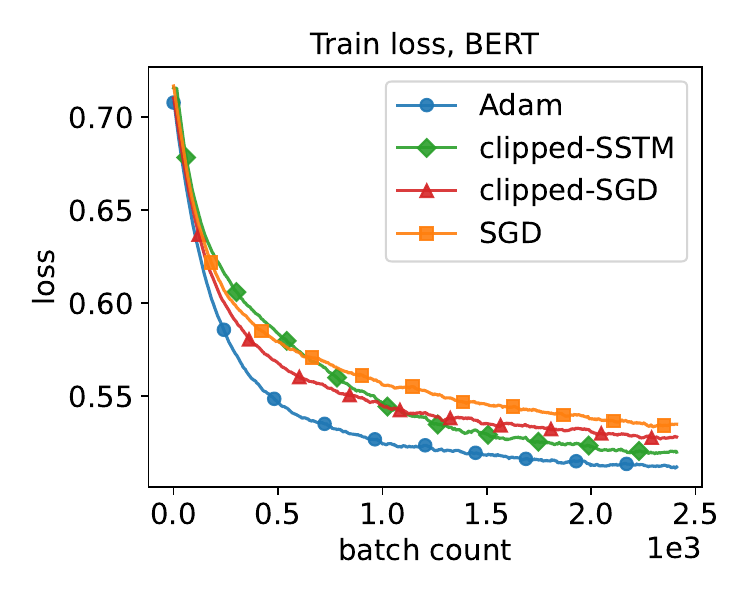}
  \includegraphics[scale=0.33]{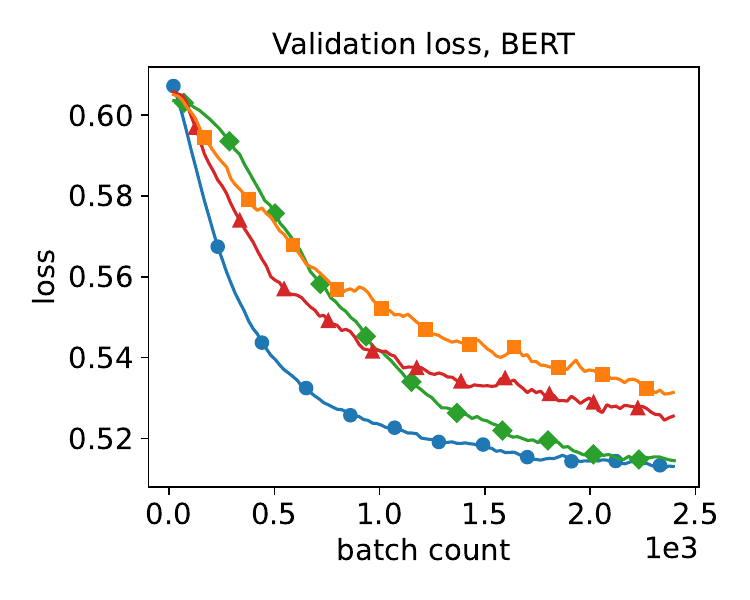}
  \includegraphics[scale=0.33]{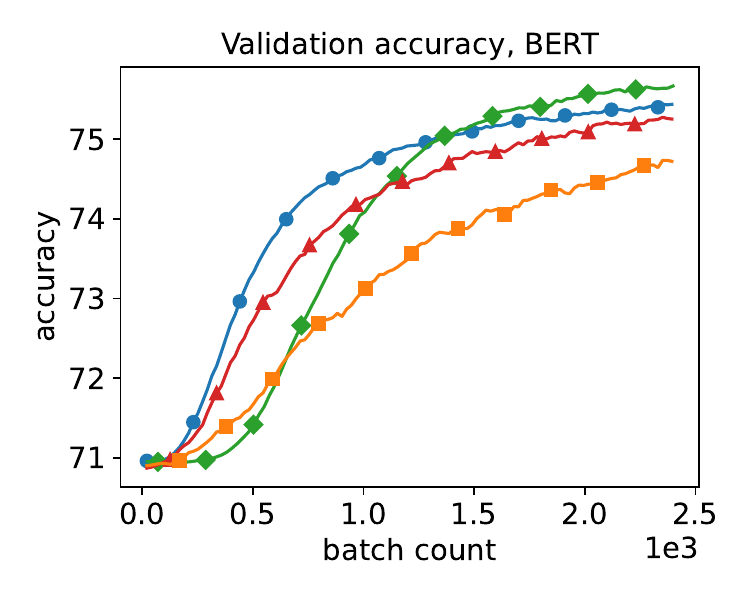}
  \caption{Train and validation loss + accuracy for different optimizers on both problems. Here, ``batch count'' denotes the total number of used stochastic gradients.}
  \label{figure:losses_graph_bert_resnet}
\end{figure}

\paragraph{Image classification.} On {\tt ResNet-18} + {\tt ImageNet-100} task, \algname{SGD} performs relatively well, and even ties with \algname{Adam} (with batch size of $4\times 32$) in validation loss. \algname{clipped-SSTM} (with batch size of $2\times 32$) also ties with \algname{Adam} and \algname{clipped-SGD} is not far from them. The results were averaged from 5 different launches (with different starting points/weight initializations). Since the noise distribution is almost Gaussian, even vanilla \algname{SGD} performs well, i.e., gradient clipping is not required. At the same time, the clipping does not slow down the convergence significantly.

\paragraph{Text classification.} On {\tt BERT} + {\tt CoLA} task, when the noise distribution is heavy-tailed, the methods with clipping outperform \algname{SGD} by a large margin. This result is in good correspondence with the derived high-probability complexity bounds for \algname{clipped-SGD}, \algname{clipped-SSTM}, and the best-known ones for \algname{SGD}. Moreover, \algname{clipped-SSTM} (with batch size of $8\times 32$) achieves the same loss on validation as \algname{Adam}, and has better accuracy. These results were averaged from $5$ different train-val splits and $20$ launches (with different starting points/weight initializations) for each of the splits, $100$ launches in total. We provide additional experiments with different NLP tasks in Appendix~\ref{appendix:additional_NLP}.


\section*{Acknowledgements}
This work was supported by a grant for research centers in the field of artificial intelligence, provided by the Analytical Center for the Government of the Russian Federation in accordance with the subsidy agreement (agreement identifier 000000D730324P540002) and the agreement with the Moscow Institute of Physics and Technology dated November 1, 2021 No. 70-2021-00138.

\paragraph{Data availability statement} The codes for the conducted numerical experiments are publicly available:\\ \url{https://github.com/ClippedStochasticMethods/clipped-SSTM}.

\appendix  
\section{Basic Facts, Technical Lemmas, and Auxiliary Results}\label{sec:basic_facts}
\subsection{Useful Inequalities}


For all $a,b\in\R^n$ 
\begin{equation}
    \|a+b\|_2^2 \le 2\|a\|_2^2 + 2\|b\|_2^2,\label{eq:squared_norm_sum}
\end{equation}
\begin{equation}
    \la a, b\ra = \frac{1}{2}\left(\|a+b\|_2^2 - \|a\|_2^2 - \|b\|_2^2\right). \label{eq:inner_product_representation}
\end{equation}

\subsection{Auxiliary Lemmas}
The following lemma is a standard result about functions with $(\nu, M_\nu)$-H\"older continuous gradient \citep{devolder2014first,nesterov2015universal}. 

\begin{lemma}
    Let $f$ has $(\nu, M_\nu)$-H\"older continuous gradient on $Q\subseteq \R^n$. Then for all $x,y\in Q$ and for all $\delta > 0$
\begin{equation}
    f(y) \le f(x) + \la\nabla f(x), y-x \ra + \frac{M_\nu}{1+\nu} \|x-y\|_2^{1+\nu}, \label{eq:holder_cor1}
\end{equation}
\begin{equation}
    f(y) \le f(x) + \la\nabla f(x), y-x \ra + \frac{L(\delta,\nu)}{2} \|x-y\|_2^{2} + \frac{\delta}{2}, \quad  L(\delta,\nu) = \left(\frac{1}{\delta}\right)^{\frac{1-\nu}{1+\nu}}M_\nu^{\frac{2}{1+\nu}}. \label{eq:holder_cor2}
\end{equation}
\end{lemma}

The next result is known as Bernstein inequality for martingale differences \citep{bennett1962probability,dzhaparidze2001bernstein,freedman1975tail}.
\begin{lemma}\label{lem:Bernstein_ineq}
    Let the sequence of random variables $\{X_i\}_{i\ge 1}$ form a martingale difference sequence, i.e.\ $\EE\left[X_i\mid X_{i-1},\ldots, X_1\right] = 0$ for all $i \ge 1$. Assume that conditional variances $\sigma_i^2\eqdef\EE\left[X_i^2\mid X_{i-1},\ldots, X_1\right]$ exist and are bounded and also assume that there exists deterministic constant $c>0$ such that $\|X_i\|_2 \le c$ almost surely for all $i\ge 1$. Then for all $b > 0$, $F > 0$ and $n\ge 1$
    \begin{equation}
        \PP\left\{\Big|\sum\limits_{i=1}^nX_i\Big| > b \text{ and } \sum\limits_{i=1}^n\sigma_i^2 \le F\right\} \le 2\exp\left(-\frac{b^2}{2F + \nicefrac{2cb}{3}}\right).
    \end{equation}
\end{lemma}

\subsection{Technical Lemmas}

\begin{lemma}\label{lem:alpha_k_A_K_lemma}
    Let sequences $\{\alpha_k\}_{k\ge0}$ and $\{A_k\}_{k\ge 0}$ satisfy
    \begin{equation}
        \alpha_{0} = A_0 = 0,\quad \alpha_{k+1} = \frac{(k+1)^{\frac{2\nu}{1+\nu}}(\nicefrac{\varepsilon}{2})^{\frac{1-\nu}{1+\nu}}}{2^{\frac{2\nu}{1+\nu}}aM_\nu^{\frac{2}{1+\nu}}}, \quad  A_{k+1} = A_k + \alpha_{k+1},\quad a,\varepsilon, M_{\nu} > 0,\; \nu \in [0,1] \label{eq:alpha_k_A_k_def}
    \end{equation}
    for all $k\ge 0$. Then for all $k\ge 0$ we have
    \begin{equation}
        A_{k} \ge a L_{k} \alpha_{k}^2,\quad A_{k} \ge  \frac{k^{\frac{1+3\nu}{1+\nu}}(\nicefrac{\varepsilon}{2})^{\frac{1-\nu}{1+\nu}}}{2^{\frac{1+3\nu}{1+\nu}}aM_\nu^{\frac{2}{1+\nu}}}, \label{eq:A_k_lower_bound}
    \end{equation}
    where $L_0 = 0$ and for $k > 0$
    \begin{equation}
        L_{k} = \left(\frac{2A_{k}}{\alpha_{k}\varepsilon}\right)^{\frac{1-\nu}{1+\nu}} M_\nu^{\frac{2}{1+\nu}}. \label{eq:L_k_definition}
    \end{equation}
    Moreover, for all $k \ge 0$
    \begin{equation}
        A_k \le \frac{k^{\frac{1+3\nu}{1+\nu}}(\nicefrac{\varepsilon}{2})^{\frac{1-\nu}{1+\nu}}}{2^{\frac{2\nu}{1+\nu}}aM_\nu^{\frac{2}{1+\nu}}}. \label{eq:A_k_upper_bound}
    \end{equation}
\end{lemma}
\begin{proof}
    We start with deriving the second inequality from \eqref{eq:A_k_lower_bound}. The proof goes by induction. For $k = 0$, the inequality holds. Next, we assume that it holds for all $k \le K$. Then,
    \begin{equation*}
        A_{K+1} = A_{K} + \alpha_{K+1} \ge \frac{K^{\frac{1+3\nu}{1+\nu}}(\nicefrac{\varepsilon}{2})^{\frac{1-\nu}{1+\nu}}}{2^{\frac{1+3\nu}{1+\nu}}aM_\nu^{\frac{2}{1+\nu}}} + \frac{(K+1)^{\frac{2\nu}{1+\nu}}(\nicefrac{\varepsilon}{2})^{\frac{1-\nu}{1+\nu}}}{2^{\frac{2\nu}{1+\nu}}aM_\nu^{\frac{2}{1+\nu}}}.
    \end{equation*}
    Let us estimate the right-hand side of the previous inequality. We want to show that
    \begin{eqnarray*}
        \frac{K^{\frac{1+3\nu}{1+\nu}}(\nicefrac{\varepsilon}{2})^{\frac{1-\nu}{1+\nu}}}{2^{\frac{1+3\nu}{1+\nu}}aM_\nu^{\frac{2}{1+\nu}}} + \frac{(K+1)^{\frac{2\nu}{1+\nu}}(\nicefrac{\varepsilon}{2})^{\frac{1-\nu}{1+\nu}}}{2^{\frac{2\nu}{1+\nu}}aM_\nu^{\frac{2}{1+\nu}}} &\ge& \frac{(K+1)^{\frac{1+3\nu}{1+\nu}}(\nicefrac{\varepsilon}{2})^{\frac{1-\nu}{1+\nu}}}{2^{\frac{1+3\nu}{1+\nu}}aM_\nu^{\frac{2}{1+\nu}}}
    \end{eqnarray*}
    that is equivalent to the inequality:
    \begin{equation*}
        \frac{K^{\frac{1+3\nu}{1+\nu}}}{2} + (K+1)^{\frac{2\nu}{1+\nu}} \ge \frac{(K+1)^{\frac{1+3\nu}{1+\nu}}}{2} \Longleftrightarrow \frac{K^{\frac{1+3\nu}{1+\nu}}}{2} \ge \frac{(K+1)^{\frac{2\nu}{1+\nu}}(K-1)}{2}.
    \end{equation*}
    If $K = 1$, it trivially holds. If $K > 1$, it is equivalent to
    \begin{equation*}
        \frac{K}{K-1} \ge \left(\frac{K+1}{K}\right)^{2 - \frac{2}{1+\nu}}.
    \end{equation*}
    Since $2 - \frac{2}{1+\nu}$ is monotonically increasing function for $\nu\in[0,1]$ we have that
    \begin{equation*}
        \left(\frac{K+1}{K}\right)^{2 - \frac{2}{1+\nu}} \le \frac{K+1}{K} \le \frac{K}{K-1}.
    \end{equation*}
    That is, the second inequality in \eqref{eq:A_k_lower_bound} holds for $k = K+1$, and, as a consequence, it holds for all $k \ge 0$. Next, we derive the first part of \eqref{eq:A_k_lower_bound}. For $k = 0$, it trivially holds. For $k > 0$ we consider cases $\nu = 0$ and $\nu > 0$ separately. When $\nu = 0$ the inequality is equivalent to
    \begin{equation*}
        1 \ge \frac{2a\alpha_k M_0^2}{\varepsilon}, \text{ where } \frac{2a\alpha_k M_0^2}{\varepsilon} \overset{\eqref{eq:alpha_k_A_k_def}}{=} 1,
    \end{equation*}
    i.e., we have $A_k = aL_k\alpha_k^2$ for all $k\ge 0$. When $\nu > 0$ the first inequality in \eqref{eq:A_k_lower_bound} is equivalent to
    \begin{equation*}
        A_{k} \ge a^{\frac{1+\nu}{2\nu}}\alpha_{k}^{\frac{1+3\nu}{2\nu}}(\nicefrac{\varepsilon}{2})^{-\frac{1-\nu}{2\nu}}M_\nu^{\frac{1}{\nu}} \overset{\eqref{eq:alpha_k_A_k_def}}{\Longleftrightarrow} A_{k} \ge \frac{k^{\frac{1+3\nu}{1+\nu}}(\nicefrac{\varepsilon}{2})^{\frac{1-\nu}{1+\nu}}}{2^{\frac{1+3\nu}{1+\nu}}aM_\nu^{\frac{2}{1+\nu}}},
    \end{equation*}
    where the last inequality coincides with the second inequality from \eqref{eq:A_k_lower_bound} that we derived earlier in the proof.
    
    To finish the proof, it remains to derive \eqref{eq:A_k_upper_bound}. Again, the proof goes by induction. For $k=0$ inequality \eqref{eq:A_k_upper_bound} is trivial. Next, we assume that it holds for all $k \le K$. Then,
    \begin{equation*}
        A_{K+1} = A_{K} + \alpha_{K+1} \le \frac{K^{\frac{1+3\nu}{1+\nu}}(\nicefrac{\varepsilon}{2})^{\frac{1-\nu}{1+\nu}}}{2^{\frac{2\nu}{1+\nu}}aM_\nu^{\frac{2}{1+\nu}}} + \frac{(K+1)^{\frac{2\nu}{1+\nu}}(\nicefrac{\varepsilon}{2})^{\frac{1-\nu}{1+\nu}}}{2^{\frac{2\nu}{1+\nu}}aM_\nu^{\frac{2}{1+\nu}}}.
    \end{equation*}
    Let us estimate the right-hand side of the previous inequality. We want to show that
    \begin{eqnarray*}
        \frac{K^{\frac{1+3\nu}{1+\nu}}(\nicefrac{\varepsilon}{2})^{\frac{1-\nu}{1+\nu}}}{2^{\frac{2\nu}{1+\nu}}aM_\nu^{\frac{2}{1+\nu}}} + \frac{(K+1)^{\frac{2\nu}{1+\nu}}(\nicefrac{\varepsilon}{2})^{\frac{1-\nu}{1+\nu}}}{2^{\frac{2\nu}{1+\nu}}aM_\nu^{\frac{2}{1+\nu}}} &\le& \frac{(K+1)^{\frac{1+3\nu}{1+\nu}}(\nicefrac{\varepsilon}{2})^{\frac{1-\nu}{1+\nu}}}{2^{\frac{2\nu}{1+\nu}}aM_\nu^{\frac{2}{1+\nu}}}
    \end{eqnarray*}
    that is equivalent to the inequality:
    \begin{equation*}
        K^{\frac{1+3\nu}{1+\nu}} + (K+1)^{\frac{2\nu}{1+\nu}} \le (K+1)^{\frac{1+3\nu}{1+\nu}}.
    \end{equation*}
    This inequality holds due to
    \begin{equation*}
        K^{\frac{1+3\nu}{1+\nu}} \le (K+1)^{\frac{2\nu}{1+\nu}}K.
    \end{equation*}
    That is, \eqref{eq:A_k_upper_bound} holds for $k = K+1$, and, as a consequence, it holds for all $k \ge 0$. \qed
\end{proof}

\begin{lemma}\label{lem:gradient_bound}
    Let $f$ have H\"older continuous gradients on \revision{$\R^n$} for some $\nu \in [0,1]$ with constant $M_\nu > 0$, be convex and $x^*$ be  some minimum of $f(x)$ on $\R^n$. Then, for all \revision{$x\in \R^n$}
    \begin{equation}
        \|\nabla f(x)\|_2 \le \left(\frac{1+\nu}{\nu}\right)^{\frac{\nu}{1+\nu}}M_\nu^{\frac{1}{1+\nu}} \left(f(x) - f(x^*)\right)^{\frac{\nu}{1+\nu}}, \label{eq:holder_gradient_bound_appendix}
    \end{equation}
    where for $\nu = 0$ we use $\left[\left(\frac{1+\nu}{\nu}\right)^{\frac{\nu}{1+\nu}}\right]_{\nu=0} := \lim_{\nu\to 0}\left(\frac{1+\nu}{\nu}\right)^{\frac{\nu}{1+\nu}} = 1$.
\end{lemma}
\begin{proof}
    For $\nu = 0$ inequality \eqref{eq:holder_gradient_bound_appendix} follows from \eqref{eq:holder_def} and\footnote{When $f$ is not differentiable, we use subgradients. In this case, $0$ belongs to the subdifferential of $f$ at the point $x^*$, and we take it as $\nabla f(x^*)$.} $\nabla f(x^*) = 0$. When $\nu > 0$ for arbitrary point \revision{$x\in \R^n$} we consider the point $y = x - \alpha\nabla f(x)$, where $\alpha = \left(\frac{\|\nabla f(x)\|_2^{1-\nu}}{M_\nu}\right)^{\frac{1}{\nu}}$.
    For the pair of points $x,y$ we apply \eqref{eq:holder_cor1} and get
    \begin{eqnarray*}
        f(y) &\le& f(x) + \langle\nabla f(x), y-x\rangle + \frac{M_\nu}{1+\nu}\|x-y\|_2^{1+\nu}\\
        &=& f(x) - \alpha\|\nabla f(x)\|^2 + \frac{\alpha^{\nu+1}M_\nu}{1+\nu}\|\nabla f(x)\|_2^{1+\nu}\\
        &=& f(x) - \frac{\|\nabla f(x)\|_2^{\frac{1+\nu}{\nu}}}{M_\nu^{\frac{1}{\nu}}} + \frac{\|\nabla f(x)\|_2^{\frac{1+\nu}{\nu}}}{(1+\nu)M_\nu^{\frac{1}{\nu}}} = f(x) - \frac{\nu\|\nabla f(x)\|_2^{\frac{1+\nu}{\nu}}}{(1+\nu)M_\nu^{\frac{1}{\nu}}}
    \end{eqnarray*}
    implying
    \begin{equation*}
         \|\nabla f(x)\|_2 \le \left(\frac{1+\nu}{\nu}\right)^{\frac{\nu}{1+\nu}}M_\nu^{\frac{1}{1+\nu}} \left(f(x) - f(y)\right)^{\frac{\nu}{1+\nu}} \le \left(\frac{1+\nu}{\nu}\right)^{\frac{\nu}{1+\nu}}M_\nu^{\frac{1}{1+\nu}} \left(f(x) - f(x^*)\right)^{\frac{\nu}{1+\nu}}. \qed
    \end{equation*}
\end{proof}

\begin{lemma}\label{lem:gradient_bound_2}
    Let $f$ have H\"older continuous gradients on \revision{$\R^n$} for some $\nu \in [0,1]$ with constant $M_\nu > 0$, be convex and $x^*$ be  some  minimum of $f(x)$ on $\R^n$. Then, for all $x\in\R^n$  and all $\delta >0$,
    \begin{equation}
        \|\nabla f(x)\|_2^2 \le
        2\left(\frac{1}{\delta}\right)^{\frac{1-\nu}{1+\nu}}M_{\nu}^{\frac{2}{1+\nu}}\left(f(x)-f(x^*)\right) + \delta^{\frac{2\nu}{1+\nu}} M_{\nu}^{\frac{2}{1+\nu}} .\label{eq:holder_gradient_bound_2_appendix}
    \end{equation}
\end{lemma}
\begin{proof}
    For a given $\delta > 0$ we consider an arbitrary point $x\in Q$ and $y = x - \frac{1}{L(\delta,\nu)}\nabla f(x)$, where $L(\delta,\nu) = \left(\frac{1}{\delta}\right)^{\frac{1-\nu}{1+\nu}}M_\nu^{\frac{2}{1+\nu}}$. 
    For the pair of points $x,y$ we apply \eqref{eq:holder_cor2} and get
    \begin{eqnarray*}
        f(y) &\le& f(x) + \la\nabla f(x), y-x \ra + \frac{L(\delta,\nu)}{2} \|x-y\|_2^{2} + \frac{\delta}{2}\\
        &=& f(x) - \frac{1}{2L(\delta,\nu)}\|\revision{\nabla f(x)}\|_2^2 + \frac{\delta}{2}
    \end{eqnarray*}
    implying
    \begin{eqnarray*}
         \|\nabla f(x)\|_2^2 &\le& 2L(\delta,\nu)\left(f(x) - f(y)\right) + \delta L(\delta, \nu)\\
         &\le& 2\left(\frac{1}{\delta}\right)^{\frac{1-\nu}{1+\nu}}M_{\nu}^{\frac{2}{1+\nu}}\left(f(x)-f(x^*)\right) + \delta^{\frac{2\nu}{1+\nu}} M_{\nu}^{\frac{2}{1+\nu}}. \qed
    \end{eqnarray*}
\end{proof}


\section{Additional Experimental Details and Results}\label{sec:extra_experiments}

\subsection{Experiments on Synthetic Data}




\subsubsection{Hyper-Parameters tuning}

We grid-searched hyper-parameters for each model. Commonly for all models we considered batch sizes of $\revision{\{5, 10, 20, 50, 100, 200\}}$ and stepsizes \revision{$lr\in[1\mathrm{e}{-1},1\mathrm{e}{-5}]$}. As to model-specific parameters:
\begin{itemize}
    \item for \algname{Adam} we grid-search over $betas\in(\{0.8, 0.9, 0.95, 0.99\}, \{0.9, 0.99, 0.999\})$,
    \item for \algname{SGD} ~--- over \revision{$momentum\in\{0.8, 0.9, 0.99, 0.999\}$},
    \item for \algname{clipped-SSTM} ~--- over clipping parameter $B\in \revision{\{1\mathrm{e}{-0},1\mathrm{e}{-1}, 1\mathrm{e}{-2}, 1\mathrm{e}{-3}\}}$,
    \item for \algname{clipped-SGD} ~--- over \revision{$momentum\in\{0.8, 0.9, 0.99, 0.999\}$} and clipping parameter \revision{$B\in \revision{\{1\mathrm{e}{-0},1\mathrm{e}{-1}, 1\mathrm{e}{-2}, 1\mathrm{e}{-3}\}}$}.
\end{itemize}

For \algname{clipped-SSTM} we additionally use $\nu=1$ and norm clipping (we did not gridsearch over it extensively; however, in our experiments on real data, these parameters were the best). For \algname{clipped-SGD} we use coordinate-wise clipping.

For  \algname{Adam}, \algname{clipped-SSTM} and \algname{clipped-SGD} the best parameters for each $p$ were approximately the same:
\begin{itemize}
    \item \algname{Adam}: $lr=1\mathrm{e}{-3}$, $betas=(0.9, 0.9)$ and batch size of $10$
    \item \algname{clipped-SSTM}: $lr=1\mathrm{e}{-3}$, $\nu = 1$, $B=1\mathrm{e}{-2}$, norm clipping and a batch size of $5$
    \item \algname{clipped-SGD}: $lr=1\mathrm{e}{-3}$ and $B=1\mathrm{e}{-1}$ or $lr=1\mathrm{e}{-2}$ and $B=1\mathrm{e}{-2}$, $momentum=0.8$, coordinate-wise clipping and a batch size of $5$
\end{itemize}

\subsubsection{Comparison w.r.t. certain relative train loss level} \label{sec:extra_experiments_synth_x2_train_speed}

In Figure~\ref{figure:results_best_train_synth}, we reported the performance of the methods in terms of the best models w.r.t.\ train loss achieved. However, it is also interesting to compare the methods w.r.t.\ the speed they achieve a certain ($2.0$) level of relative loss on train ($f_p(x_{\text{pred}})/f_p(x_{\text{true}})$). This is a valid metric, since $f_p(x_{\text{true}})$ is non-zero, after adding noise to the train part of the dataset, and $x_{\text{true}}$ is still a good approximation of the optimal solution. The results are represented in Figure~\ref{figure:results_best_by_x2_train_loss_synth}. As in the previous set of experiments, one can see that \algname{clipped-SSTM} outperforms other algorithms and achieves this $2.0$ level of relative loss much faster, though later \revision{loses} to \algname{Adam}/\algname{clipped-SGD}.

\begin{figure}[h]
  \centering
  \includegraphics[scale=0.26]{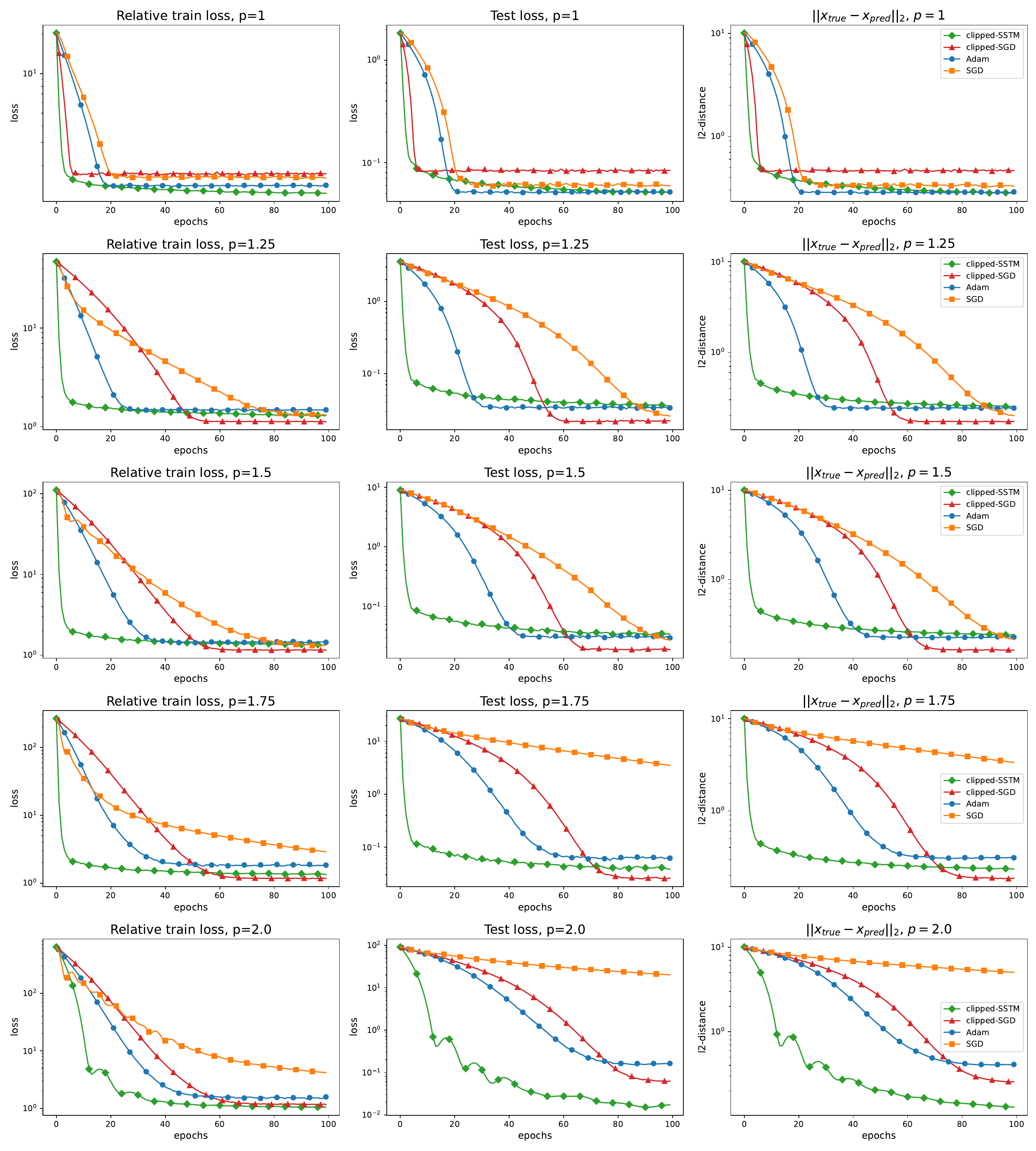}
  \caption{Results obtained for different $p$ by the lowest epoch when model achieved $\times 2$ from loss in $x_{\text{true}}$}
  \label{figure:results_best_by_x2_train_loss_synth}
\end{figure}

\subsection{Neural Networks Training}

\subsubsection{Hyper-Parameters}

In our experiments with the training of neural networks, we use standard implementations of \algname{Adam} and \algname{SGD} from PyTorch \cite{NEURIPS2019_9015}; we write only the parameters we changed from the default.

To conduct these experiments, we used Nvidia RTX 2070s. The longest experiment (evolution of the noise distribution for image classification task) took 53 hours (we iterated several times over the train dataset to build a better histogram; see Appendix~\ref{sec:extra_experiments_noise_distr_evolution}). 

\paragraph{Image Classification.} For {\tt ResNet-18} + {\tt ImageNet-100} the parameters of the methods were chosen as follows:
\begin{itemize}
    \item \algname{Adam}: $lr=1e-3$ and a batch size of $4\times 32$
    \item \algname{SGD}: $lr=1e-2$, $momentum=0.9$ and a batch size of $32$
    \item \algname{clipped-SGD}: $lr=5e-2$, $momentum=0.9$, coordinate-wise clipping with clipping parameter $B=0.1$ and a batch size of $32$
    \item \algname{clipped-SSTM}: $\nu = 1$, stepsize parameter $\alpha = 1e-3$ (in code we use separately $lr=1e-2$ and $L=10$ and $\alpha = \frac{lr}{L}$), norm clipping with clipping parameter $B=1$ and a batch size of $2\times 32$. We also upper bounded the ratio $\nicefrac{A_k}{A_{k+1}}$ by $0.99$ (see $a\_k\_ratio\_upper\_bound$ parameter in code)
\end{itemize}

The main two parameters that we grid-searched were $lr$ and batch size. For both of them, we used a logarithmic grid (i.e., for $lr$, we used $1e-5,2e-5,5e-5,1e-4,\ldots,1e-2,2e-2,5e-2$ for \algname{Adam}). Batchsize was chosen from $32, 2\cdot 32, 4\cdot 32$, and $8\cdot 32$. For \algname{SGD}, we also tried various momentum parameters. 

For \algname{clipped-SSTM} and \algname{clipped-SGD}, we used clipping levels of $1$ and $0.1$, respectively. Too small a choice of the clipping level, e.g. $0.01$, slows down the convergence significantly.

Another important parameter for \algname{clipped-SSTM} here was $a\_k\_ratio\_upper\_bound$ -- we used it to upper bound the maximum ratio of $\nicefrac{A_k}{A_{k+1}}$. Without this modification, the method is \revision{too} conservative. e.g.,  after $10^4$ steps $\nicefrac{A_k}{A_{k+1}}\approx 0.9999$. Effectively, it can be seen as a momentum parameter of \algname{SGD}.

\paragraph{Text Classification, CoLA.} For {\tt BERT} + {\tt CoLA} the parameters of the methods were chosen as follows:
\begin{itemize}
    \item \algname{Adam}: $lr=5e-5$, $weight\_decay=5e-4$ and a batch size of $32$
    \item \algname{SGD}: $lr=1e-3$, $momentum=0.9$ and a batch size of $32$
    \item \algname{clipped-SSTM}: $\nu = 1$, stepsize parameter $\alpha = 8e-3$, norm clipping with clipping parameter $B=1$ and a batch size of $8\times 32$
    \item \algname{clipped-SGD}: $lr=2e-3$, $momentum=0.9$, coordinate-wise clipping with clipping parameter $B=0.1$ and a batch size of $32$
\end{itemize}

There, we used the same grid as in the previous task. The main difference here is that we didn't bound \algname{clipped-SSTM} $A_k/A_{k+1}$ ratio -- there are only $\approx 300$ steps of the method (because the batch size is $8\cdot 32$). \revision{Thus, the} method is still not too conservative.

\subsubsection{On the Relation Between Stepsize Parameter $\alpha$ and Batchsize}

In our experiments, we noticed that \algname{clipped-SSTM} shows similar results when the \revision{ratio} $\nicefrac{bs^2}{\alpha}$ is kept unchanged, where $bs$ is batch size (see Figure~\ref{figure:losses_graph_bert_alpha-bs_relation}). We compare the performance of \algname{clipped-SSTM} with $4$ different choices of $\alpha$ and the batch size. 

\begin{figure}[h]
  \centering
  \includegraphics[scale=0.3]{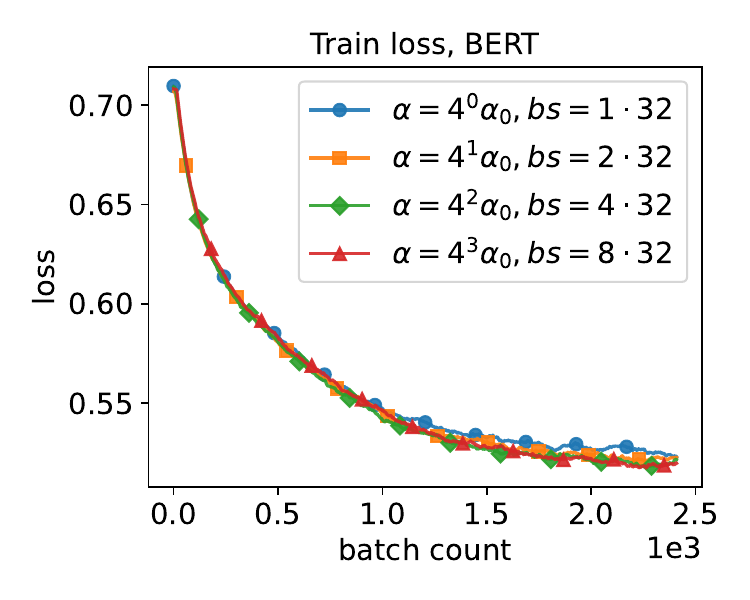}
  \includegraphics[scale=0.3]{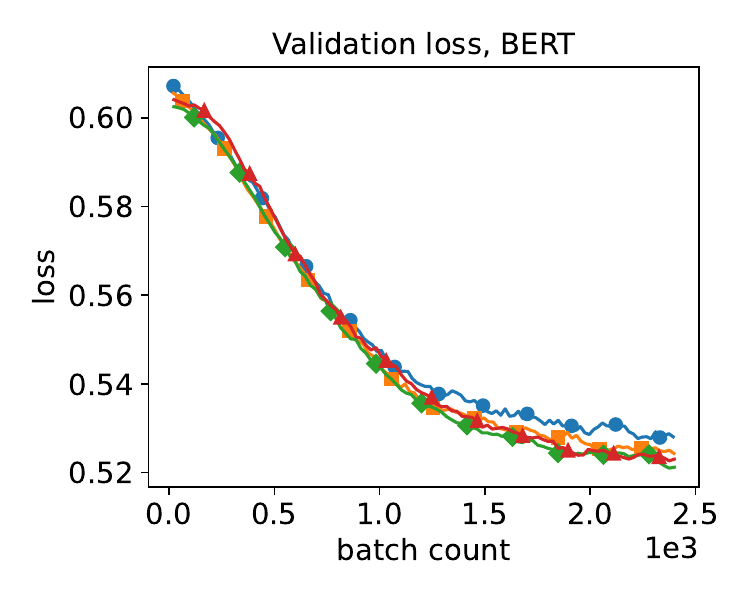}
  \includegraphics[scale=0.3]{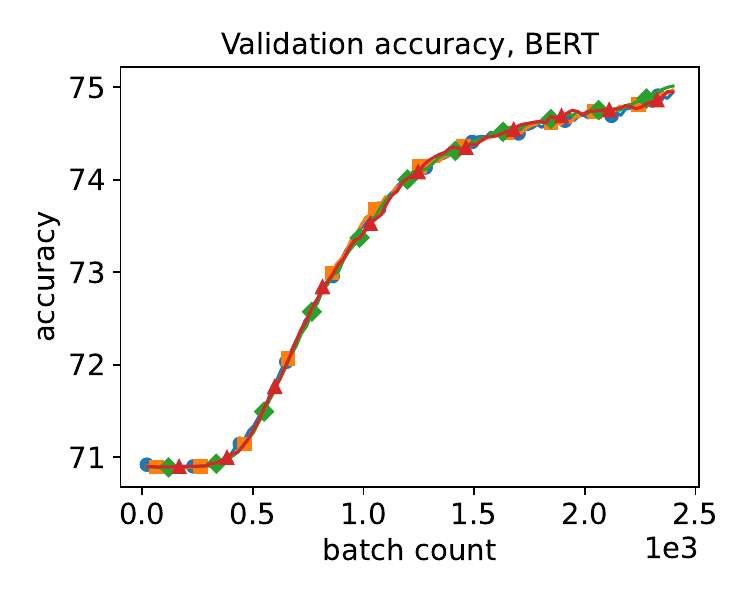}
  \caption{Train and validation loss + accuracy for \algname{clipped-SSTM} with different parameters. Here $\alpha_0 = 0.000125$, $bs$ means batch size. As we can see from the plots, increasing $\alpha$ 4 times and batch size 2 times almost does not affect the method's behavior.}
  \label{figure:losses_graph_bert_alpha-bs_relation}
\end{figure}

Theorem~\ref{thm:main_result_clipped_SSTM} explains this phenomenon in the convex case. For the case of $\nu = 1$ we have (from \eqref{eq:bathces_clipped_SSTM} and \eqref{eq:C_definition_clipped_SSTM}):
\begin{align*}
    \alpha \sim \frac{1}{aM_1},\quad \alpha_k \sim k\alpha,\quad m_k \sim \frac{N a \sigma^2 \alpha_{k+1}^2}{C^2R_0^2\ln \frac{4N}{\beta}},\quad N \sim \frac{a^{\frac{1}{2}}CR_0M_1^{\frac{1}{2}}}{\varepsilon^{\frac{1}{2}}}\sim \frac{CR_0}{\alpha^{\frac{1}{2}}\varepsilon^{\frac{1}{2}}},
\end{align*}
whence
\begin{align*}
    m_k \sim \frac{CR_0 a \sigma^2 \alpha^2(k+1)^2}{\alpha^{\frac{1}{2}}\varepsilon^{\frac{1}{2}}C^2R_0^2\ln \frac{4N}{\beta}} \sim \frac{\sigma^2 \alpha^2(k+1)^2}{\alpha^{\frac{1}{2}}\alpha M_1\varepsilon^{\frac{1}{2}}CR_0\ln \frac{4N}{\beta}}\sim \alpha^{\frac{1}{2}},
\end{align*}
where the dependencies on numerical constants and logarithmic factors are omitted. Therefore, the observed empirical relation between batch size ($m_k$) and $\alpha$ correlates well with the established theoretical results for \algname{clipped-SSTM}.

\subsubsection{Evolution of the Noise Distribution}\label{sec:extra_experiments_noise_distr_evolution}
In this section, we provide our empirical study of the noise distribution evolution along the trajectories of different optimizers. As one can see from the plots \revision{in Figures \ref{figure:bert_norm_diffs_evolution} and \ref{figure:resnet_18_norm_diffs_evolution}}, the noise distribution for {\tt ResNet-18} + {\tt ImageNet-100} task is always close to Gaussian distribution, whereas for  {\tt BERT} + {\tt CoLA} task it is significantly heavy-tailed.
\begin{figure}[h]
  \centering
  \includegraphics[scale=0.26]{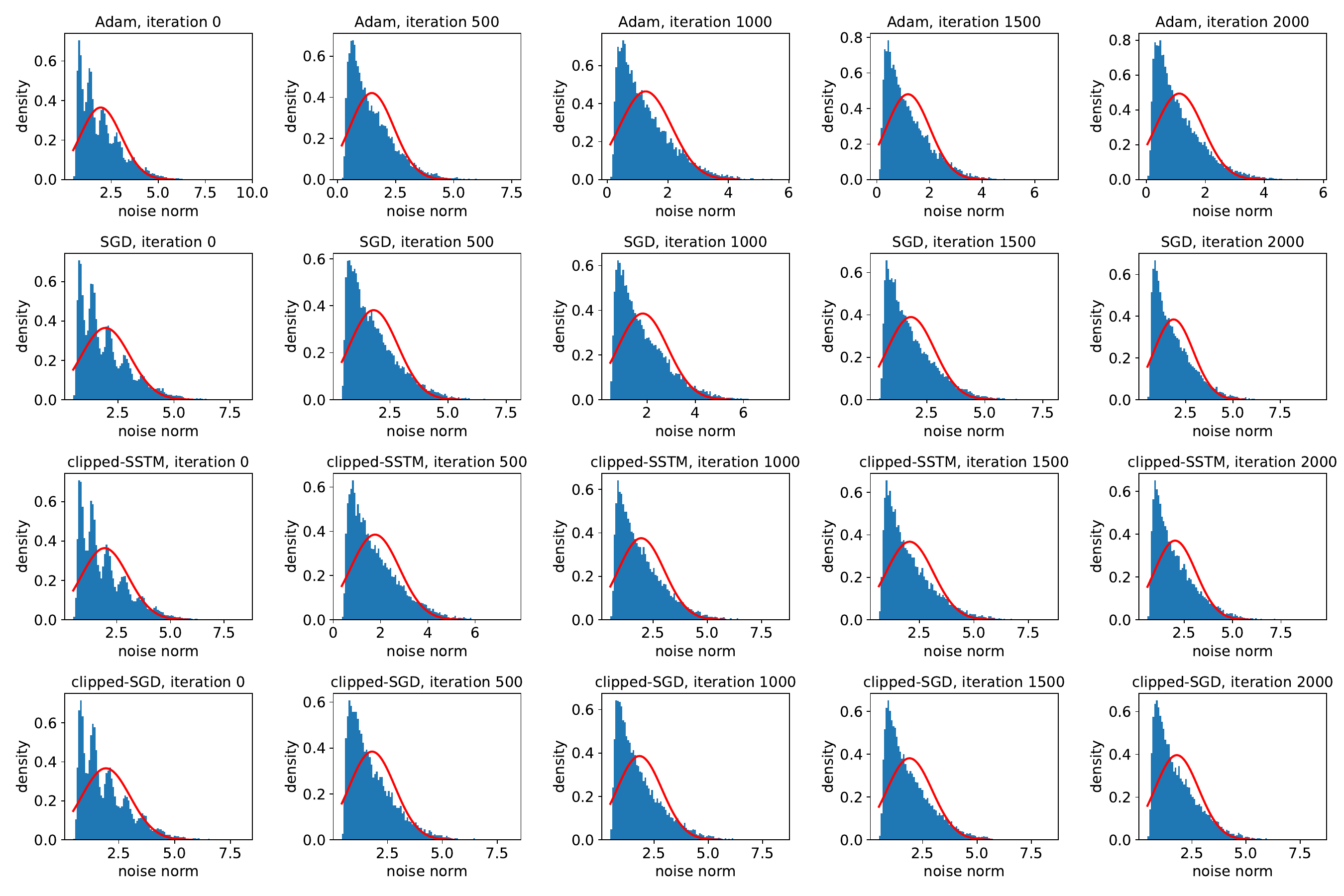}
  \caption{Evolution of the noise distribution for {\tt BERT} + {\tt CoLA} task.}
  \label{figure:bert_norm_diffs_evolution}
\end{figure}

\begin{figure}[h]
  \centering
  \includegraphics[scale=0.26]{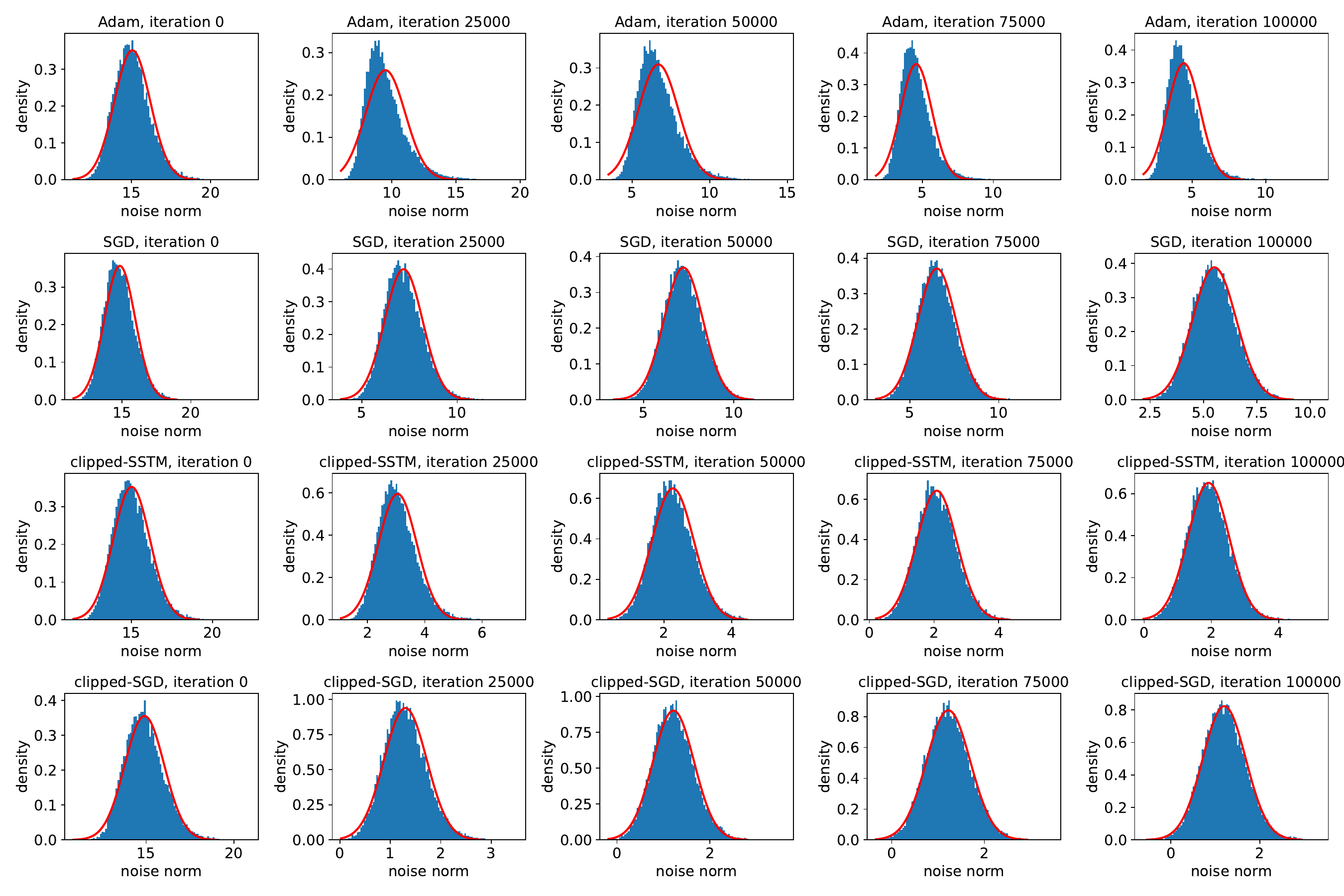}
  \caption{Evolution of the noise distribution for {\tt ResNet-18} + {\tt ImageNet-100} task.}
  \label{figure:resnet_18_norm_diffs_evolution}
\end{figure}

\subsubsection{Additional Results on NLP Tasks}\label{appendix:additional_NLP}

In addition to the previous experiments, we tried several different tasks from GLUE benchmark \citep{wang2018glue}: binary classification ({\tt SST-2}), paraphrase ({\tt MRPC}) and text similarity ({\tt STS-B}). Noise distributions for these tasks are represented in Figure~\ref{figure:norm_diffs_distribution_GLUE}. As for {\tt BERT} + {\tt CoLA}, the noise in the stochastic gradients is heavy-tailed.

In these additional tests, we also consider $\algname{\text{\sf clipped-SSTM}^{\ast}}$ -- a modification of \algname{clipped-SSTM} with linearly increasing batch size. This method provides faster convergence at early steps, as well as better results overall in comparison to \algname{clipped-SSTM}. Parameters for all methods were tuned and are reported below.

\begin{figure}[h]
  \centering
  \includegraphics[scale=0.25]{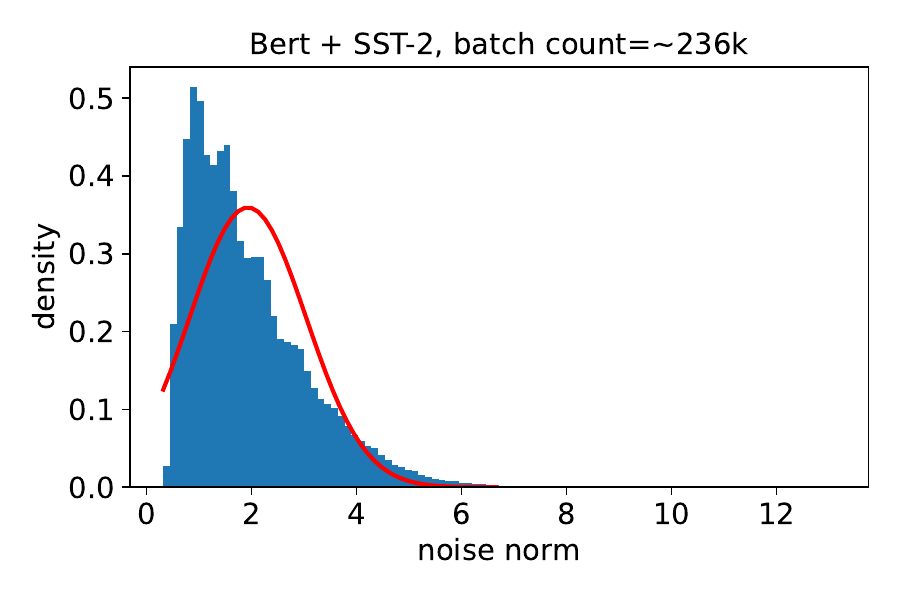}
  \includegraphics[scale=0.25]{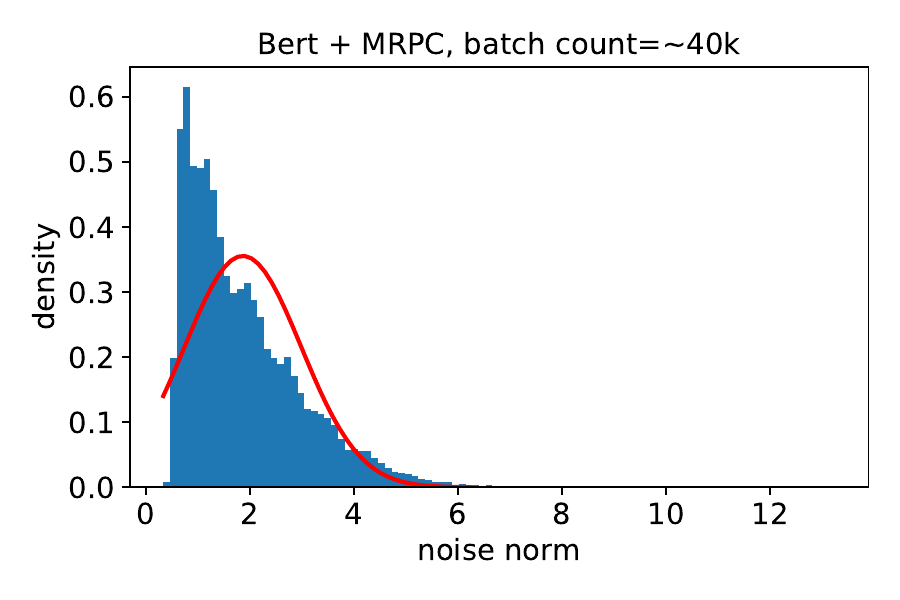}
  \includegraphics[scale=0.25]{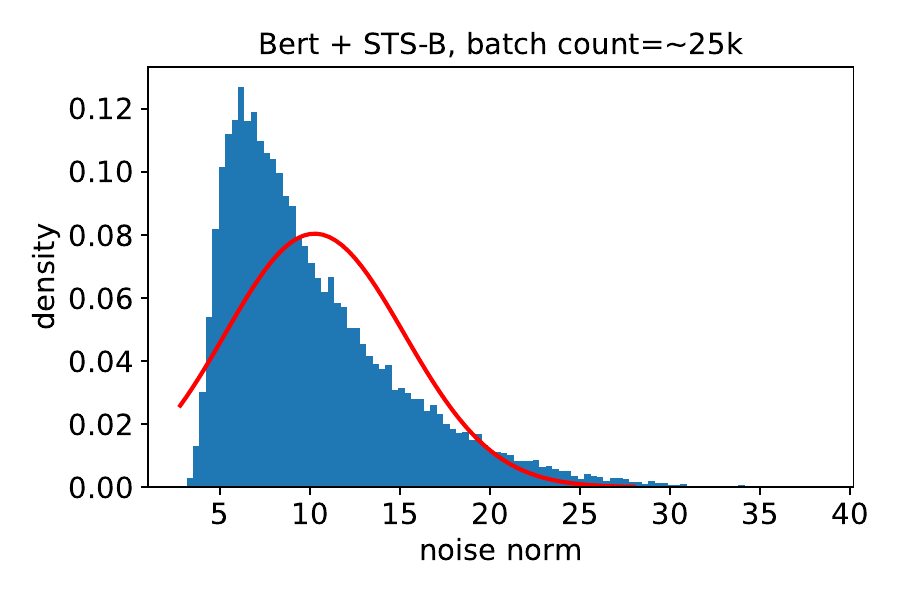}
  \caption{Noise distribution of the stochastic gradients for GLUE benchmark and {\tt BERT}  before the training. Red lines: probability density functions with means and variances empirically estimated by the samples. Batch count is the total number of samples used to build a histogram.}
  \label{figure:norm_diffs_distribution_GLUE}
\end{figure}

\paragraph{Text Classification, SST-2.} For {\tt BERT} + {\tt SST-2} the parameters of the methods were chosen as follows:
\begin{itemize}
    \item \algname{Adam}: $lr=5e-4$ and a batch size of $32$
    \item \algname{SGD}: $lr=1e-3$, $momentum=0.9$ and a batch size of $32$
    \item \algname{clipped-SGD}: $lr=2e-3$, $momentum=0.9$, coordinate-wise clipping with clipping parameter $B=0.1$ and a batch size of $32$
    \item \algname{clipped-SSTM}: $\nu = 1$, stepsize parameter $\alpha = 1e-3$, norm clipping with clipping parameter $B=1$ and a batch size of $8\times 32$
    \item $\algname{\text{\sf clipped-SSTM}^{\ast}}$: $\nu = 1$, stepsize parameter $\alpha = 1e-3$, norm clipping with clipping parameter $B=1$ and a batch size of $k\times 32$ where $k = 1 + \lfloor\frac{batch~count}{400}\rfloor$
\end{itemize}

The results are represented in Figure~\ref{figure:losses_graph_bert_sst_2}. 

\begin{figure}[h]
  \centering
  \includegraphics[scale=0.3]{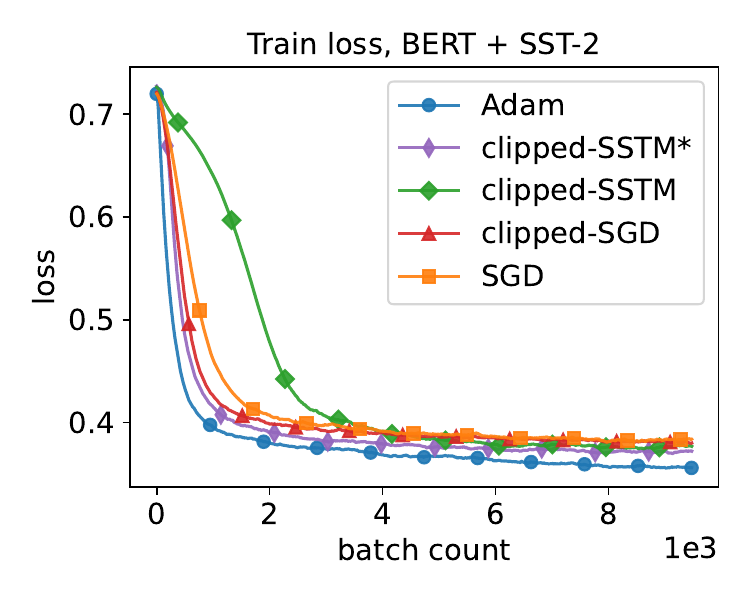}
  \includegraphics[scale=0.3]{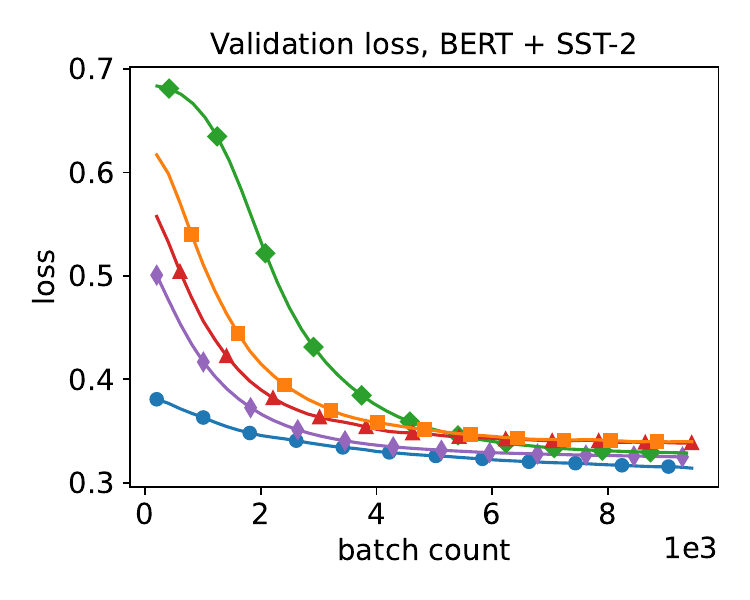}
  \includegraphics[scale=0.3]{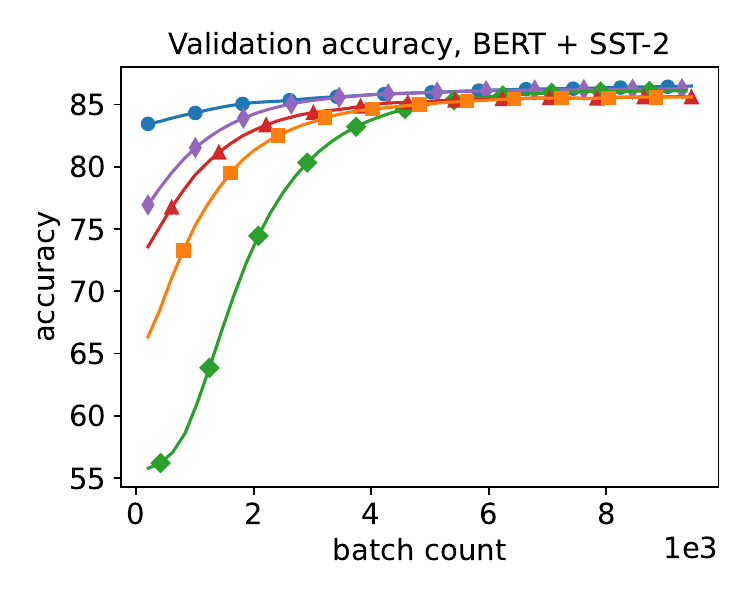}
  \caption{Train and validation loss + accuracy for different optimizers for {\tt BERT} + {\tt SST-2} task.}
  \label{figure:losses_graph_bert_sst_2}
\end{figure}

\paragraph{Paraphrase, MRPC.} For {\tt BERT} + {\tt MRPC} the parameters of the methods were chosen as follows:
\begin{itemize}
    \item \algname{Adam}: $lr=5e-4$ and a batch size of $32$
    \item \algname{SGD}: $lr=3e-4$, $momentum=0.99$ and a batch size of $32$
    \item \algname{clipped-SGD}: $lr=1e-3$, $momentum=0.95$, coordinate-wise clipping with clipping parameter $B=0.1$ and a batch size of $32$
    \item \algname{clipped-SSTM}: $\nu = 1$, stepsize parameter $\alpha = 3e-3$, norm clipping with clipping parameter $B=1$ and a batch size of $4\times 32$
    \item $\algname{\text{\sf clipped-SSTM}^{\ast}}$: $\nu = 1$, stepsize parameter $\alpha = 1e-2$, norm clipping with clipping parameter $B=1$ and a batch size of $k\times 32$ where $k = 1 + \lfloor\frac{batch~count}{100}\rfloor$
\end{itemize}

The results are represented in Figure~\ref{figure:losses_graph_bert_mrpc}. 

\begin{figure}[h]
  \centering
  \includegraphics[scale=0.3]{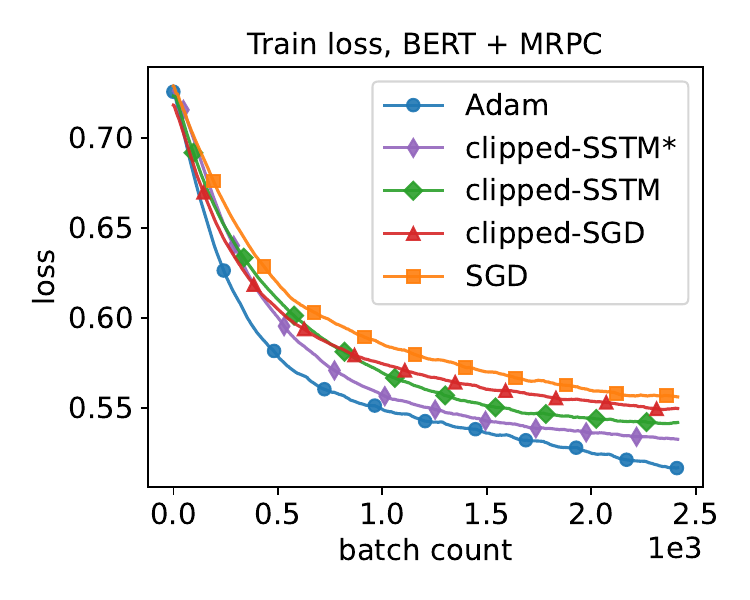}
  \includegraphics[scale=0.3]{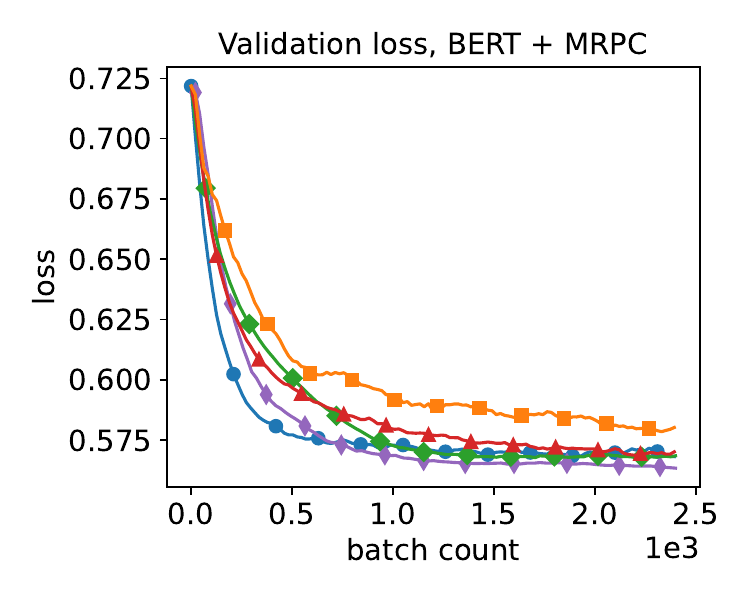}
  \includegraphics[scale=0.3]{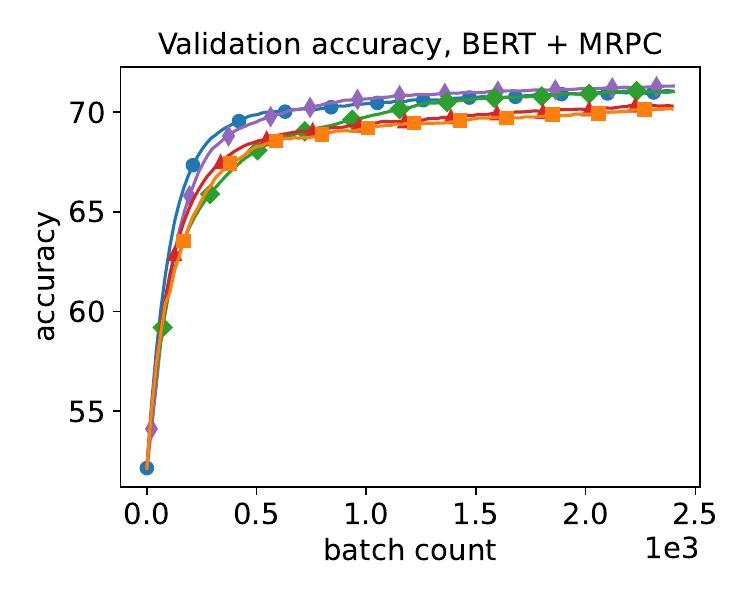}
  \caption{Train and validation loss + accuracy for different optimizers for {\tt BERT} + {\tt MRPC} task.}
  \label{figure:losses_graph_bert_mrpc}
\end{figure}

\paragraph{Text similarity, STS-B.} For {\tt BERT} + {\tt STS-B} the parameters of the methods were chosen as follows:
\begin{itemize}
    \item \algname{Adam}: $lr=1e-3$ and a batch size of $32$
    \item \algname{SGD}: $lr=1e-4$, $momentum=0.99$ and a batch size of $32$
    \item \algname{clipped-SGD}: $lr=1e-3$, $momentum=0.995$, coordinate-wise clipping with clipping parameter $B=0.1$ and a batch size of $32$
    \item \algname{clipped-SSTM}: $\nu = 1$, stepsize parameter $\alpha = 1e-2$, norm clipping with clipping parameter $B=1$ and a batch size of $8\times 32$
    \item $\algname{\text{\sf clipped-SSTM}^{\ast}}$: $\nu = 1$, stepsize parameter $\alpha = 3e-3$, norm clipping with clipping parameter $B=1$ and a batch size of $k\times 32$ where $k = 1 + \lfloor\frac{batch~count}{200}\rfloor$
\end{itemize}

The results are represented in Figure~\ref{figure:losses_graph_bert_sts_b}. 

\begin{figure}[h]
  \centering
  \includegraphics[scale=0.3]{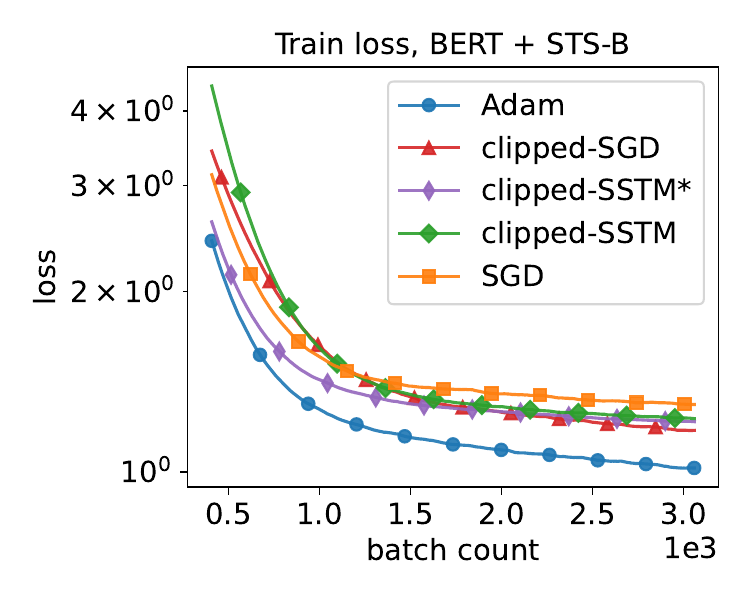}
  \includegraphics[scale=0.3]{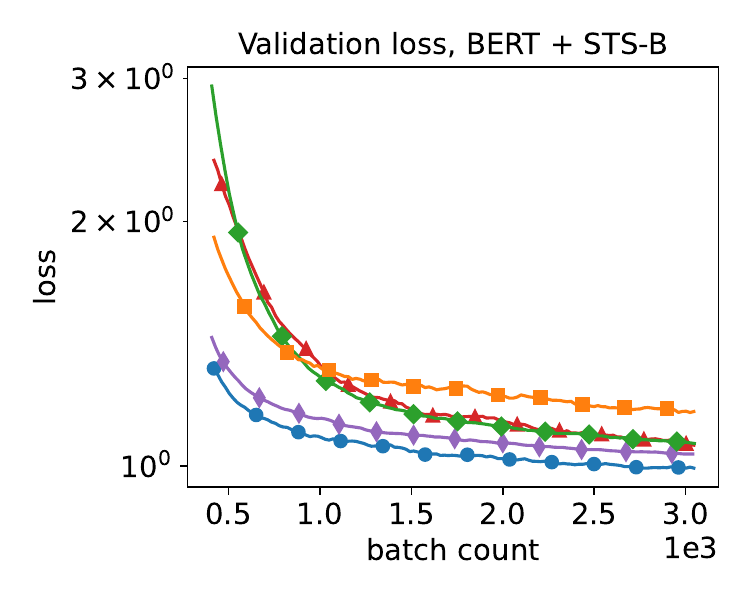}
  \includegraphics[scale=0.3]{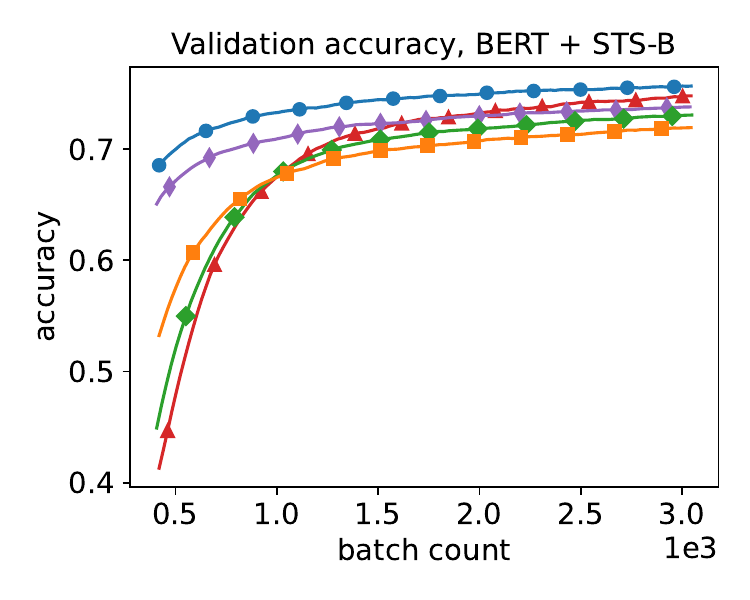}
  \caption{Train and validation loss + accuracy for different optimizers for {\tt BERT} + {\tt STS-B} task.}
  \label{figure:losses_graph_bert_sts_b}
\end{figure}

As for {\tt BERT} + {\tt CoLA}, we see that methods with clipping are superior to \algname{SGD}. This is expected in view of the histograms reported in Figure~\ref{figure:norm_diffs_distribution_GLUE} and previous empirical studies. We also point out that \algname{clipped-SSTM}/\algname{clipped-SSTM}$^{\ast}$/\algname{clipped-SGD} achieve either comparable or even better validation accuracy than \algname{Adam}.









\bibliography{refs1}

\begin{thebibliography}{}

\bibitem[Bennett, 1962]{bennett1962probability}
Bennett, G. (1962).
\newblock Probability inequalities for the sum of independent random variables.
\newblock {\em Journal of the American Statistical Association}, 57(297):33--45.

\bibitem[Chaux et~al., 2007]{Chaux_2007}
Chaux, C., Combettes, P.~L., Pesquet, J.-C., and Wajs, V.~R. (2007).
\newblock A variational formulation for frame-based inverse problems.
\newblock {\em Inverse Problems}, 23(4):1495--1518.

\bibitem[Davis et~al., 2021]{davis2021low}
Davis, D., Drusvyatskiy, D., Xiao, L., and Zhang, J. (2021).
\newblock From low probability to high confidence in stochastic convex optimization.
\newblock {\em Journal of Machine Learning Research}, 22(49):1--38.

\bibitem[Devlin et~al., 2019]{devlin2018bert}
Devlin, J., Chang, M.-W., Lee, K., and Toutanova, K. (2019).
\newblock {BERT}: Pre-training of deep bidirectional transformers for language understanding.
\newblock In {\em North American Chapter of the Association for Computational Linguistics}.

\bibitem[Devolder, 2013]{devolder2013exactness}
Devolder, O. (2013).
\newblock {\em Exactness, inexactness and stochasticity in first-order methods for large-scale convex optimization}.
\newblock PhD thesis, UCLouvain.

\bibitem[Devolder et~al., 2014]{devolder2014first}
Devolder, O., Glineur, F., and Nesterov, Y. (2014).
\newblock First-order methods of smooth convex optimization with inexact oracle.
\newblock {\em Mathematical Programming}, 146(1):37--75.

\bibitem[Dvurechensky and Gasnikov, 2016]{dvurechensky2016stochastic}
Dvurechensky, P. and Gasnikov, A. (2016).
\newblock Stochastic intermediate gradient method for convex problems with stochastic inexact oracle.
\newblock {\em Journal of Optimization Theory and Applications}, 171(1):121--145.

\bibitem[Dzhaparidze and Van~Zanten, 2001]{dzhaparidze2001bernstein}
Dzhaparidze, K. and Van~Zanten, J. (2001).
\newblock On {B}ernstein-type inequalities for martingales.
\newblock {\em Stochastic processes and their applications}, 93(1):109--117.

\bibitem[Freedman et~al., 1975]{freedman1975tail}
Freedman, D.~A. et~al. (1975).
\newblock On tail probabilities for martingales.
\newblock {\em the Annals of Probability}, 3(1):100--118.

\bibitem[Gasnikov and Nesterov, 2018]{gasnikov2016universal}
Gasnikov, A.~V. and Nesterov, Y.~E. (2018).
\newblock Universal method for stochastic composite optimization problems.
\newblock {\em Computational Mathematics and Mathematical Physics}, 58:48--64.

\bibitem[Gasnikov et~al., 2015]{gasnikov2015efficiency}
Gasnikov, A.~V., Nesterov, Y.~E., and Spokoiny, V.~G. (2015).
\newblock On the efficiency of a randomized mirror descent algorithm in online optimization problems.
\newblock {\em Computational Mathematics and Mathematical Physics}, 55(4):580--596.

\bibitem[Gehring et~al., 2017]{gehring2017convolutional}
Gehring, J., Auli, M., Grangier, D., Yarats, D., and Dauphin, Y.~N. (2017).
\newblock Convolutional sequence to sequence learning.
\newblock In {\em Proceedings of the 34th International Conference on Machine Learning-Volume 70}, pages 1243--1252. JMLR. org.

\bibitem[Ghadimi and Lan, 2012]{ghadimi2012optimal}
Ghadimi, S. and Lan, G. (2012).
\newblock Optimal stochastic approximation algorithms for strongly convex stochastic composite optimization i: A generic algorithmic framework.
\newblock {\em SIAM Journal on Optimization}, 22(4):1469--1492.

\bibitem[Ghadimi and Lan, 2013]{ghadimi2013stochastic}
Ghadimi, S. and Lan, G. (2013).
\newblock Stochastic first-and zeroth-order methods for nonconvex stochastic programming.
\newblock {\em SIAM Journal on Optimization}, 23(4):2341--2368.

\bibitem[Goodfellow et~al., 2016]{Goodfellow-et-al-2016}
Goodfellow, I., Bengio, Y., and Courville, A. (2016).
\newblock {\em Deep Learning}.
\newblock MIT Press.
\newblock \url{http://www.deeplearningbook.org}.

\bibitem[Gorbunov et~al., 2020]{gorbunov2020clipped_sstm}
Gorbunov, E., Danilova, M., and Gasnikov, A. (2020).
\newblock Stochastic optimization with heavy-tailed noise via accelerated gradient clipping.
\newblock In Larochelle, H., Ranzato, M., Hadsell, R., Balcan, M.~F., and Lin, H., editors, {\em Advances in Neural Information Processing Systems}, volume~33, pages 15042--15053. Curran Associates, Inc.

\bibitem[Gower et~al., 2019]{gower2019sgd}
Gower, R.~M., Loizou, N., Qian, X., Sailanbayev, A., Shulgin, E., and Richt{\'a}rik, P. (2019).
\newblock {SGD}: General analysis and improved rates.
\newblock In Chaudhuri, K. and Salakhutdinov, R., editors, {\em Proceedings of the 36th International Conference on Machine Learning}, volume~97 of {\em Proceedings of Machine Learning Research}, pages 5200--5209. PMLR.

\bibitem[Guigues et~al., 2017]{guigues2017non}
Guigues, V., Juditsky, A., and Nemirovski, A. (2017).
\newblock Non-asymptotic confidence bounds for the optimal value of a stochastic program.
\newblock {\em Optimization Methods and Software}, 32(5):1033--1058.

\bibitem[Guzm{\'a}n and Nemirovski, 2015]{guzman2015lower}
Guzm{\'a}n, C. and Nemirovski, A. (2015).
\newblock On lower complexity bounds for large-scale smooth convex optimization.
\newblock {\em Journal of Complexity}, 31(1):1--14.

\bibitem[Hazan et~al., 2015]{hazan2015beyond}
Hazan, E., Levy, K., and Shalev-Shwartz, S. (2015).
\newblock Beyond convexity: Stochastic quasi-convex optimization.
\newblock In Cortes, C., Lawrence, N., Lee, D., Sugiyama, M., and Garnett, R., editors, {\em Advances in Neural Information Processing Systems}, volume~28. Curran Associates, Inc.

\bibitem[Juditsky and Nemirovski, 2011]{juditsky2011first}
Juditsky, A. and Nemirovski, A. (2011).
\newblock First order methods for nonsmooth convex large-scale optimization, i: general purpose methods.
\newblock {\em Optimization for Machine Learning}, pages 121--148.

\bibitem[Kingma and Ba, 2015]{kingma2014adam}
Kingma, D.~P. and Ba, J. (2015).
\newblock Adam: A method for stochastic optimization.
\newblock In {\em International Conference on Learning Representations}.

\bibitem[Lan, 2012]{lan2012optimal}
Lan, G. (2012).
\newblock An optimal method for stochastic composite optimization.
\newblock {\em Mathematical Programming}, 133(1-2):365--397.

\bibitem[Mai and Johansson, 2021]{mai2021stability}
Mai, V.~V. and Johansson, M. (2021).
\newblock Stability and {C}onvergence of {S}tochastic {G}radient {C}lipping: {B}eyond {L}ipschitz {C}ontinuity and {S}moothness.
\newblock In Meila, M. and Zhang, T., editors, {\em Proceedings of the 38th International Conference on Machine Learning}, volume 139 of {\em Proceedings of Machine Learning Research}, pages 7325--7335. PMLR.

\bibitem[Menon et~al., 2020]{menon2020can}
Menon, A.~K., Rawat, A.~S., Reddi, S.~J., and Kumar, S. (2020).
\newblock Can gradient clipping mitigate label noise?
\newblock In {\em International Conference on Learning Representations}.

\bibitem[Moulines and Bach, 2011]{moulines2011non}
Moulines, E. and Bach, F. (2011).
\newblock Non-asymptotic analysis of stochastic approximation algorithms for machine learning.
\newblock In Shawe-Taylor, J., Zemel, R., Bartlett, P., Pereira, F., and Weinberger, K., editors, {\em Advances in Neural Information Processing Systems}, volume~24, pages 451--459. Curran Associates, Inc.

\bibitem[Nazin et~al., 2019]{nazin2019algorithms}
Nazin, A.~V., Nemirovsky, A.~S., Tsybakov, A.~B., and Juditsky, A.~B. (2019).
\newblock Algorithms of robust stochastic optimization based on mirror descent method.
\newblock {\em Automation and Remote Control}, 80(9):1607--1627.

\bibitem[Nemirovski et~al., 2009]{nemirovski2009robust}
Nemirovski, A., Juditsky, A., Lan, G., and Shapiro, A. (2009).
\newblock Robust stochastic approximation approach to stochastic programming.
\newblock {\em SIAM Journal on optimization}, 19(4):1574--1609.

\bibitem[Nemirovski and Yudin, 1983]{nemirovsky1983problem}
Nemirovski, A.~S. and Yudin, D.~B. (1983).
\newblock {\em Problem Complexity and Method Efficiency in Optimization}.
\newblock A Wiley-Interscience publication. Wiley.

\bibitem[Nesterov, 2015]{nesterov2015universal}
Nesterov, Y. (2015).
\newblock Universal gradient methods for convex optimization problems.
\newblock {\em Mathematical Programming}, 152(1-2):381--404.

\bibitem[Nesterov, 1983]{nesterov1983method}
Nesterov, Y.~E. (1983).
\newblock A method for solving the convex programming problem with convergence rate {O}$(1/k^2)$.
\newblock In {\em Dokl. akad. nauk Sssr}, volume 269, pages 543--547.

\bibitem[Pascanu et~al., 2013]{pascanu2013difficulty}
Pascanu, R., Mikolov, T., and Bengio, Y. (2013).
\newblock On the difficulty of training recurrent neural networks.
\newblock In Dasgupta, S. and McAllester, D., editors, {\em Proceedings of the 30th International Conference on Machine Learning}, volume~28 of {\em Proceedings of Machine Learning Research}, pages 1310--1318, Atlanta, Georgia, USA. PMLR.

\bibitem[Paszke et~al., 2019]{NEURIPS2019_9015}
Paszke, A., Gross, S., Massa, F., Lerer, A., Bradbury, J., Chanan, G., Killeen, T., Lin, Z., Gimelshein, N., Antiga, L., Desmaison, A., Kopf, A., Yang, E., DeVito, Z., Raison, M., Tejani, A., Chilamkurthy, S., Steiner, B., Fang, L., Bai, J., and Chintala, S. (2019).
\newblock Py{T}orch: An imperative style, high-performance deep learning library.
\newblock In Wallach, H., Larochelle, H., Beygelzimer, A., d\textquotesingle Alch\'{e}-Buc, F., Fox, E., and Garnett, R., editors, {\em Advances in Neural Information Processing Systems 32}, pages 8024--8035. Curran Associates, Inc.

\bibitem[Robbins and Monro, 1951]{robbins1951stochastic}
Robbins, H. and Monro, S. (1951).
\newblock A stochastic approximation method.
\newblock {\em The annals of mathematical statistics}, pages 400--407.

\bibitem[Russakovsky et~al., 2015]{ILSVRC15}
Russakovsky, O., Deng, J., Su, H., Krause, J., Satheesh, S., Ma, S., Huang, Z., Karpathy, A., Khosla, A., Bernstein, M., Berg, A.~C., and Fei-Fei, L. (2015).
\newblock {ImageNet Large Scale Visual Recognition Challenge}.
\newblock {\em International Journal of Computer Vision (IJCV)}, 115(3):211--252.

\bibitem[Sadiev et~al., 2023]{sadiev2023high}
Sadiev, A., Danilova, M., Gorbunov, E., Horv{\'a}th, S., Gidel, G., Dvurechensky, P., Gasnikov, A., and Richt{\'a}rik, P. (2023).
\newblock High-probability bounds for stochastic optimization and variational inequalities: the case of unbounded variance.
\newblock In {\em International Conference on Machine Learning}, pages 29563--29648. PMLR.

\bibitem[{\c{S}}im{\c{s}}ekli et~al., 2019]{csimcsekli2019heavy}
{\c{S}}im{\c{s}}ekli, U., G{\"u}rb{\"u}zbalaban, M., Nguyen, T.~H., Richard, G., and Sagun, L. (2019).
\newblock On the heavy-tailed theory of stochastic gradient descent for deep neural networks.
\newblock {\em arXiv preprint arXiv:1912.00018}.

\bibitem[Simsekli et~al., 2019]{simsekli2019tail}
Simsekli, U., Sagun, L., and Gurbuzbalaban, M. (2019).
\newblock A tail-index analysis of stochastic gradient noise in deep neural networks.
\newblock In Chaudhuri, K. and Salakhutdinov, R., editors, {\em Proceedings of the 36th International Conference on Machine Learning}, volume~97 of {\em Proceedings of Machine Learning Research}, pages 5827--5837. PMLR.

\bibitem[Spokoiny, 2012]{spokoiny2012parametric}
Spokoiny, V. (2012).
\newblock Parametric estimation. finite sample theory.
\newblock {\em The Annals of Statistics}, 40(6):2877--2909.

\bibitem[Wang et~al., 2018]{wang2018glue}
Wang, A., Singh, A., Michael, J., Hill, F., Levy, O., and Bowman, S.~R. (2018).
\newblock {GLUE}: A multi-task benchmark and analysis platform for natural language understanding.
\newblock {\em arXiv preprint arXiv:1804.07461}.

\bibitem[Warstadt et~al., 2019]{warstadt2018neural}
Warstadt, A., Singh, A., and Bowman, S.~R. (2019).
\newblock Neural network acceptability judgments.
\newblock {\em Transactions of the Association for Computational Linguistics}, 7:625--641.

\bibitem[Wolf et~al., 2020]{wolf-etal-2020-transformers}
Wolf, T., Debut, L., Sanh, V., Chaumond, J., Delangue, C., Moi, A., Cistac, P., Rault, T., Louf, R., Funtowicz, M., Davison, J., Shleifer, S., von Platen, P., Ma, C., Jernite, Y., Plu, J., Xu, C., Scao, T.~L., Gugger, S., Drame, M., Lhoest, Q., and Rush, A.~M. (2020).
\newblock Transformers: State-of-the-art natural language processing.
\newblock In {\em Proceedings of the 2020 Conference on Empirical Methods in Natural Language Processing: System Demonstrations}, pages 38--45, Online. Association for Computational Linguistics.

\bibitem[Zhang et~al., 2020a]{zhang2020gradient}
Zhang, J., He, T., Sra, S., and Jadbabaie, A. (2020a).
\newblock Why gradient clipping accelerates training: A theoretical justification for adaptivity.
\newblock In {\em International Conference on Learning Representations}.

\bibitem[Zhang et~al., 2020b]{zhang2020why}
Zhang, J., Karimireddy, S.~P., Veit, A., Kim, S., Reddi, S., Kumar, S., and Sra, S. (2020b).
\newblock Why are adaptive methods good for attention models?
\newblock In Larochelle, H., Ranzato, M., Hadsell, R., Balcan, M.~F., and Lin, H., editors, {\em Advances in Neural Information Processing Systems}, volume~33, pages 15383--15393. Curran Associates, Inc.

\end{thebibliography}

\end{document}